\documentclass[12pt,reqno,hidelinks]{report} 

\usepackage{amsmath,amsfonts}
\usepackage{lipsum}

\usepackage[singlespacing,noindent,parskip,showlof,showlot,showframe]{tamuconfig}
\tamuthesissetfloatsep{2\baselineskip plus 0.5\baselineskip minus 0.5\baselineskip}

\tamuthesissetBIBTEXfilename{References.bib}

\usepackage{amsmath,amssymb,amsthm,mathdots}
\usepackage{array}
\usepackage{color,colordvi,graphicx}
\usepackage{xypic}
\usepackage{bm} 
\usepackage{multirow}
\usepackage{url}
\newtheorem{theorem}{Theorem}[section]

\newtheorem{conjecture}[theorem]{Conjecture}
\newtheorem{proposition}[theorem]{Proposition}
\newtheorem{corollary}[theorem]{Corollary}
\newtheorem{example}[theorem]{Example}
\newtheorem{remark}[theorem]{Remark}
\newtheorem{definition}[theorem]{Definition}
\DeclareMathOperator{\codim}{codim}
\DeclareMathOperator{\rank}{rank}

\DeclareMathOperator{\Gr}{Gr}
\DeclareMathOperator{\Fl}{Fl}
\DeclareMathOperator{\Jac}{Jac}
\DeclareMathOperator{\Pl}{Pl}

\DeclareMathOperator{\rowspace}{rowspace}
\DeclareMathOperator{\colspace}{colspace}
\DeclareMathOperator{\s}{S}
\DeclareMathOperator{\hats}{\widehat{\s}}
\DeclareMathOperator{\tbox}{\rule{-2pt}{8pt}}
\DeclareMathOperator{\Span}{span}
\DeclareMathOperator{\St}{St}
\DeclareMathOperator{\Syl}{Syl}
\DeclareMathOperator{\SYT}{SYT}
\DeclareMathOperator{\Sturm}{Sturm}
\DeclareMathOperator{\Wr}{Wr}
\DeclareMathOperator{\GL}{GL}
\DeclareMathOperator{\HG}{HG}
\DeclareMathOperator{\LG}{LG}
\DeclareMathOperator{\RGr}{\mathbb{R}Gr}
\DeclareMathOperator{\RLG}{\mathbb{R}LG}
\DeclareMathOperator{\SL}{SL}
\DeclareMathOperator{\Sp}{Sp}
\DeclareMathOperator{\Id}{Id}
\DeclareMathOperator{\IN}{in}
\DeclareMathOperator{\Mat}{Mat}
\DeclareMathOperator{\vari}{var}
\DeclareMathOperator{\sign}{sign}
\DeclareMathOperator{\closure}{closure}
\newcommand{\ba}{{\bf a}}

\newcommand{\be}{{\bf e}}
\newcommand{\bff}{{\bf f}}
\newcommand{\bg}{{\bf g}}
\newcommand{\bh}{{\bf h}}
\newcommand{\bk}{{\bf k}}
\newcommand{\balpha}{{\bm \alpha}}
\newcommand{\bbeta}{{\bm \beta}}
\newcommand{\C}{{\mathbb C}}
\newcommand{\R}{{\mathbb R}}
\newcommand{\p}{{\mathbb P}}

\newcommand{\Fdot}{F_\bullet}
\newcommand{\Gdot}{G_\bullet}
\newcommand{\Edot}{E_\bullet}
\newcommand{\Eop}{E_\bullet'}
\newcommand{\barF}{\overline{\Fdot}}
\newcommand{\Fdual}{F_\bullet^\perp}
\newcommand{\I}{\raisebox{-1pt}{\includegraphics{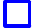}}}
\newcommand{\Is}{\includegraphics{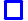}}
\newcommand{\It}{\includegraphics{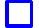}}
\newcommand{\Iss}{\includegraphics{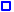}}

\newcommand{\IIs}{\includegraphics{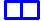}}
\newcommand{\III}{\raisebox{-1pt}{\includegraphics{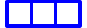}}}

\newcommand{\IIIt}{\includegraphics{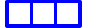}}

\newcommand{\IIIIIt}{\includegraphics{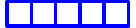}}

\newcommand{\IIi}{\raisebox{-4.5pt}{\includegraphics{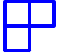}}}
\newcommand{\IIis}{\includegraphics{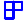}}
\newcommand{\IIit}{\includegraphics{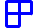}}
\newcommand{\IIii}{\raisebox{-4.5pt}{\includegraphics{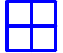}}}
\newcommand{\IIiib}{\includegraphics{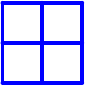}}
\newcommand{\IIiis}{\includegraphics{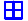}}
\newcommand{\IIiit}{\includegraphics{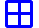}}
\newcommand{\IIIi}{\raisebox{-4.5pt}{\includegraphics{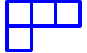}}}
\newcommand{\IIIis}{\includegraphics{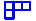}}
\newcommand{\IIIit}{\includegraphics{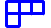}}

\newcommand{\IIIiiit}{\includegraphics{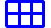}}
\newcommand{\IIIIiiiit}{\includegraphics{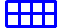}}

\newcommand{\lenIIi}{\raisebox{-4.5pt}{\includegraphics{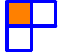}}}
\newcommand{\lenIIii}{\raisebox{-4.5pt}{\includegraphics{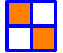}}}

\newcommand{\IiI}{{\raisebox{-8pt}{\includegraphics{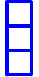}}}}

\newcommand{\IiIt}{\includegraphics{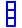}}
\newcommand{\IIiI}{{\raisebox{-8pt}{\includegraphics{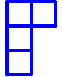}}}}
\newcommand{\IIiIs}{\includegraphics{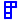}}

\newcommand{\IIiiIIt}{\includegraphics{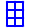}}
\newcommand{\IIIiI}{{\raisebox{-8pt}{\includegraphics{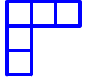}}}}
\newcommand{\IIIiIs}{\includegraphics{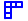}}
\newcommand{\IIIiIt}{\includegraphics{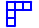}}

\newcommand{\IIIiiIs}{\includegraphics{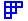}}
\newcommand{\IIIiiIt}{\includegraphics{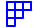}}
\newcommand{\IIIiiiII}{{\raisebox{-8pt}{\includegraphics{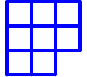}}}}

\newcommand{\IIIiiiIIt}{\includegraphics{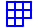}}
\newcommand{\IIIiiiIII}{{\raisebox{-8pt}{\includegraphics{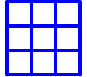}}}}

\newcommand{\IIIiiiIIIt}{\includegraphics{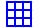}}
\newcommand{\IIIIiiiI}{{\raisebox{-8pt}{\includegraphics{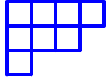}}}}

\newcommand{\lenIIIIiiiI}{{\raisebox{-8pt}{\includegraphics{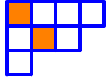}}}}

\newcommand{\omas}{{\includegraphics{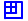}}}
\newcommand{\omat}{\includegraphics{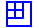}}

\newcommand{\IIIiiiIminusIi}{{\raisebox{-8pt}{\includegraphics{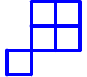}}}}

\newcommand{\IIIiiiIminusIit}{\includegraphics{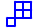}}
\newcommand{\IIIIiiiIminusIIi}{{\raisebox{-8pt}{\includegraphics{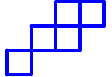}}}}
\newcommand{\IIIIiiiIminusIIib}{{\includegraphics{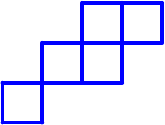}}}
\newcommand{\IIIiiiIIminusIIiit}{\includegraphics{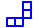}}
\newcommand{\IIIiiiIIminusIIiib}{\includegraphics{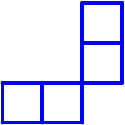}}
\newcommand{\IIIiiiIIminusIb}{\includegraphics{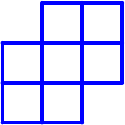}}

\newcommand{\IIIiGrIIIvii}{{\raisebox{-8pt}{\includegraphics{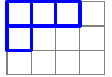}}}}
\newcommand{\IIIIiiiIGrIIIvii}{{\raisebox{-8pt}{\includegraphics{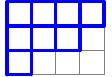}}}}
\newcommand{\IIIiIp}{\includegraphics{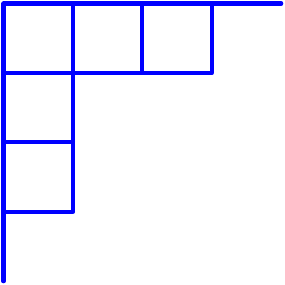}}

\newcommand{\ik}{{(\includegraphics{figures/1t.pdf}\rule{0pt}{8pt}^{k^2})}}

\definecolor{lightgray}{gray}{0.8}
\definecolor{4LinesPurple}{rgb}{0.49,0,0.49} 
\definecolor{4LinesGreen}{rgb}{0.20,0.68,0} 
\definecolor{4LinesRed}{rgb}{0.95,0,0} 
\definecolor{4LinesBlue}{rgb}{0,0,0.9} 
\definecolor{4LinesDarkBlue}{rgb}{0,0,0.55} 
\newcommand{\diagdots}[3][-25]{%
  \rotatebox{#1}{\makebox[0pt]{\makebox[#2]{\xleaders\hbox{$\cdot$\hskip#3}\hfill\kern0pt}}}%
}

\begin{document}

  %
  \tamuthesissetmanuscripttitle{Reality and Computation in Schubert Calculus}
  \tamuthesissetpapertype{Dissertation}
  \tamuthesissetfullname{Nickolas Jason Hein}
  \tamuthesissetdegree{Doctor of Philosophy}
  \tamuthesissetdepthead{Emil Straube}
  \tamuthesissetgradmonth{August}
  \tamuthesissetgradyear{2013}
  \tamuthesissetdepartment{Mathematics}
  \tamuthesissetchairone{Frank Sottile}       
  \tamuthesissetchairtwo{}    
  \tamuthesissetmemberone{Joseph Landsberg}   
  \tamuthesissetmembertwo{J.\ Maurice Rojas}   
  \tamuthesissetmemberthree{Nancy Amato} 
  
  \tamuthesistitlepage 

  \begin{tamuthesisabstractchapter}

The Mukhin-Tarasov-Varchenko Theorem (previously the Shapiro Conjecture) asserts that a Schubert problem has all solutions distinct and real if the Schubert varieties involved osculate a rational normal curve at real points.
When conjectured, it sparked interest in real osculating Schubert calculus, and computations played a large role in developing the surrounding theory.
Our purpose is to uncover generalizations of the Mukhin-Tarasov-Varchenko Theorem, proving them when possible.
We also improve the state of the art of computationally solving Schubert problems, allowing us to more effectively study ill-understood phenomena in Schubert calculus.

We use supercomputers to methodically solve real osculating instances of Schubert problems.
By studying over 300 million instances of over 700 Schubert problems, we amass data significant enough to reveal possible generalizations of the Mukhin-Tarasov-Varchenko Theorem and compelling enough to support our conjectures.
Combining algebraic geometry and combinatorics, we prove some of these conjectures.
To improve the efficiency of solving Schubert problems, we reformulate an instance of a Schubert problem as the solution set to a square system of equations in a higher-dimensional space.

During our investigation, we found the number of real solutions to an instance of a symmetrically defined Schubert problem is congruent modulo four to the number of complex solutions.
We proved this congruence, giving a new invariant in enumerative real algebraic geometry.
We also discovered a family of Schubert problems whose number of real solutions to a real osculating instance has a lower bound depending only on the number of defining flags with real osculation points.

We conclude that our method of computational investigation is effective for uncovering phenomena in enumerative real algebraic geometry.
Furthermore, we point out that our square formulation for instances of Schubert problems may facilitate future experimentation by allowing one to solve instances using certifiable numerical methods in lieu of more computationally complex symbolic methods.
Additionally, the methods we use for proving the congruence modulo four and for producing an unexpected square system of equations are both quite general, and they may be of use in future projects.

\end{tamuthesisabstractchapter}

  \begin{tamuthesisdedicationchapter}

Mom, your confidence in me believes I can do anything.
It gives me hope and pushes me on.

Dad, thanks for teaching me to work hard and keep having fun.
I hope I never stop learning and growing.

Jeanette, where would I be without you?
I don't know, but it would be less fun.

\end{tamuthesisdedicationchapter}

  \begin{tamuthesisacknowledgementschapter}

I would like to thank my advisor Frank Sottile for giving generously of his time and energy to help me achieve my goals.

\end{tamuthesisacknowledgementschapter}


\tamuthesistoc

\tamuthesislof

\tamuthesislot

  \tamuthesisintroductionchapter
\label{chapIntro}

The fundamental theorem of algebra states that the number of complex roots of a univariate polynomial is the degree of the polynomial, counting multiplicities.
B\'ezout's Theorem gives the number of points of intersection of two projective plane curves, thereby generalizing the fundamental theorem of algebra.
Enumerative algebraic geometry studies the further generalization of counting solutions to polynomial systems with geometric meaning.
The most elegant results in enumerative algebraic geometry, such as the fundamental theorem of algebra and B\'ezout's Theorem, depend on working over an algebraically closed field, so they are of limited use in applications which require information about real solutions.

One real analogue to the fundamental theorem of algebra is Descar\-tes's rule of signs, which bounds the number of positive roots of a univariate polynomial with coefficients in $\R$.
With very little work, one may use the rule of signs to  find an upper bound $R$ on the number $r$ of real roots of a real polynomial.
Since nonreal roots of real polynomials come in pairs, we have $r\equiv R$ mod 2.

The inelegance of counting the real roots of a polynomial compared to counting complex roots is typical of statements in enumerative real algebraic geometry.
This makes real statements harder to detect and less attractive to prove.
As a result, the enumerative theory of real algebraic geometry is not as well formed as its complex companion.

With the use of computers we may now engage in a study of enumerative real algebraic geometry that is long overdue.
One example of success in this field is the Shapiro Conjecture, made by the brothers Boris and Michael Shapiro in 1993.
The conjecture was refined and supported by computational data collected by Sottile \cite{Sottile2000}.
Eremenko and Gabrielov proved partial results \cite{EG2002b}, and the full conjecture for the real Schubert calculus of Grassmannians was proved by Mukhin, Tarasov, and Varchenko \cite{MTV2009a,MTV2009b}.
The Mukhin-Tarasov-Varchenko Theorem states that a Schubert problem has all solutions real and distinct if the Schubert varieties involved are defined with respect to distinct real flags osculating a single real parametrized rational normal curve.
Thus the number of real solutions to the corresponding system of real polynomials depends only on the Schubert problem, and this number may be obtained using the Littlewood-Richardson rule.

Computational projects \cite{secant,monotone,RSSS2006} have suggested generalizations to the Mukhin-Tarasov-Varchenko Theorem, some of which now have been proven \cite{EG2002b,mod4,MTV2009a,MTV2009b}.
In this thesis, we describe a computational project extending the study by Eremenko and Gabrielov \cite{EG2002a} of lower bounds on the number of real solutions to certain Schubert problems.
Eremenko and Gabrielov computed a topological degree which gives a lower bound for the number of real points in an intersection of osculating Schubert varieties when the intersection is stable under complex conjugation and at most two of the Schubert varieties are not hypersurfaces.
We solved over 339 million instances of 756 Schubert problems, including those involving non-hypersurface Schubert varieties, to investigate these Eremenko-Gabrielov type lower bounds.
For Schubert problems involving at most two non-hypersurface Schubert varieties, we tested the sharpness of known bounds.

During our computational investigation, we observed that the number of real solutions to a real Schubert problem with certain symmetries is congruent modulo four to the number of complex solutions.
This is stronger than the usual congruence modulo two arising from nonreal solutions coming in complex conjugate pairs.
While this congruence was unexpected, the underlying reason is simple enough: there are two involutions acting on the solutions to a real symmetric Schubert problem, complex conjugation and a Lagrangian involution.
When subtle nondegeneracy conditions are satisfied, the involutions are independent.
This gives Theorem \ref{thm:mod4}, the first of our two main results.

Computational complexity can be a serious obstacle when studying systems of equations, and even more so when we investigate systems by the hundreds of millions.
As with previous large-scale computations, we were limited by the severe complexity of symbolic computation \cite{HL2011,MM1982}.
Numerical homotopy methods provide an alternative for solving problems which are infeasible by Gr\"obner basis methods in characteristic zero, but the approximate solutions produced do not come with a certificate verifying the solutions.
There is software which may be used to certify approximate solutions \cite{alphaCertified}, but the algorithms used require a square polynomial system.
That is, the number of equations must equal the number of variables, and there must be finitely many solutions.
Like many other problems in algebraic geometry, Schubert problems are traditionally not defined by a square system.

We give a primal-dual formulation of a Schubert problem which presents it as a square system in local coordinates.
This reformulation is presented in our second main result, Theorem \ref{thm:SchubProbCompInt}, and it allows one to certify approximate solutions obtained numerically.

In Chapter \ref{chapSchubert}, we give definitions needed for Schubert calculus and a brief history of conjectures and theorems in real Schubert calculus.
In Chapter \ref{chapLower}, we describe the computational project extending the study of Eremenko-Gabrielov type lower bounds.
In Chapter ~\ref{chapMod4}, we prove a new theorem in enumerative real algebraic geometry, the congruence modulo four discovered in the computational project.
In Chapter ~\ref{chapSquare}, we give a method for formulating a general Schubert problem as a square system.

\chapter{\uppercase {Real Schubert Calculus}}
  \label{chapSchubert}
  Schubert calculus is the study of linear spaces having special position with respect to fixed but general linear spaces.
We provide background and describe a series of surprising conjectures and theorems about real solutions to problems in Schubert calculus.
Large computations played a big role in uncovering conjectures and motivating theorems in this area.

\section{Preliminaries}\label{sec:prelim}

We assume knowledge of \cite{CLO2007} as a basic reference.
We are interested in enumerative problems in real Schubert calculus which are solved by counting real points in a variety.
We provide background which is useful for counting these points when the associated ideal is generated by a set of real multivariate polynomials.
The first step is to write the generators in a standard form using a Gr\"obner basis.

Let $x=(x_1,\dotsc,x_q)$ denote variables and $a=(a_1,\dotsc,a_q)$ denote an \emph{exponent vector} so that $x^a := {x_1}^{a_1}\cdots {x_q}^{a_q}$ is a monomial.
A \emph{term order} on $\C[x]$ is a well-ordering of monomials of $\C[x]$, for which $1$ is minimal, and which respects multiplication.
The \emph{lexicographic term order $\prec$} on $\C[x]$ is the term order, such that $x^a\prec x^b$ if the last nonzero entry of $b-a$ is positive.
We give some comparisons for $q=3$,
\[1 \prec x_1 \prec x_1^9 \prec x_1^{17} \prec x_2 \prec x_2^2x_1^6 \prec x_2^3 \prec x_3 \prec x_3x_1.\]

Let $f(x)\in \C[x]$ be a multivariate polynomial.
The \emph{initial term $\IN_{\prec} f$} of $f(x)$ is the maximal term of $f$ with respect to $\prec$.
For example, if $f(x) = 4 - 2x_1 + x_2 + 3x_3$ and $g(x) = x_2x_1^{9} + 3 x_2^2x_1 + 2 x_2^{5} - 5 x_3x_1$, then $\IN_\prec f=3x_3$ and $\IN_\prec g= -5 x_3x_1$.

This definition extends to ideals.
The \emph{initial ideal $\IN_\prec I$} of an ideal $I\subset\C[x]$ is the ideal generated by the initial terms of elements of $I$,
\[\IN_\prec I :=( \IN_\prec f |\ f\in I )\,.\]

A \emph{Gr\"obner basis $B=(g_1,\dotsc,g_N)$} of an ideal $I$ is a generating set for $I$ with the property that $(\IN_\prec g_1,\dotsc,\IN_\prec g_N)$ is a generating set of $\IN_\prec I$.
There are efficient algorithms implemented in the computer algebra system \url{Singular}, which calculate Gr\"obner bases \cite{Singular}.

Suppose $f=(f_1,\dotsc,f_p)$ is a system of multivariate polynomials in the variables $x$ with finitely many common zeros and let $I$ be the ideal generated by $f$.
An \emph{eliminant of $I$} is a univariate polynomial $g(x_1)\subset I$ of minimal degree.
This implies that the roots of $g_1$ are the $x_1$-values of the points in the variety $\mathcal{V}(I)\subset \C^q$.

If $\mathcal{V}(I)$ is zero-dimensional, then the \emph{degree} $d$ of $I$ is the number of points in $\mathcal{V}(I)$, counting multiplicity.
(Here, the multiplicity of a point in a zero-dimensional scheme is the usual Hilbert-Samuel multiplicity.)
In this case, if the points of $\mathcal{V}(I)$ have distinct $x_1$-values, then an eliminant $g$ of $I$ has degree $d$.
An eliminant may be calculated using a Gr\"obner basis with respect to the lexicographic term order $\prec$.
Indeed, one of the generators will be an eliminant.

A \emph{reduced lexicographic Gr\"obner basis of $I$} is a Gr\"obner basis $B=(b_1,\dotsc,b_N)$ with respect to the lexicographic term order $\prec$ such that $\IN_\prec b_i$ does not divide any term of $b_j$ for distinct $i,j\leq N$.
Given a Gr\"obner basis with respect to $\prec$, one may obtain a reduced lexicographic Gr\"obner basis by iteratively reducing the generators using the Euclidean algorithm.
\begin{proposition}[The Shape Lemma \cite{shape}]\label{prop:shapeLemma}
 Let $I\subset\C[x]$ be an ideal such that $\mathcal{V}(I)$ is zero-dimensional.
 Suppose $f$ is a generating set for $I$ and $B$ is a reduced lexicographic Gr\"obner basis of $I$ obtained by applying Buchberger's algorithm to $f$.
 If the eliminant $g\in B$ has degree $d = \deg(I)$ and $g$ is square-free, then
 \[B = (g(x_1),x_2-g_2(x_1),\dotsc,x_q-g_q(x_1))\,,\]
 with $\deg(g_{j})<d$ for $j>1$.
\end{proposition}
\begin{proof}
 The polynomial $g$ generates $I\cap \C[x_1]$, so $1,x,\dotsc,x^{d-1}$ are standard monomials.
 The number of standard monomials of $I$ with respect to $\prec$ is $\deg(I)$, so there are no other standard monomials.
 The generators in $B$ have initial terms $x_1^d,x_2,\dotsc,x_q$, so reducing the generators gives a Gr\"obner basis of the stated form.
\end{proof}

For each root $r_1$ of the eliminant $g\in B$, there is a unique point $r:=(r_1,g_2(r_1),\dotsc,g_q(r_1))\in\mathcal{V}(I)$.
Therefore, the projection $\pi : \C^q\rightarrow \C$ given by $(c_1,\dotsc,c_q)\mapsto c_1$ sends the points of $\mathcal{V}(I)$ to the roots of $g$.
When $I$ has a generating set of real polynomials, Buchberger's algorithm produces a Gr\"obner basis of real polynomials.
Using this Gr\"obner basis, one obtains a reduced Gr\"obner basis $B$ whose generators are real polynomials.
Thus the Shape Lemma asserts that $r_1$ is real in and only if $r$ is real.
This allows us to use an eliminant to calculate the number of real points in a zero-dimensional variety.
The following corollary to the Shape Lemma has been useful in computational experiments in Schubert calculus \cite{secant,monotone,RSSS2006}.

\begin{corollary}\label{prop:shape}
 Suppose the hypotheses of Proposition \ref{prop:shapeLemma} are satisfied.
 If $f$ is real then the number of real points in $\mathcal{V}(I)$ is equal to the number of real roots of $g$.
\end{corollary}

If the restriction of the projection $\pi$ to $\mathcal{V}(I)$ is not injective, then one may permute the variables $x$ or use more sophisticated methods to rectify this \cite{Rou99}.
To use Corollary \ref{prop:shape}, we require an algorithm for counting the real roots of $g$, which is based on sequences of polynomials.
Let $y$ denote the minimal variable in $x$ after reordering.
\begin{definition}
 If $f_1,f_2 \in \C[y]$ are univariate polynomials, the \emph{Sylvester sequence} $\Syl(f_1,f_2)$ is the subsequence of nonzero entries of the recursively defined sequence,
 \[f_{j}:=-\mbox{remainder}(f_{j-2},f_{j-1})\quad \mbox{for}\quad j>2\,.\]
 Here, the remainder is calculated via the Euclidean algorithm, so $\Syl(f_1,f_2)$ is finite with final entry $f_s = \pm \gcd(f_1,f_2)$.
\end{definition}
\begin{definition}
 If $f \in \C[y]$ is a univariate polynomial, the \emph{Sturm sequence of $f\in \C[y]$} is $\Sturm(f):=\Syl(f,f')$.
\end{definition}
 We point out that while none of the entries of $\Sturm(f)$ are identically zero, its evaluation $\Sturm(f(a))$ at a point $a\in\C$ may contain zeros.
 We are concerned with the number of sign changes that occur between the nonzero entries.
\begin{definition}
 Suppose $f \in \C[y]$ is a univariate polynomial, $a\in \C$ is a complex number, $\Sigma^a$ is the subsequence of nonzero entries of\/ $\Sturm(f(a))$, and $l$ is the number of entries in $\Sigma^a$.
 For $j\in [l-1]$, the product $\Sigma^a_j \Sigma^a_{j+1}$ is negative if and only if the $j$th and $(j+1)$th entries of $\Sigma^a$ have different signs.
 The \emph{variation of $f$ at $a$} is obtained by counting sign alternations,
 \[\vari(f,a):=\#\{j\in [l-1]\ |\ \Sigma^a_j \Sigma^a_{j+1}<0\}\,.\]
\end{definition}
\begin{theorem}[Sturm's Theorem]
 Let $f\in \R[y]$ be a univariate polynomial and $a,b\in\mathbb R$ with $a<b$ and $f(a),f(b)\neq 0$.
 Then the number of distinct zeros of $f$ in the interval $(a,b)$ is the difference $\vari(f,a)-\vari(f,b)$.
\end{theorem}
The proof is standard.
One treatment may be found in \cite[p.\ 57]{BPR}.
The bitsize of coefficients in a Sturm sequence may grow quickly.
Implementations may control this growth by using a normalized Sturm-Habicht sequence.
Each entry of a Sturm-Habicht sequence is a positive multiple of the corresponding entry of a Sturm sequence, so $\vari(f,a)$ may be calculated via the normalized sequence.
The library \url{rootsur.lib} written by Enrique A.\ Tobis for \url{Singular} implements algorithms from \cite{BPR} to compute a Sturm-Habicht sequence of a univariate polynomial to count its distinct real roots.

\section{The Grassmannian}

We fix positive integers $k<n$ and a complex linear space $V$ of dimension $n$.
The choice of standard basis $\be$ identifies $V$ with $\mathbb C^n$, giving it a real structure.
Complex conjugation $v\mapsto\overline{v}$ is an involution on $V$.
\begin{definition}
 The \emph{Grassmannian $\Gr(k,V)$ of $k$-planes in $V$} is the set of $k$-dimensional linear subspaces of $V$,
 \[\Gr(k,V) := \{ H \subset V\ |\ \dim(H)=k\}\,.\]
\end{definition}
The automorphism $v\mapsto \overline{v}$ preserves the dimension of subspaces, so $H\in \Gr(k,V)$ implies $\overline{H}\in \Gr(k,V)$.

Let $\Mat_{k \times n}$ denote the set of $k\times n$ matrices with complex entries.
The determinant of an $i\times i$ submatrix of $M\in \Mat_{k \times n}$ is called an \emph{$i\times i$ minor} of $M$.
The determinant of a maximal square submatrix of $M$ is called a \emph{maximal minor} of $M$.
\begin{definition}
 The \emph{Stiefel manifold $\St(k,n)$} is the set of full-rank $k\times n$ matrices,
 \[\St(k,n) := \{M\in \Mat_{k \times n}\ |\ \rank(M)=k\}\,.\]
\end{definition}
Since $\rank(M)<k$ is a closed condition (given by the vanishing of minors), $\St(k,n)$ is a dense open subset of a vector space and thus a smooth manifold.

The Stiefel manifold parametrizes the Grassmannian by associating $P\in \St(k,n)$ to its row space $H\in \Gr(k,V)$.
There is a left action of $\GL(k,\C)$ on $\St(k,n)$ given by multiplication.
Since the set of all points in $\St(k,n)$ with row space $H$ is the $\GL(k,\C)$ orbit of $P$, $\St(k,n)$ is a $\GL(k,\C)$ fiber bundle over $\Gr(k,V)$.
Complex conjugation extends to matrices, and $\rowspace(P) = H$ implies $\rowspace(\overline{P})=\overline{H}$.
\begin{definition}
 A complex projective algebraic variety $X$ is called a \emph{real variety} if\/ $\overline{X} = X$.
\end{definition}
Note that a nonempty real variety need not contain any closed points with residue field $\R$.
For example, the curve defined by $x^2+y^2+z^2=0$ in $\p^2$ is real, but contains no closed points with residue field $\R$.
\begin{definition}
 Let $\wedge$ denote the usual exterior product in $V$, and $\bigwedge^k V$ the $k$th exterior power of\/ $V$.
 The product $v_1 \wedge \cdots \wedge v_k\in \bigwedge^k V$ is alternating, since transposing $v_i$ and $v_{i+1}$ is equivalent to multiplication by $-1$.
 If $H$ is a $k$-plane then $\bigwedge^k H$ is a line through the origin in $\bigwedge^k V$.
 Thus $\bigwedge^k H$ is a point in projective space, and we have a well-defined map
\begin{align*}
  \Phi \colon \Gr(k,V) &\longrightarrow \p(\mbox{$\bigwedge^k V$})\,,\\
  H &\longmapsto \mbox{$\bigwedge^k H$}
\end{align*}
 called the \emph{Pl\"ucker map}.
 We call the space $\p(\bigwedge^k V)$ \emph{Pl\"ucker space}.
\end{definition}
\begin{definition}
 Let $\tbinom{[n]}{k}$ denote the set of sublists of $[n]:=\{1,2,\dotsc,n\}$ with $k$ entries.
\end{definition}
\begin{definition}
The basis $\be$ of $V$ induces a basis of $\bigwedge^k V$ whose generators are
\[e_\alpha := e_{\alpha_1}\wedge \cdots \wedge e_{\alpha_k}\]
for $\alpha\in\tbinom{[n]}{k}$.
The coordinates $[\,p_\alpha\ |\ \alpha\in\tbinom{[n]}{k}]$ dual to this basis are called \emph{Pl\"ucker coordinates}.
For $H\in\Gr(k,V)$ we write
\[\Phi(H) = \sum_{\alpha\in\binom{[n]}{k}} p_\alpha(H)e_\alpha\,,\]
with $p_\alpha(H)\in\C$.
We call $p_\alpha(H)$ the \emph{$\alpha$th Pl\"ucker coordinate of $H$}.
\end{definition}
The Pl\"ucker coordinates are closely related to the parametrization of $\Gr(k,V)$ given by $\St(k,n)$.
Suppose $Q\in\St(k,n)$ has row space $H\in\Gr(k,V)$ and $\alpha\in\tbinom{[n]}{k}$.
Let $Q_\alpha$ denote the maximal minor of $Q$ involving columns $\alpha_1,\dotsc,\alpha_k$.
Then $[\,Q_\alpha\ |\ \alpha\in\tbinom{[n]}{k}]$ and $[\,p_\alpha(H)\ |\ \alpha\in\tbinom{[n]}{k}]$ are the same point in Pl\"ucker space.
The proofs of the two following propositions are based partially on \cite{KL1972}.
\begin{proposition}\label{prop:PluckerCoords}
 The Pl\"ucker map is injective.
\end{proposition}
\begin{proof}
 Let $Q\in \St(k,n)$ be a matrix with row space $H\in\Gr(k,V)$.
 The $k$-plane $H$ has some nonzero Pl\"ucker coordinate, so without loss of generality $p_{[k]}(H) \neq 0$.
 Thus $Q$ may be written in block form $[A|B]$ where $A$ is a $k\times k$ invertible matrix.
 Multiplying, we have $A^{-1}Q = [\Id_k|A^{-1}B]$, which gives another matrix with row space $H$.

 For $i\in[k]$ and $j\in\{k+1,\dotsc,n\}$ we define $\alpha(i,j):=(1,\dotsc,\widehat{i},\dotsc,k,j)$.
 We may express the $(i,j)$th entry of $A^{-1}Q$ as a maximal minor
 \[(A^{-1}Q)_{ij} = (-1)^{k-i} (A^{-1}Q)_{\alpha(i,j)} = p_{\alpha(i,j)}(H)\,.\]
 Since the maximal minors of $A^{-1}Q$ are the Pl\"ucker coordinates $[\,p_\alpha(H)\ |\ \alpha\in\tbinom{[n]}{k}]$, $H$ may be recovered from the Pl\"ucker coordinates $[\,p_\alpha(H)\ |\ \alpha\in\tbinom{[n]}{k}]$.
 Therefore, the Pl\"ucker map is injective.
\end{proof}
In the course of the proof, we used an affine cover of Pl\"ucker space.
To formalize this, let
\begin{equation}\label{def:Ucover}
 \mathcal U := \{U_\alpha\ |\ \alpha\in\tbinom{[n]}{k}\}
\end{equation}
be the cover of $\p(\bigwedge^k V)$ where $U_\alpha$ is the open set of $\p(\bigwedge^k V)$ given by the open condition $p_\alpha \neq 0$.
If $\alpha=[k]$ then the set $\s$ of $k\times n$ matrices of the form $[\Id_k|B]$ \emph{parametrize} $\Phi(\Gr(k,V))\cap U_{[k]}$, i.e., the map $\rowspace : \s \rightarrow U_{[k]}$ gives injective coordinates for $\Phi(\Gr(k,V))\cap U_{[k]}$ which are linear in the parameters of $\s$.
By permuting the columns of matrices in $\s$, we may similarly parametrize $\Phi(\Gr(k,V))\cap U_\alpha$ for $\alpha\in\tbinom{[n]}{k}$.
\begin{proposition}
 The image of the Pl\"ucker map is a projective variety.
\end{proposition}
\begin{proof}
 Since $\mathcal U$ is an affine cover of Pl\"ucker space, it suffices to show that the dense open set $\Phi(\Gr(k,V))\cap U_\alpha$ is an affine variety for each $\alpha\in\tbinom{[n]}{k}$.
 We show this for $\alpha=[k]$, and the other cases follow by symmetry.
 Let
 \[G_\alpha := \Phi^{-1}(\Phi(\Gr(k,V))\cap U_\alpha)\,.\]
 In the proof of Proposition \ref{prop:PluckerCoords}, we show that points in $G_\alpha$ are linear spaces of the form $\rowspace[\Id_k|B]\subset V$ such that $B\in \Mat_{k\times (n-k)}$.
 This identification defines a bijective map $\Psi: \Mat_{k \times n-k} \rightarrow G_\alpha$.
 The composition $\Phi\circ\Psi$ is injective, by Proposition \ref{prop:PluckerCoords}.
 Since this composition is given by minors, it is a regular map.
 We observed that the entries of $B$ are Pl\"ucker coordinates, so they span an affine space in Pl\"ucker space.
 Let $W$ denote the complementary affine space, and $\Omega:U_\alpha\rightarrow W$ the projection.
 Then $G_\alpha$ is the graph of the regular map $\Omega\circ\Phi\circ\Psi$.
 It follows that $G_\alpha$ is defined by polynomials in $U_\alpha$, so it is an affine variety, and the image of the Pl\"ucker map is a projective variety.
\end{proof}
\begin{corollary}\label{cor:GrassDim}
 The Grassmannian $\Gr(k,V)$ is a projective variety of dimension $k(n-k)$.
\end{corollary}
\begin{proof}
 The Pl\"ucker map is injective, so $\Gr(k,V)$ is a projective variety.
 The dense subset $G_\alpha\subset \Gr(k,V)$ is isomorphic to $\Mat_{k \times n-k}$, so $\dim(\Gr(k,V)) = k(n-k)$.
\end{proof}


\section{Schubert Varieties}

Schubert varieties are distinguished projective subvarieties of a Grassmannian.
They are defined with respect to a flag and a list $\alpha\in\tbinom{[n]}{k}$.

\begin{definition}
 A \emph{flag $\Fdot$} on $V$ is a list of nested linear subspaces of $V$,
 \[\Fdot : 0\subsetneq F_1 \subsetneq F_2 \subsetneq \cdots \subsetneq F_n = V\,,\]
 with $\dim(F_i)=i$ for $i\in[n]$.
 If $f_1,\dotsc,f_n\in V$ and $F_i = \langle f_1,\dotsc,f_i \rangle$ for $i\in[n]$, then we say the $n\times n$ matrix
 \begin{equation}\label{eqn:flagMatrix}
  \Fdot :=
   \left(
    \begin{matrix}
     f_1 \\
     \vdots \\
     f_n
    \end{matrix}
   \right)  
 \end{equation}
 is a \emph{basis} for the flag $\Fdot$.
 We sometimes refer to the list $(f_1,\dotsc,f_n)$ as a basis for $\Fdot$.
\end{definition}
The flag $\Edot$ with basis $(e_1,\dotsc,e_n)$ is called the \emph{standard flag}.
We note that the identity matrix $\Id_n$ is a basis for the standard flag.
\begin{definition}\label{Def:SchubertVariety}
 Let $\alpha\in\tbinom{[n]}{k}$ and $\Fdot$ a flag in $V$.
 The \emph{Schubert variety $X_\alpha \Fdot\subset \Gr(k,V)$} is the set of $k$-planes satisfying the incidence conditions,
 \[X_\alpha \Fdot := \{ H\in \Gr(k,V)\ |\ \dim(H\cap F_{\alpha_i})\geq i\ \mbox{for}\ i\in[k]\}\,.\]
 We call $\alpha$ a \emph{Schubert condition} on $\Gr(k,V)$ and $\Fdot$ a \emph{defining flag} for $X_\lambda \Fdot$.
\end{definition}
We will give determinantal equations in Proposition \ref{prop:detConds} which locally define $X_\alpha\Fdot$ as a subvariety of $\Gr(k,V)$.
If $\alpha_{i+1}=\alpha_i+1$ then the incidence condition on $X_\alpha\Fdot$ given by $\alpha_i$ is implied by the condition given by $\alpha_{i+1}$.
The implied conditions are called \emph{irrelevant}.
If $\alpha_k=n$, then the corresponding condition is also irrelevant since $H\cap F_n = H$ has dimension $k$ for $H\in\Gr(k,V)$.
The necessary defining conditions are called \emph{relevant}.
\begin{example}
 The $k$-planes $H\in X_{(2,3,5)} \Fdot \subset \Gr(3,\C^5)$ satisfy
 \begin{itemize}
  \item[(1)] $\dim(H\cap F_2)\geq 1$,
  \item[(2)] $\dim(H\cap F_3)\geq 2$, and
  \item[(3)] $\dim(H\cap F_5)\geq 3$.
 \end{itemize}
 Condition (3) is trivial since $\dim(H\cap F_5) = \dim(H) = 3 \geq 3$.
 Condition (1) is implied by (2) and is thus irrelevant.
 Condition (2) is the only relevant condition defining $X_{(2,3,5)} \Fdot$.
\end{example}

We give sets of matrices $\s(\alpha)$, $S_\alpha$, and $S_\alpha^\beta$, which locally parametrize the Grassmannian $\Gr(k,V)$, the Schubert variety $X_\alpha \Edot$, and the intersection $X_\alpha \Edot \cap X_\beta \Eop$ respectively, with respect to the standard basis $\be$ of $V$.
\begin{definition}\label{def:StiefelCoords}
 For $\alpha\in\tbinom{[n]}{k}$, the subset $\s(\alpha)\subset\St(k,n)$ of the Stiefel manifold is the set of matrices $M$ with $(i,\alpha_j)$th entry
 \[M_{i,\alpha_j} := \delta_{ij} \quad \mbox{for} \quad i,j\in[k] \,,\]
 and with other entries arbitrary.
 The parameters of $M$ give coordinates for the dense open set $\Gr(k,V) \cap U_\alpha \subset \Gr(k,V)$, and we call $\s(\alpha)$ \emph{Stiefel coordinates} on $\Gr(k,V)$.
\end{definition}
\begin{example}
 If $\alpha=[k]$, then matrices in $\s(\alpha)$ have block form $[\Id_k|B]$.
\end{example}
\begin{example}
 For $k=3$ and $n=7$, the matrices in $\s(2,5,7)$ have the form
 \[
  \left(
   \begin{matrix}
    * & 1 & * & * & 0 & * & 0 \\
    * & 0 & * & * & 1 & * & 0 \\
    * & 0 & * & * & 0 & * & 1 \\
   \end{matrix}
  \right)\,.
 \]
\end{example}
The 1 in position $(i,\alpha_i)$ of a matrix in $\s(\alpha)$ is called a \emph{pivot}.
The following is a consequence of the proof of Proposition \ref{prop:PluckerCoords}.
\begin{definition}\label{def:SAlpha}
 For $\alpha\in\tbinom{[n]}{k}$, the subset $\s_\alpha \subset \s(\alpha)$ is the subset of matrices such that each entry to the right of a pivot is $0$.
 We call $\s_\alpha$ the \emph{Stiefel coordinates} on $X_\alpha \Edot$.
\end{definition}
\begin{example}
 For $k=3$ and $n=7$, the matrices in $\s_{(2,5,7)}$ have the form
 \[
  \left(
   \begin{matrix}
    * & 1 & 0 & 0 & 0 & 0 & 0 \\
    * & 0 & * & * & 1 & 0 & 0 \\
    * & 0 & * & * & 0 & * & 1 \\
   \end{matrix}
  \right)\,.
 \]
\end{example}
\begin{definition}
 Let $\alpha\in\tbinom{[n]}{k}$.
 We call $X_\alpha \Edot^\circ := X_\alpha \Edot \cap U_\alpha$ the \emph{big cell} of $X_\alpha \Edot$.
\end{definition}
\begin{proposition}\label{prop:SchubBirat}
 The restriction to $\s_\alpha$ of the isomorphism $\phi:\s(\alpha)\rightarrow \Gr(k,V) \cap U_\alpha$ given by $H\mapsto [\,p_\alpha(H)\ |\ \alpha\in\tbinom{[n]}{k}]$ is an isomorphism $\phi_\alpha:\s_\alpha \rightarrow X_\alpha \Edot^\circ$.
\end{proposition}
\begin{proof}
 The incidence conditions on $H\in X_\alpha \Edot^\circ$ given in Definition \ref{Def:SchubertVariety} are equivalent to the conditions that $H$ contains independent vectors $h_i\in\langle e_1,\dotsc,e_{\alpha_i}\rangle$ for $i\in[k]$.
 If $H\in \Gr(k,V) \cap U_\alpha$, then $h_i$ may be chosen to be
 \[h_i = e_{\alpha^i} + \sum_{j=1}^{\alpha^i - 1} h_{ij}e_j\,.\]
 Therefore, $\s_\alpha$ is a subset of $\s(\alpha)$ which maps into $X_\alpha \Edot^\circ$ via $\phi$.
 The inverse $\phi_\alpha^{-1}$ exists on $X_\alpha \Edot^\circ$.
 The map $\phi_\alpha$ is given by minors, which are polynomials.
 The inverse $\phi^{-1}_\alpha$ is given by polynomials as the nonzero entries which are not identically $1$ are Pl\"ucker coordinates.
\end{proof}
Any flag $\Fdot$ has a basis $f:=(f_1,\dotsc,f_n)$.
Using $f$ as a basis for $V$ realizes $\Fdot$ as the standard flag.
We apply Proposition \ref{prop:SchubBirat}.
\begin{corollary}\label{cor:SVdimAux}
 Suppose $\alpha\in\tbinom{[n]}{k}$ and $\Fdot$ is a flag in $V$.
 Then a matrix $M_\alpha$ parametrizing $\s_\alpha$ gives local coordinates for $X_\alpha \Fdot$.
\end{corollary}
The $i$th row of $M_\alpha$ has $a_i-i$ indeterminates.
Corollary \ref{cor:SVdimAux} allows us to calculate the dimension of a Schubert variety.
\begin{corollary}
 The dimension of $X_\alpha \Fdot$ is
 \[\dim(X_\alpha \Fdot) = \sum_{i=1}^k \alpha_i - i.\]
\end{corollary}
Using Corollary \ref{cor:GrassDim}, we calculate the codimension of a Schubert variety.
\begin{definition}
 The codimension of $X_\alpha \Fdot$ in $\Gr(k,V)$ is
 \[|\alpha| := k(n-k) - \sum_{i=1}^k \alpha_i - i\,.\]
\end{definition}
With this definition, we see that each Grassmannian $\Gr(k,V)$ admits a unique Schubert condition $(k,k+2,\dotsc,n)$ which defines Schubert varieties of codimension one.
We write ${\It}$ to denote this condition, and we call $X_{\Is} \Fdot$ a \emph{hypersurface Schubert variety}.

There is an implicit way express the open dense subset $X_\alpha \Fdot \cap U_\beta \subset X_\alpha \Fdot$ using the Stiefel coordinates $\s(\beta)$ parametrizing $U_\beta$ with respect to $\be$.
Let the matrix $\Fdot$ denote a basis for the flag $\Fdot$ with respect $\be$.
Similarly, let $F_i$ denote the $i\times n$ submatrix of $\Fdot$ whose row space is the subspace $F_i$ in the flag $\Fdot$.
\begin{proposition}\label{prop:detConds}
 Let $\alpha,\beta\in\tbinom{[n]}{k}$ be Schubert conditions.
 Let $X_\alpha \Fdot\subset \Gr(k,V)$, and $M \in \s(\beta)$ be a matrix parametrizing $U_\beta\subset \Gr(k,n)$.
 Then the open dense subset $X_\alpha \Fdot \cap U_\beta \subset X_\alpha \Fdot$ is defined by the vanishing of the $r_i\times r_i$ minors of $\binom{M}{F_{\alpha_i}}$, where $r_i = k+\alpha_i-i+1$ for $i\in[k]$.
\end{proposition}
\begin{proof}
 The definition (\ref{Def:SchubertVariety}) is equivalent to the requirement that the rows of $M$ and rows of $F_{\alpha_i}$ span a space of dimension at most $r_i-1$.
 The implied rank conditions on $\binom{M}{F_{\alpha_i}}$ are given by the vanishing of $r_i\times r_i$ minors.
\end{proof}
\begin{example}\label{ex:7max}
 The Schubert variety $H\in X_{(2,3,5,6)} \Fdot\subset \Gr(4,6)$ has only one relevant condition, $\dim(H\cap F_3)\geq 2$, so its determinantal conditions from Proposition \ref{prop:detConds} consist of the seven maximal minors of $\binom{M}{F_{3}}$.
\end{example}
\begin{definition}
 Regarding $\Gr(k,V)$ as a variety in Pl\"ucker space via the Pl\"ucker embedding, the \emph{Pl\"ucker ideal} $\Pl_{k,n}$ is the ideal $\Pl_{k,n}:=\mathcal{I}(\Gr(k,V))$.
\end{definition}

The partially ordered set of Schubert conditions in $\tbinom{[n]}{k}$ given by
\[\alpha \leq \beta\quad\mbox{if}\quad \alpha_i\leq \beta_i\quad\mbox{for}\quad i\in[k]\]
is called the \emph{Bruhat order}.
This order gives us a way to determine the number of determinants needed to define a Schubert variety.
\begin{proposition}\label{prop:standardBruhat}
 The ideal of the Schubert variety $X_\alpha \Edot$ in Pl\"ucker space is
 \[\Pl_{k,n} + (p_\beta\ |\ \beta\not\leq\alpha)\,.\]
\end{proposition}
\begin{proof}
Suppose the matrix $M$ parametrizes $\s(\alpha)$, and consider the Stiefel coordinates $\s_\alpha\subset \s(\alpha)$ on $X_\alpha \Edot^\circ \subset \Gr(k,V) \cap U_\alpha$.
As observed in the proof of Proposition \ref{prop:PluckerCoords}, the parameters of $M$ which are identically zero on $\s_\alpha$ are the Pl\"ucker coordinates $p_\beta$ such that $\beta\not\leq\alpha$.
\end{proof}
This gives us the number of linearly independent generators of $\mathcal{I}(X_\alpha \Fdot)$ as a subvariety of $\Gr(k,V)$.
The right action of $g\in\GL(n,\C)$ on $V$ induces a dual left action on the Pl\"ucker coordinates of $\Gr(k,V)$.
The Grassmannian is invariant under the action of $\GL(n,\C)$, so the Pl\"ucker ideal is invariant under the dual action.
Thus for $g\in \GL(n,\C)$ the ideal $\mathcal{I}(X_\alpha \Fdot.g)$ is given by the sum of ideals 
\[\mathcal{I}(X_\alpha \Fdot.g) = \Pl_{k,n} + (g^{-1}.p_\beta\ |\ \beta\not\leq\alpha)\,.\]
\begin{corollary}\label{cor:numGens}
 Let $\Fdot$ be any flag in $V$.
 The ideal of the Schubert variety $X_\alpha \Fdot$ as a subvariety of $\Gr(k,V)$ is generated by
 \[\#\{p_\beta\ |\ \beta\not\leq\alpha\}\]
 linearly independent determinantal equations.
\end{corollary}
Using this, we see how far one may reduce the system of determinantal equations given by Proposition \ref{prop:detConds}.
For example, the seven maximal minors in Example \ref{ex:7max} may be reduced to three linearly independent minors.

The hypersurface $X_{\Is}\Fdot \subset \Gr(k,V)$ has one relevant condition given by $\det \binom{M}{F_{n-k}} = 0$.
Using Corollary \ref{cor:numGens}, we see the number of linearly independent determinants from Proposition \ref{prop:detConds} needed to define $X_\alpha\Fdot$ is greater than $|\alpha|$ when $|\alpha|>1$ and $\min\{k,n-k\}\geq 2$.

\section{Schubert Problems}

We have now seen two ways to locally express a Schubert variety $X_\alpha \Fdot$, one by choosing a basis $\bff$ of $V$ so that $\s_\alpha$ parametrizes a dense subset of $X_\alpha \Fdot$ and another by determinantal equations in parameters for some $U_\beta$ with respect to the standard basis $\be$.
Thus we may express the intersection points of $X_\alpha \Fdot \cap X_\beta \Gdot$ using either determinantal conditions defining $X_\alpha \Fdot$ and $X_\beta \Gdot$ in local Stiefel coordinates for $\Gr(k,V)$ or determinantal conditions defining $X_\beta \Gdot$ in local Stiefel coordinates for $X_\alpha \Fdot$.
We give a third formulation of $X_\alpha \Fdot \cap X_\beta \Gdot$ when $\Fdot$ and $\Gdot$ are in sufficiently general position.

\begin{definition}
 The flag $\Eop$ with basis $(e_n,\dotsc,e_1)$ is called the \emph{standard opposite flag}.
\end{definition}
We note that the $n\times n$ matrix with ones along the antidiagonal and zeros elsewhere is a basis for the standard opposite flag.
\begin{definition}\label{def:local2vars}
 For $\alpha,\beta\in\tbinom{[n]}{k}$, the subset $\s_\alpha^\beta\subset\Mat_{k \times n}$ consists of matrices $M$ whose entries satisfy
 \[M_{ij} = 1 \mbox{ if } j=\alpha_i\qquad\mbox{and}\qquad M_{ij} = 0 \mbox{ if } j>\alpha_i \mbox{ or } j<n+1-\beta_{k-i+1}\,,\]
and whose other entries are arbitrary.
\end{definition}
\begin{example}
 Let $\alpha=(2,5,7,9)$ and $\beta=(4,5,7,8)$ be Schubert conditions in $\tbinom{[9]}{4}$.
 The variety $X_\alpha \Edot \cap X_\beta \Eop$ has local Stiefel coordinates
 \begin{equation}\label{eqn:StiefelStandardOpp}
  \left(
   \begin{matrix}
    0 & 1 & 0 & 0 & 0 & 0 & 0 & 0 & 0 \\
    0 & 0 & * & * & 1 & 0 & 0 & 0 & 0 \\
    0 & 0 & 0 & 0 & * & * & 1 & 0 & 0 \\
    0 & 0 & 0 & 0 & 0 & * & * & * & 1
   \end{matrix}
  \right)\,.
 \end{equation}
\end{example}

We describe flags in sufficiently general position.
\begin{definition}\label{def:generalFlags}
 The flags $\Fdot$ and $\Gdot$ in $V$ are in \emph{linear general position} if
 \[\dim (F_i\cap G_j) = \max\{0, i + j - n\} \quad \mbox{for} \quad i,j\in[n]\,.\]
\end{definition}
\begin{proposition}\label{prop:linGenFlagBases}
 If the flags $\Fdot$ and $\Gdot$ are in linear general position, then they have bases $(f_1,\dotsc,f_n)$ and $(g_1,\dotsc,g_n)$ respectively, such that $g_i=f_{n-i+1}$ for $i\in[n]$.
\end{proposition}
\begin{proof}
 We simply choose nonzero vectors $f_i\in F_i\cap G_{n-i+1}$, and Definition \ref{def:generalFlags} ensures that the sets $(f_1,\dotsc,f_i)$ and $(f_n,\dotsc,f_i)$ are each linearly independent.
\end{proof}
Let $\Fdot,\Gdot$ be flags in $V$ in linear general position and $\alpha,\beta\in \tbinom{[n]}{k}$ with $\alpha_i+\beta_{k-i+1}\geq n+1$ for $i\in[k]$.
Let $\bff$ be a basis of $\Fdot$ as described in Proposition \ref{prop:linGenFlagBases}, so that $(f_n,\dotsc,f_1)$ is a basis of $\Gdot$.
We observe that $\s_\alpha^\beta$ parametrizes the dense subset of $X_\alpha \Fdot \cap X_\beta \Gdot$, given by the linear span of the vectors
\[M_{i,\alpha_i} + \sum_{j = n-\beta_{k-i+1}+1}^{\alpha_i-1} M_{i,j}f_j\,,\]
for $i\in [k]$.

For an example, the row spaces of the matrices of the form (\ref{eqn:StiefelStandardOpp}) with respect to the standard basis $\be$ form a dense open subset of $X_\alpha\Edot \cap X_\beta \Edot'$.

\begin{definition}
 An intersection $X:=X_1\cap \cdots \cap X_m$ of subvarieties of a variety $G$ is said to be transverse at a point $x\in X$ if the equations defining the tangent spaces of $X_1,\dotsc,X_m$ at the point $x$ are in direct sum.
\end{definition}
\begin{definition}
 An intersection $X:=X_1\cap \cdots \cap X_m$ of subvarieties of a variety $G$ is said to be generically transverse if, for each component $Y\subset X$, there is a dense open subset $Z\subset Y$ such that $X$ is transverse at every point in $Z$.
 If $X$ is zero dimensional, then it is generically transverse if and only if it is transverse at every point $x\in X$.
\end{definition}
\begin{definition}
 Let $\balpha = (\alpha^1,\dotsc,\alpha^m)$ be a list of Schubert conditions in $\tbinom{[n]}{k}$.
 We define $|\balpha| := |\alpha^1|+\cdots+|\alpha^m|$.
\end{definition}
The following result is fundamental for Schubert calculus.
\begin{proposition}[Generic Transversality]\label{prop:genTrans}
 Let $\balpha = (\alpha^1,\dotsc,\alpha^m)$ be a list of Schubert conditions in $\tbinom{[n]}{k}$.
 If $\Fdot^1,\dotsc,\Fdot^m$ are general flags, then
 \begin{equation}\label{GeneralIntersection}
  X:=X_{\alpha^1} \Fdot^1 \cap \cdots \cap X_{\alpha^m} \Fdot^m
 \end{equation}
 is generically transverse.
 In particular, if $X$ is nonempty, then $\codim(X) = |\balpha|$.
\end{proposition}
Kleiman proved Proposition \ref{prop:genTrans} for algebraically closed fields of characteristic zero \cite{Kleiman}, and Vakil proved the analogue for algebraically closed fields of positive characteristic \cite{Va06b}.
\begin{remark}
 As an immediate consequence, if $\Fdot,\Gdot$ are general and $X_\alpha\Fdot\cap X_\beta\Gdot\neq \emptyset$, then
 \[\codim (X_\alpha \Fdot \cap X_\beta \Gdot) = |\alpha|+|\beta|\,.\]
 Straightforward calculation shows that $\s_\alpha^\beta$ has dimension $k(n-k)-|\alpha|-|\beta|$.
\end{remark}
\begin{definition}
 A list $\balpha = (\alpha^1,\dotsc,\alpha^m)$ of Schubert conditions on $\Gr(k,V)$ satisfying
 \[\sum_{i=1}^m |\alpha^i| = k(n-k) = \dim(\Gr(k,V))\]
 is called a \emph{Schubert problem on $\Gr(k,V)$}.
 By Proposition \ref{prop:genTrans}, given general flags $\Fdot^1,\dotsc,\Fdot^m$ on $V$, the intersection
 \[X := X_{\alpha^1} \Fdot^1 \cap \cdots \cap X_{\alpha^m} \Fdot^m\]
 is empty or zero-dimensional.
 We call $X$ an \emph{instance of the Schubert problem $\balpha$}.
\end{definition}
Since general flags are in linear general position, we may formulate an instance $X$ of $\balpha$ with minors involving local coordinates for $\Gr(k,V)$, $X_{\alpha^1}\Fdot^1$, or $X_{\alpha^1}\Fdot^1\cap X_{\alpha^2}\Fdot^2$.
The third formulation may be the most efficient for computation, since it involves the fewest determinantal equations and variables.
A \emph{real instance} of a Schubert problem $\balpha$ is an instance
\[X_{\alpha^1} \Fdot^1 \cap \cdots \cap X_{\alpha^m} \Fdot^m\,,\]
which is a real variety.
\begin{remark}
 Traditionally, Schubert calculus asks for the number of intersection points in a general instance of a Schubert problem.
 In this thesis, we study the number of intersection points with residue field $\R$ (i.e.\ real subspaces of $V$) in a real instance of a Schubert problem.
 We say that a real instance of a Schubert problem has been solved if we have successfully counted the number of real points in the intersection.
 We call the complex intersection points \emph{solutions} to the Schubert problem.
\end{remark}
\begin{definition}
 A \emph{parametrized rational normal curve} $\gamma\subset \p^{n-1}$ is a curve of the form
 \[\gamma(s,t):=(\gamma_1(s,t)\,,\dotsc,\gamma_n(s,t))\,,\quad\mbox{for}\quad (s,t)\in\p^1\,,\]
so that the components $\gamma_1,\dotsc,\gamma_n$ give a basis for the space of degree $n-1$ forms on $\p^1$.
 If each $\gamma_i$ has real coefficients, then we say that $\gamma$ is a \emph{real parametrized rational normal curve}.
\end{definition}
If $\gamma^1$ and $\gamma^2$ are parametrized rational normal curves, then their components give bases for the space of degree $n-1$ forms, so they differ by a change of basis $B\in \GL(n,\C)$,
\[\gamma^1(s,t)B=\gamma^2(s,t)\,.\]
Furthermore, if $\gamma^1$ and $\gamma^2$ are real, then they give real bases for the space of $n-1$ forms on $\p^1$, and there is a real change of basis $C\in\GL(n,\R)$,
\[\gamma^1(s,t)C=\gamma^2(s,t)\,.\]
Therefore, all real parametrized rational normal curves are equivalent by the action of $\GL(n,\R)$.

Throughout this thesis, we consider the real curve $\gamma(s,t)$ to be fixed.
While we may make different choices of $\gamma$ to facilitate proof, the resulting theorems hold for all other choices by applying the $\GL(n,\R)$ action.
\begin{example}
 The Veronese curve parametrized by
 \[\gamma(s,t) := (s^{n-1}\,,s^{n-2}t\,,\dotsc,st^{n-2}\,,t^{n-1})\]
 is a real parametrized rational normal curve.
 By convention, $\gamma(t) := \gamma(1,t)$ for $t\in\C$, and $\gamma(\infty) := \gamma(0,1)$.
\end{example}
\begin{definition}
 For $a\in \p^1$, the \emph{osculating flag $\Fdot(a)$} is the flag whose $i$th subspace $F_i(a)$ is the $i$-dimensional row space of the matrix,
 \begin{equation}\label{def:osculatingMatrix}
  F_i(a) := \left(
   \begin{matrix}
    \gamma(a)      \\
    \gamma'(a)      \\
    \vdots           \\
    \gamma^{(i-1)}(a) \\
   \end{matrix}\right)\,.
 \end{equation}
\end{definition}
If $\gamma$ is the Veronese curve, then $\Fdot(0)$ is the standard flag, and $\Fdot(\infty)$ is the standard opposite flag.
\begin{definition}
 Let $\alpha$ be a Schubert condition on $\Gr(k,V)$, $a\in \p^1$, and $\Fdot(a)$ the flag osculating $\gamma$ at $\gamma(a)$.
 We call the Schubert variety $X_\alpha (a):=X_\alpha \Fdot(a)$ an \emph{osculating Schubert variety}.
 We say that \emph{$X_\alpha (a)$ osculates $\gamma$ at $\gamma(a)$}.
\end{definition}

\section{Associated Schubert Varieties}

Let $V^*$ be the usual dual vector space to $V$.
The duality between $V$ and $V^*$ induces an association between the Grassmannian $\Gr(k,V)$ and the Grassmannian $\Gr(n-k,V^*)$.
We find it useful to study the corresponding association of Schubert varieties.
\begin{definition}
 Let $\Fdot$ be a flag in $V$.
 The flag $\Fdual$ \emph{dual to} $\Fdot$ is the flag in $V^*$ whose $i$-dimensional subspace $F_i^\perp$ is the annihilator of $F_{n-i}$ for $i\in[n-1]$,
 \[\Fdual : 0 \subsetneq (F_{n-1})^\perp \subsetneq \cdots \subsetneq (F_{1})^\perp \subsetneq F^\perp_n := V^*\,.\]
\end{definition}
The \emph{complement} of $\alpha\in\tbinom{[n]}{k}$ is the list $\alpha^c:=[n]\setminus \alpha$.
We realize a Schubert condition $\alpha\in\tbinom{[n]}{k}$ as a permutation $\sigma(\alpha)$ on $[n]$, by appending $\alpha^c$ to $\alpha$,
\[\sigma(\alpha) := (\alpha,\alpha^c)\,.\]
\begin{example}
 The Schubert condition $(1,3,6)\in\tbinom{[7]}{3}$ is a permutation
 \[\sigma(\alpha) = (1,3,6\,|\,2,4,5,7)\,.\]
 We use a vertical line in place of a comma to denote the position where the entries of $\sigma(\alpha)$ are allowed to decrease.
\end{example}
Let $\omega:=(n,n-1,\dotsc,2,1)$ be the longest permutation on $[n]$.
\begin{definition}
 Let $\alpha\in\tbinom{[n]}{k}$ be a Schubert condition.
 The Schubert condition $\alpha^\perp \in\tbinom{[n]}{n-k}$ associated to $\alpha$ is given by the composition of permutations
 \[\alpha^\perp := \omega\sigma(\alpha)\omega\,.\]
\end{definition}
\begin{example}
 Let $\alpha = (2,3) \in\tbinom{[5]}{2}$ be a Schubert condition.
 Writing
 \[\alpha^\perp = \omega (2,3\,|\,1,4,5) \omega = (1,2,5\,|\,3,4)\]
 as an element of $\tbinom{[5]}{3}$ gives the Schubert condition $\alpha^\perp = (1,2,5)$.
\end{example}
\begin{definition}
 Let ${\perp}:\Gr(k,V)\rightarrow \Gr(n-k,V^*)$ be the \emph{dual map}, mapping a $k$-plane to its annihilator, $H\mapsto H^\perp$.
 Since $(H^\perp)^\perp = H$, $\perp$ is a bijection.
\end{definition}
\begin{proposition}\label{prop:dualSV}
 Let $X_\alpha \Fdot \subset \Gr(k,V)$ be a Schubert variety.
 Then ${\perp}(X_\alpha \Fdot) = X_{\alpha^\perp} \Fdual$.
\end{proposition}
\begin{proof}
 Let $H\in X_\alpha \Fdot$.
 Definition (\ref{Def:SchubertVariety}) is equivalent to the condition
 \[\dim(H \cap F_i) \geq \#\{\alpha_j\in \alpha\ |\ \alpha_j\in[i]\}\]
 for $i\in[n]$.
 Equivalently, $\dim(\Span(H,F_i)) \leq k+i-\#\{\alpha_j\in \alpha\ |\ \alpha_j\in[i]\}$, so $\dim(\Span(H,F_i)^\perp)$ is at least
 \[n-k-i+\#\{\alpha_j\in \alpha\ |\ \alpha_j\in[i]\} = n-i-\#\{\alpha_j\in \alpha\ |\ \alpha_j \geq i+1\}\,.\]
 This yields
 \[\dim(\Span(H,F_i)^\perp) = \dim(H^\perp \cap F^\perp_{n-i}) \geq n-i-\#\{\alpha_j\in \alpha\ |\ \alpha_j \geq i+1\}\,.\]
 By changing indices and applying the definition of $\alpha^\perp$, we have
 \[\dim(H^\perp \cap F^\perp_{i}) \geq i-\#\{\alpha_j\in \alpha\ |\ \alpha_j \geq n-i+1\} = \#\{\alpha^\perp_j\in \alpha^\perp\ |\ \alpha^\perp_j \in[i]\}\,,\]
 for $i\in[n]$.
 This is equivalent to Definition \ref{Def:SchubertVariety} for $X_{\alpha^\perp}\Fdual$.
\end{proof}
Let $\Fdot$ be the standard flag, whose basis is given by the row vectors $e_1,\dotsc,e_n$.
Since $\Fdual$ is a flag in the dual space $V^*$, it has a dual basis of column vectors,
\[e_n^*=
 \left(
  \begin{matrix}
   0 \\ \vdots \\ 0 \\ 0 \\ 1
  \end{matrix}
 \right)\,,\ 
 e_{n-1}^*=
 \left(
  \begin{matrix}
   0 \\ \vdots \\ 0 \\ 1 \\ 0
  \end{matrix}
 \right)\,,\dotsc,\ 
 e_2^*=
 \left(
  \begin{matrix}
   0 \\ 1 \\ 0 \\ \vdots \\ 0
  \end{matrix}
 \right)\,,\ 
 e_1^*:=
 \left(
  \begin{matrix}
   1 \\ 0 \\ 0 \\ \vdots \\ 0
  \end{matrix}
 \right)\,.
\]
We adapt the coordinates (\ref{def:SAlpha}) on $X_\alpha \Fdot$, giving local coordinates on the associated Schubert variety $X_{\alpha^\perp} \Fdual$.
\begin{definition}
 Let $\alpha^\perp\in\tbinom{[n]}{n-k}$ be a Schubert condition for $\Gr(n-k,V^*)$.
 The set $\hats_{\alpha^\perp}\subset\Mat_{n\times (n-k)}$ consists of matrices $M$ whose entries satisfy
 \begin{equation}\label{def:HatSBeta}
  M_{n+1-\alpha^\perp_{i,j}} = \delta_{i,j}\quad\mbox{if}\quad i,j\in[n-k]\,,\quad\mbox{and}\quad M_{i,j} = 0\quad\mbox{if}\quad i < n+1-\alpha^\perp_j\,,
 \end{equation}
 and whose other entries are arbitrary.
\end{definition}
\begin{remark}
 The matrices of\/ $\hats_{\alpha^{\perp}}$ are related to transposes of the matrices of $\s_{\alpha^\perp}$.
 Suppose $M_{\alpha^\perp}$ is a matrix of indeterminates parametrizing $\s_{\alpha^\perp}$, and $N:=(\delta_{i,n-j+1})$ is the $n\times n$ matrix with ones along the antidiagonal.
 Then $\hats_{\alpha^\perp}$ is parametrized by the product
 \[M^{\alpha^\perp} := NM_{\alpha^\perp}\,.\]
\end{remark}
\begin{example}
 If $\alpha = (2,5)$ is a Schubert condition on $\Gr(2,6)$, then we have $\alpha^\perp = (1,3,4,6)$.
 The coordinates $S_\alpha$ and $\widehat{S}_{\alpha^\perp}$ are given by the matrices
 \[
  \left(
   \begin{matrix}
    a & 1 & 0 & 0 & 0 & 0\\
    b & 0 & c & d & 1 & 0
   \end{matrix}
  \right)\qquad\mbox{and}\qquad
  \left(
   \begin{matrix}
    0 &  0 &  0 &  1 \\
    0 &  0 &  0 & -a \\
    0 &  0 &  1 &  0 \\
    0 &  1 &  0 &  0 \\
    0 & -d & -c & -b \\
    1 &  0 &  0 &  0 \\
   \end{matrix}
  \right)\,.
 \]
 Note that choosing the arbitrary entries of one matrix determines those of the other so that each gives the null space of the other.
 This identification is canonical.
\end{example}
Let $(x_0,y_0)$ and $(x_1,y_1)$ be points in the Cartesian plane with $x_0>x_1$ and $y_0>y_1$.
A \emph{left step} is the vector $(-1,0)$, and a \emph{down step} is the vector $(0,-1)$.
A path from $(x_0,y_0)$ to $(x_1,y_1)$ is a sequence $p$ of length $L:=x_0-x_1+y_0-y_1$ of left steps and down steps such that $(x_0,y_0) + \sum_{i=1}^L p_i = (x_1,y_1)$.

\begin{definition}\label{def:path}
To $\alpha\in\tbinom{[n]}{k}$ we associate the path $p(\alpha)$ from $(n-k,0)$ to $(0,-k)$ given by
\[p(\alpha)_i = (0,-1)\mbox{ if }i\in \alpha\,,\quad\mbox{and}\quad p(\alpha)_i = (-1,0)\mbox{ if }i\not\in\alpha\,.\]
\end{definition}
The association $\alpha \leftrightarrow p(\alpha)$ is a bijection between Schubert conditions $\tbinom{[n]}{k}$ and paths from $(n-k,0)$ to $(0,-k)$.
\begin{example}
 If $\alpha = (2,5) \in \tbinom{[6]}{2}$ then $\alpha^\perp = (1,3,4,6) \in \tbinom{[6]}{4}$.
 Then $p(\alpha)$ and $p(\alpha^\perp)$ are given by thick lines in Figure \ref{fig:dualPaths}.
\begin{figure}[H]
  \begin{centering}
    $\includegraphics[scale=1.6]{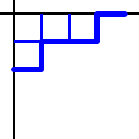}\qquad\qquad \includegraphics[scale=1.6]{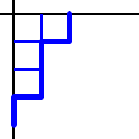}$
  \caption{\label{fig:dualPaths}$p(\alpha)$ and $p(\alpha^\perp)$.}
  \end{centering}
\end{figure}
\end{example}
\begin{proposition}\label{prop:perpSameCodim}
 We have the equality of codimensions $|\alpha| = |\alpha^\perp|$.
\end{proposition}
\begin{proof}
 Given $\alpha\in\tbinom{[n]}{k}$, $|\alpha|$ is equal to the area of the region enclosed by $p(\alpha)$ and the axes.
 Similarly, $|\alpha^\perp|$ corresponds to the region enclosed by $p(\alpha^\perp)$ and the axes.
 The path $p(\alpha^\perp)$ is the reflection of the path $p(\alpha)$ across the line $y=-x$, so the regions defining $|\alpha|$ and $|\alpha^\perp|$ have the same area.
\end{proof}
The enclosed regions in Figure \ref{fig:dualPaths} illustrate the equality $|\alpha| = 4 = |\alpha^\perp|$ for $\alpha = (2,5) \in \tbinom{[6]}{2}$.

\section{Osculating Schubert Calculus}

The study of osculating Schubert calculus is made possible by work of Eisenbud and Harris \cite{EH1983}.
They showed that given a set of Schubert varieties that osculate a rational normal curve at distinct points, their intersection is dimensionally transverse.
To prove this, we use a correspondence between Schubert calculus and the Wronskian which originated in work by Castelnuovo \cite{Castelnuovo1889}.
\begin{definition}
 Let $\C_n[t]$ be the vector space of polynomials in the variable $t$ of degree less than $n$ with coefficients in $\C$.
 The \emph{Wronskian} of $f_1,\dotsc,f_k\in\C_n[t]$ is the determinant
 \begin{equation}\label{def:Wr}
  \Wr(f_1,\dotsc,f_k):=\det
  \left(
   \begin{matrix}
    f_1         & \cdots & f_k      \\
    f'_1        & \cdots & f'_k      \\
    \vdots      &        & \vdots     \\
    f^{(k-1)}_1 & \cdots & f^{(k-1)}_k \\
   \end{matrix}
  \right)\,.
 \end{equation}
\end{definition}
Suppose $\bff:=(f_1,\dotsc,f_k)$ spans a $k$-dimensional subspace $H$.
If $\bg$ is another basis of $H$, and $B$ is a change-of-basis matrix such that $B\bff=\bg$, then $\det(B)\Wr(\bff)=\Wr(\bg)$.
Therefore, the roots of $\Wr(f_1,\dotsc,f_k)$ depend only on $H$.
\begin{proposition}\label{prop:degreeOfWronskian}
 Suppose $f_1,\dotsc,f_k\in\C_n[t]$ are complex univariate polynomials of degree at most $n-1$.
 The Wronskian $\Wr(f_1,\dotsc,f_k)$ is a univariate polynomial of degree at most $k(n-k)$.
\end{proposition}
\begin{proof}
 If $\bff$ is not linearly independent, then $\Wr(\bff)=0$, so we may assume that $\bff$ spans a $k$-plane $H\in \Gr(k,V)$.
 Representing polynomials in the monomial basis, we may assume $H$ is the row space of a matrix in reduced row echelon form.
 Denoting the rows by $g_1,\dotsc,g_k$, we have
 \[\deg(g_1) > \cdots > \deg(g_k)\,.\]

 Since $\bg$ and $\bff$ span the same $k$-plane, their Wronskians have the same roots, so $\deg(\Wr(\bff))=\deg(\Wr(\bg))$.
 Let $M$ denote the matrix in Definition (\ref{def:Wr}) giving $\Wr(\bg)$, whose entries are polynomials.
 Since $\deg(g_i)\leq n-i$ for $i\in[n]$, we have $\deg(M_{ij})\leq n-i-j+1$.
 It follows directly that $\deg(\det(M)) \leq k(n-k)$.
\end{proof}
\begin{remark}
 In general, the upper bound $k(n-k)$ on the degree of $\Wr$ is attained.
 In particular, we will prove Proposition \ref{prop:WrOrdAt0}, which implies that if $H$ is a solution to an instance of a Schubert problem involving only osculating hypersurface Schubert varieties, then $\Wr(H)$ has $k(n-k)$ distinct roots in $\p^1$.
\end{remark}
Since the Wronskians of bases $\bff$ and $\bg$ of a $k$-plane $H\in\C_n[t]$ are proportional, the Wronskian induces a well-defined map, called the \emph{Wronski map},
\[\Wr:\Gr(k,\C_n[t])\longrightarrow \p\C_{k(n-k)+1}[t]\,.\]
By Proposition \ref{prop:degreeOfWronskian}, $\dim(\p\C_{k(n-k)+1}[t]) = k(n-k)= \dim(\Gr(k,\C_n[t]))$.

The proofs of the following Proposition and Corollary are based on an argument in \cite{Sottile2011}.
Recall the definition (\ref{def:osculatingMatrix}) of the matrix $F_i(a)$.
\begin{proposition}\label{prop:WrCorresp}
 Let $V=\C_n[x]$ have standard basis $(1,x,\dotsc,x^{n-1})$, let $H\in\Gr(k,V)$, and let $L:=H^{\perp}\in\Gr(n-k,V^*)$ be the annihilator of $H$.
 If $F_k(x)$ is the matrix corresponding to the $k$-planes in $V$ osculating the Veronese curve $\gamma(t):=(1,t,\dotsc,t^{n-1})$ at $\gamma(x)$, then $L$ is the row space of a $(n-k)\times n$ matrix, also denoted by $L$, with
 \begin{equation}\label{eqn:Castelnuovo}
  \det \left(
   \begin{matrix}
    F_k(t) \\ L
   \end{matrix}
  \right) =
  \Wr(H) \in \p\C_{k(n-k)+1}[x]\,. 
 \end{equation}
\end{proposition}
\begin{proof}
 We prove this for $H$ with the general property that $\Wr(H)$ has $k(n-k)$ distinct roots.
 The other cases follow by a limiting argument.
 We reverse the roles of $\Gr(k,V)$ and $\Gr(n-k,V^*)$, so we consider $H^\perp\subset V^*$ to be spanned by row vectors and $H\subset V$ to be spanned by column vectors $h_1(x),\dotsc,h_k(x)$.

 Set $\bh:=(h_1,\dotsc,h_k)\in\Mat_{n\times k}$, where $h_i$ be the column vector of coefficients in $\C^n$ such that the polynomial $h_i(x)$ is the dot product $\gamma(x)\cdot h_i^T$.
 We observe that the product $F_k(x)\bh$ is the matrix given in Definition (\ref{def:Wr}) giving $\Wr(\bh)$, and $\rowspace(\bh) = H$, so $\det(F_k(x)\bh) = \Wr(H)$.
 Since $L$ is the null space of $H$, the determinant $W:=\det\binom{F_k(x)}{L}$ and $\Wr(H)$ vanish at the same points.

 Laplace expansion along the first $k$ rows of $\binom{F_k(x)}{L}$ gives
 \[W=\sum_{\alpha}(-1)^{(k-1)(n-k)+\sum_i \alpha_i}L_{\alpha}F_k(x)_{\alpha^c}\,,\]
 where $L_{\alpha}$ is the maximal minor of $L$ involving columns $\alpha$, and $F_k(x)_{\alpha^c}$ is the maximal minor of $F_k(x)$ involving columns $\alpha^c$.
 Thus, we have an upper bound for the degree of $W$,
 \[\deg(W) \leq \deg(F_k(x)_{(n-k+1,\dotsc,n)}) = k(n-k)\,.\]
 
 Since $W$ vanishes at the $k(n-k)$ distinct roots of $\Wr(H)$, $\deg(W) = k(n-k)$.
 Since $W$ and $\Wr(H)$ have the same roots and the same degree, they are proportional.
\end{proof}
\begin{corollary}
 If $H\in \Gr(k,V)$, then $H$ is contained in the hypersurface $X_{\Is} (t)$ for at most $k(n-k)$ values of $t\in \C$.
\end{corollary}
\begin{proof}
 We argue in the a Grassmannian $\Gr(n-k,V^*)$.
 Let $L$ denote both a matrix $L\in \St(n-k,n)$ and a $(n-k)$-plane $L\in \Gr(n-k,V^*)$, so that $\rowspace(L) = L$.
 As we have previously observed, $X_{\Is} (t)\subset \Gr(n-k,V^*)$ has one relevant condition given by $\det\binom{F_k(t)}{L} = 0$ for $L \in X_{\Is} (t)$.
 So by Proposition \ref{prop:WrCorresp}, choosing a $n\times k$ matrix $H$ with $H:=\colspace(H) = L^\perp$, we have $\Wr(\colspace(H)) = \det\binom{F_k(t)}{L}$ as a point in $\p\C_{k(n-k)+1}[x]$.
 Since
 \[\deg\left(\det\tbinom{F_k(t)}{L}\right) = \deg(\Wr(H)) \leq k(n-k)\,,\]
 there are at most $k(n-k)$ values of $t$ for which $\det\binom{F_k(t)}{L}=0$.
 Equivalently, there are at most $k(n-k)$ values of $t$ for which $L\in X_{\Is}(t)$.
 By Proposition \ref{prop:dualSV}, we reverse the roles of $\Gr(k,V)$ and $\Gr(n-k,V^*)$, giving the result.
\end{proof}
\begin{proposition}\label{prop:WrOrdAt0}
 Let $H \in X_\alpha (0)$.
 Then $\Wr(H^\perp)$ has a root at $x = 0$ of order at least $|\alpha|$.
\end{proposition}
\begin{proof}
 Using the notation of Proposition \ref{prop:WrCorresp}, we prove the dual statement, that is, if $L = H^\perp \in X_{\alpha^\perp} (0) \subset \Gr(n-k,V^*)$ then $\Wr(H)$ has a root at $0$ of order at least $|\alpha^\perp|$.
 Since $X_{\alpha^\perp} (0)$ has local coordinates $\s_{\alpha^\perp}$, we use coordinates $\hats_\alpha$ for $X_\alpha (\Fdot(0))^\perp$.
 Thus the columns $h_j$ form a basis of $H$ where $h_{ji} = 0$ if $i<n+1-\alpha_j$.
 Let $H$ denote the $n\times (n-k)$ matrix with these columns, so that the determinant of the product $F_k(x)H$ is $\Wr(H)$.
 Since $h_{ji} = 0$ for $i<n+1-\alpha_j$, every term of $\Wr(H) = \det(F_k(x)H)$ has degree at least
 \[\sum_{j=1}^k n+1-\alpha_j-j = -k(k+1) + \sum_{j=1}^k n+1-(\alpha_j-j) = |\alpha|\,.\]
 By Proposition \ref{prop:perpSameCodim}, every term of $\Wr(H)$ has a root at $0$ of order at least $|\alpha^\perp|$.
\end{proof}
Recall that the parametrized rational normal curve curve $\gamma(t)$ is in fact a local parametrization of the curve $\gamma(s,t)$ with $(s,t)\in\p^1$.
Thus the action of $\SL(2,\C)$ on $\p^1$ induces a dual action on $\gamma(t)$.
\begin{corollary}\label{cor:WrHigherOrder1}
 Let $H \in X_\alpha (t)$ for some $t\in\C$.
 Then $\Wr(H^\perp)$ has a root at $x = t$ of order at least $|\alpha|$.
\end{corollary}
\begin{proof}
 Using the $\SL(2,\C)$ action on $\p^1$ we may assume $t=0$.
 Using the $\GL(n,\C)$ action on $\gamma$ we may further assume $\gamma(x)=(1,x,\dotsc,x^{n-1})$ is the Veronese curve.
 Thus the flag defining $X_\alpha(x)$ has basis
 \[
  \Fdot(x) =
  \left(
   \begin{matrix}
    1      & x  & x^2 & \cdots & x^{n-1} \\
    0      & 1      & 2x  & \cdots & (n-1)x^{n-2} \\
    0      & 0      & 2       & \cdots & (n-1)(n-2)x^{n-3} \\
    \vdots & \vdots & \vdots  &        & \vdots \\
    0      & 0      & 0       & \cdots & (n-1)!
   \end{matrix}
  \right)\,.
 \]
 A direct calculation using \ref{prop:WrCorresp} shows the lowest-degree term of the Wronskian $\Wr(H)$ is $(-1)^{|\alpha^\perp|}p_{\alpha^\perp}(H)x^{|\alpha^\perp|}$, where $p_\bullet(H)$ are the Pl\"ucker coordinates of the null space $L$.
 Since $|\alpha| = |\alpha^\perp|$, the result follows.
\end{proof}
Recall the open cover $\mathcal{U}$ of Pl\"ucker space from Definition \ref{def:Ucover}, which restricts to an open cover
\begin{equation}\label{def:Galpha}
 \mathcal{G} := \{G_\alpha := \Gr(k,V) \cap U_\alpha\ |\ \alpha \in \tbinom{[n]}{k}\}\,.
\end{equation}
\begin{definition}\label{def:UalphaT}
 The matrix $\Fdot(t)^{-1}$ acts on $X_\alpha(0)$, giving $X_\alpha(0).\Fdot(t)^{-1} = X_\alpha(t)$.
 We define $\mathcal{G}(t)$ to be the collection of dense open sets of $\Gr(k,V)$ defined by the corresponding action,
 \[G_\alpha(t) := G_\alpha.\Fdot(t)^{-1}\quad\mbox{for}\quad G_\alpha \in \mathcal{G}\,.\]
\end{definition}
The lower bound on the order of vanishing of $\Wr(H)$ at $t=0$ given in the proof of Proposition \ref{prop:WrOrdAt0} is attained for all $H$ in the dense open subset $X_\alpha(t) \cap G_\alpha(t)$ of $X_\alpha(t)$.
This proves a stronger statement.
\begin{corollary}\label{cor:WrHigherOrder}
 Let $H \in X_\alpha (t)\cap G_\alpha(t)$ for some $t\in\C$.
 Then $\Wr(H^\perp)$ has a root at $x = t$ of order $|\alpha|$.
\end{corollary}
Given a list of Schubert conditions $\balpha = (\alpha^1,\dotsc,\alpha^m)$, we define
\[|\balpha| := |\alpha^1|+\cdots+|\alpha^m|\,.\]
We may now prove dimensional transversality for intersections of osculating Schubert varieties.
\begin{theorem}[Eisenbud-Harris]\label{thm:dimTrans}
 Let $\balpha = (\alpha^1,\dotsc,\alpha^m)$ be a list of Schubert conditions on $\Gr(k,V)$ and $a_1,\dotsc,a_m\in \p^1$ be distinct points.
 If the intersection
 \begin{equation}\label{intersection}
  X := X_{\alpha^1} (a_1) \cap \cdots \cap X_{\alpha^m} (a_m)
 \end{equation}
 is nonempty, then $\codim(X) = |\balpha|$.
\end{theorem}
\begin{proof}
 Assume for a contradiction that $X$ from (\ref{intersection}) has codimension $c < |\balpha|$.
 Consider distinct points $t_1,\dotsc,t_{k(n-k)-c}\in \p^1\setminus\{a_1,\dotsc,a_m\}$.
 Since $\dim X = k(n-k)-c$, and $X_{\Is} (t_i)$ is a hyperplane section for each $i$, we have
 \[X\cap X_{\Is} (t_1) \cap \cdots \cap X_{\Is} (t_{k(n-k)-c})\neq \emptyset\,.\]
 Let $H$ be a point in this intersection.
 By Proposition \ref{prop:degreeOfWronskian}, $\Wr(H)$ is a polynomial of degree at most $k(n-k)$.
 However, by Corollary \ref{cor:WrHigherOrder}, $\Wr(H)$ has $|\balpha|+k(n-k)-c > k(n-k)$ roots, which is a contradiction.
\end{proof}
\begin{proposition}\label{prop:determinedSCandOP}
 A $k$-plane $H\in\Gr(k,V)$ uniquely determines a Schubert problem $\balpha$ and an osculating instance $X$ of $\balpha$ with $H\in X$.
\end{proposition}
\begin{proof}
 Suppose $H$ is a solution to instances $X_1,X_2$ of Schubert problems $\balpha,\bbeta$,
 \[X_1 := X_{\alpha^1} (a_1) \cap \cdots \cap X_{\alpha^m} (a_m)\qquad\mbox{and}\qquad X_2 := X_{\beta^1} (b_1) \cap \cdots \cap X_{\beta^p} (b_p)\,.\]
 We may use the action of $\SL(2,\C)$ on each $X_i$ to avoid having any osculation points at $\infty$.
 This induces an invertible action on $H$, so we lose no generality in doing this.

 Since $\balpha$ and $\bbeta$ are Schubert problems, $|\balpha| = |\bbeta| = k(n-k)$.
 By Corollary \ref{cor:WrHigherOrder}, we have the equality
 \[\prod_{i=1}^m(x-a_i)^{|\alpha^i|} = \Wr(H) = \prod_{i=1}^p(x-b_i)^{|\beta^i|}\,,\]
 in projective space.
 So $m=p$, and we may reorder the Schubert varieties involved in $X_2$ so that $a_i=b_i$ and $|\alpha^i|=|\beta^i|$ for $i\in[m]$.
 Assume for a contradiction that $\alpha^i \neq \beta^i$ for some $i$ (without loss of generality, $i=1$).
 Thus $H\in X_{\beta^1}(a_1) \cap X_{\alpha^1}(a_1) = X_{\omega}(a_1)$ where $\omega\in\tbinom{[n]}{k}$ is given by
 \[\omega_i := \min\{\beta^1_i,\alpha^1_i\}\quad\mbox{for}\quad i\in [k]\,,\]
 and so
 \begin{equation}\label{eqn:determinedContradiction}
    H\in X_{\omega} (a_1) \cap X_{\alpha^2} (a_2) \cap \cdots \cap X_{\alpha^m} (a_m)\,.
 \end{equation}
 Since $\alpha^i \neq \beta^i$, we have $|\omega| > |\alpha^1|$, so $|\omega|+|\alpha^2|+\dotsb+|\alpha^m| > |\balpha| = k(n-k)$, which implies the intersection (\ref{eqn:determinedContradiction}) is empty by Theorem \ref{thm:dimTrans}.
 This contradiction implies $\alpha^i = \beta^i$ for all $i$, proving the statement.
\end{proof}


\section{The Shapiro Conjecture}

The dimensional transversality of Eisenbud and Harris shows that it is reasonable to study the Schubert calculus of osculating Schubert varieties.
In 1993, the brothers Boris and Michael Shapiro made the remarkable conjecture that an instance of a Schubert problem in a Grassmannian given by real osculating Schubert varieties has all solutions real.
The conjecture was proved in \cite{MTV2009a,MTV2009b}.
\begin{theorem}[Mukhin-Tarasov-Varchenko]\label{Th:MTV}
 Let $\balpha = (\alpha^1,\dotsc,\alpha^m)$ be a Schubert problem on $\Gr(k,V)$.
 If $a_1,\dotsc,a_m\in \R\p^1$ are distinct, then the intersection
 \[X_{\alpha^1} (a_1) \cap \cdots \cap X_{\alpha^m} (a_m)\]
 is transverse with all points real.
\end{theorem}
The Shapiro Conjecture may be seen in the first nontrivial Schubert problem, which asks how many 2-dimensional subspaces of $\C^4$ meet four fixed 2-dimensional subspaces nontrivially.
If the flags are general, the answer is two.
Theorem \ref{Th:MTV} asserts that both solutions are real and distinct if the flags involved osculate a rational normal curve at distinct real points.
We show this in the following example.
\begin{example}\label{ex:4lineDiscrim}
 Let $\gamma(t):=(1,t,t^2,t^3)$ parametrize the Veronese curve, and $\Fdot(t)$ be family of osculating flags.
 Suppose $t_1,\dotsc,t_4\in\R\p^1$ are distinct, and consider the four 2-dimensional subspaces $F_2(t_1),\dotsc,F_2(t_4)\subset \C^4$.
 We ask two questions: (1) how many 2-dimensional subspaces of $\C^4$ meet all four fixed subspaces nontrivially, and (2) how many \emph{real} 2-dimensional subspaces of $\C^4$ meet all four fixed subspaces nontrivially?

 We observe that Question (1) is a Schubert problem, and Question (2) is a real Schubert problem, because we are counting the points in the intersection
 \[X_{\Is}(t_1) \cap X_{\Is}(t_2) \cap X_{\Is}(t_3) \cap X_{\Is}(t_4)\,.\]

 Since $t_1,t_2$ are real and distinct, there is some $s\in\SL(2,\R)$ such that $t_1.s=0$, $t_2.s=\infty$, $t_3.s=:a\in\R$, and $t_4.s=:b\in \R$.
 Explicitly, if $t_1=(t_{11},t_{12})$ and $t_2=(t_{21},t_{22})$, then
\[ s = \left(
 \begin{matrix}
  t_{11} & t_{12} \\
  t_{21} & t_{22}
 \end{matrix}\right)^{-1}\,.
\]
 Since $s$ is invertible, the points $0,\infty,a,b$ are distinct.
 By a change of real basis on $(\gamma^1,\dotsc,\gamma^n)$, we may assume $\gamma(t)$ is the Veronese curve.
 Using these actions, we replace the flags $\Fdot(t_1),\dotsc,\Fdot(t_4)$ of Questions (1) and (2) by the flags $\Fdot(0)$, $\Fdot(\infty)$, $\Fdot(a)$, and $\Fdot(b)$, which does not affect whether solutions to the Schubert problem are real.
 
 The only relevant condition for $X_{\Is}(0)$ is that every $H\in\Gr(2,\C^4)$ meets $F_2(0)$ nontrivially.
 Similarly, if $H\in X$ it meets the other fixed 2-planes nontrivially.
 Thus Question (1) is given by counting the points in the intersection
 \[X := X_{\Is}(0) \cap X_{\Is}(\infty) \cap X_{\Is}(a) \cap X_{\Is}(b)\,.\]
 The intersection $X_{\Is}(0) \cap X_{\Is}(\infty)$ is parametrized by the matrix
 \[M:=\left(
  \begin{matrix}
   x & 1 & 0 & 0 \\
   0 & 0 & y & 1
  \end{matrix}   
 \right)\,,\]
 so we find the set on which $\rowspace(M)$ meets $F_2(t)$ nontrivially for $t=a,b$.
 This condition is equivalent to the equations
 \[\det\left(
   \begin{matrix}
    x & 1 & 0 & 0 \\
    0 & 0 & y & 1 \\
    1 & a & a^2 & a^3 \\
    0 & 1 & 2a & 3a^2
   \end{matrix}
  \right) = \det \left(
   \begin{matrix}
    x & 1 & 0 & 0 \\
    0 & 0 & y & 1 \\
    1 & b & b^2 & b^3 \\
    0 & 1 & 2b & 3b^2
   \end{matrix}
  \right) = 0\,.
 \]
 Thus we solve the system of equations
 \[
  \begin{array}{l}
   f:=-2xya^3+a^2x+3a^2y-2a = 0 \\
   g:=-2xyb^3+b^2x+3b^2y-2b = 0\,. \\
  \end{array}
 \]
 Using $f$ to eliminate the $xy$-term of $g$ yields
 \begin{equation}\label{eqn:yEquals}
   y=\frac{2a+2b-abx}{3ab}\,,
 \end{equation}
 which is defined since $a,b\neq 0$ and $a\neq b$ (we require $a\neq b$ to divide by $a-b$).
 Substituting back into $f=0$ and multiplying by the nonzero constant $\frac{3b}{2a^2}$ gives
 \begin{equation}\label{eqn:readyForDiscrim}
  abx^2-2(a+b)x+3=0\,.
 \end{equation}
 This equation has two solutions,
 \[x=\frac{a+b \pm \sqrt{a^2-ab+b^2}}{ab}\,,\]
 each determining a unique $y$-value by (\ref{eqn:yEquals}).
 We observe that the discriminant of (\ref{eqn:readyForDiscrim}) is a sum of squares,
 \[a^2-ab+b^2 = \frac{1}{2}a^2 + \frac{1}{2}b^2 + \frac{1}{2}(a-b)^2\,,\]
 so it is positive for all $a,b\neq 0$ with $a\neq b$.
 This implies that there are two distinct solutions, answering Question (1).
 
 Since the discriminant of (\ref{eqn:readyForDiscrim}) is positive, the two solutions of Question (1) have real $x$-values.
 These determine real $y$-values by Equation \ref{eqn:yEquals}, which implies that both solutions to the complex Schubert problem are real, answering Question (2).
 Question (2) is the first nontrivial example of Theorem \ref{Th:MTV}.
\end{example}

\section{The Problem of Four Real Tangent Lines}
The projective space $\p^{n-1}$ is the Grassmannian $\Gr(1,\C^n)$ of lines through the origin of $\C^n$.
That is, a 1-dimensional subspace of $\C^n$ is a point (or 0-dimensional affine space) in $\p^{n-1}$.
We extend this to higher dimensional subspaces and realize $\Gr(k,\C^n)$ as the set of $(k-1)$-dimensional affine spaces in $\p^{n-1}$.
In this way, Example \ref{ex:4lineDiscrim} is a question about lines which intersect four fixed lines in $\C^3$.

We illustrate this problem of four lines by giving another instance of Theorem \ref{Th:MTV}.
We assume $\gamma$ to be the twisted cubic curve in $\p^3$ parametrized by
\[\gamma(t) := \left( -1+6t^2,\, \frac{7}{2}t^3+\frac{3}{2}t,\, -\frac{1}{2}t^3+\frac{3}{2}t \right)\,,\]
and let $\ell_1,\ell_2,\ell_3,\ell_4$ be the fixed lines tangent to $\gamma(t)$ at $t=-1,0,1,\frac{1}{2}$ respectively.
This is the same curve used in \cite{Sottile2011}, chosen for aesthetic reasons.
Since all real rational normal curves are equivalent by a real change of basis, this curve is equivalent to the Veronese curve used in the previous example.

Since the family of quadric surfaces in $\p^3$ is 9-dimensional, and the restriction that a quadric $A$ contain a fixed line imposes $3$ independent conditions on that quadric, three mutually skew lines determine $A$.
Figure \ref{fig:curve} displays the ruling of the hyperboloid $A$ containing the lines $\ell_1,\ell_2,$ and $\ell_3$.
The lines in the opposite ruling are the lines in $\p^3$ which meet $\ell_1,\ell_2,$ and $\ell_3$.

\begin{figure}[H]
 \centering
  \includegraphics[width=.73\textwidth]{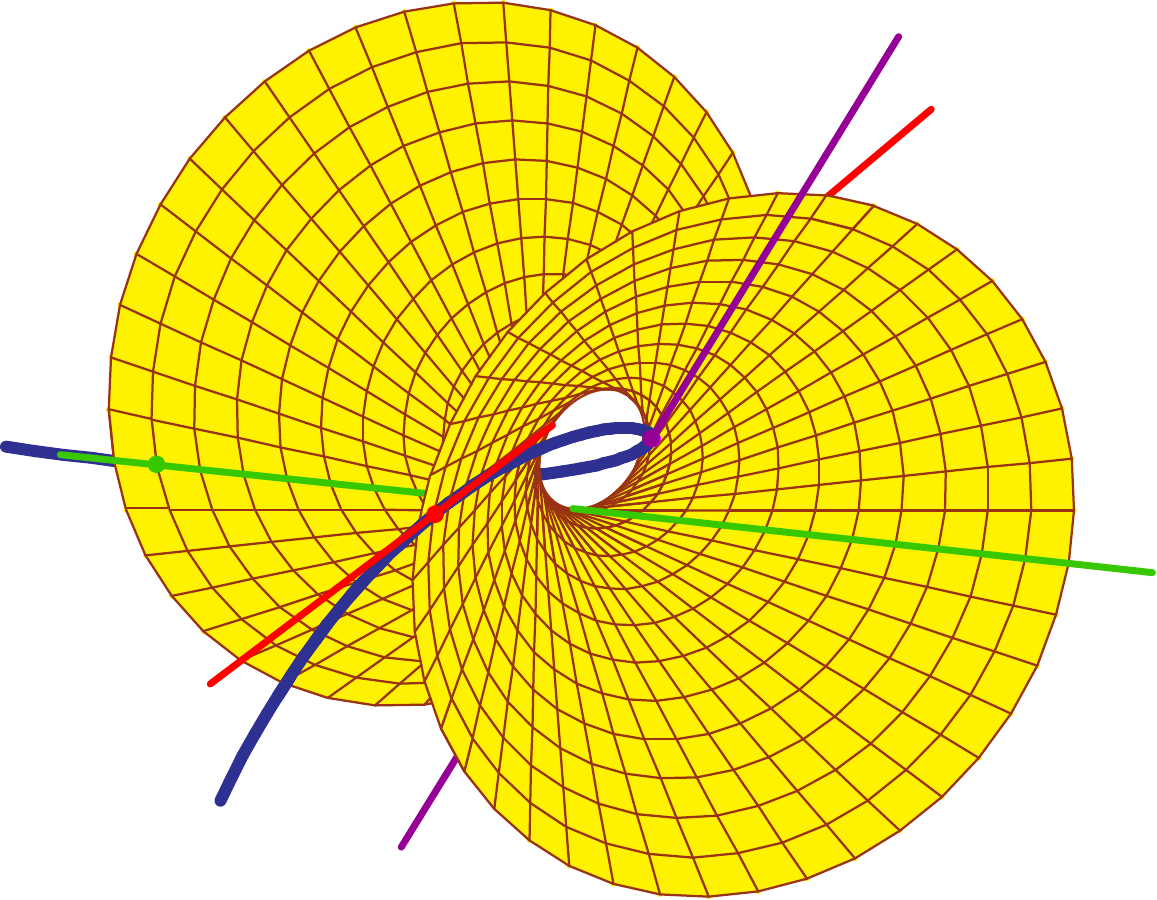}
  \caption{{\color{4LinesDarkBlue}$\gamma(t)$}, {\color{4LinesGreen}$\ell_1$}, {\color{4LinesPurple}$\ell_2$}, and {\color{4LinesRed}$\ell_3$}.}
  \label{fig:curve}
\end{figure}

Figure \ref{fig:solutions} shows the opposite ruling of $A$, containing the real lines meeting $\ell_1,\ell_2,$ and $\ell_3$.
The two lines meeting all four tangents are real if and only if the fourth tangent meets the hyperboloid at two real points, and in this case the lines containing those points are the two solutions.
The thick black line in Figure \ref{fig:solutions} is tangent to $\gamma$ at $\gamma\left(\frac{1}{2}\right)$, so the blue real lines are the two lines predicted by Schubert calculus when the four fixed lines are tangent at $t=-1,0,1,\frac{1}{2}$.

\begin{figure}[H]
 \centering
  \includegraphics[width=.68\textwidth]{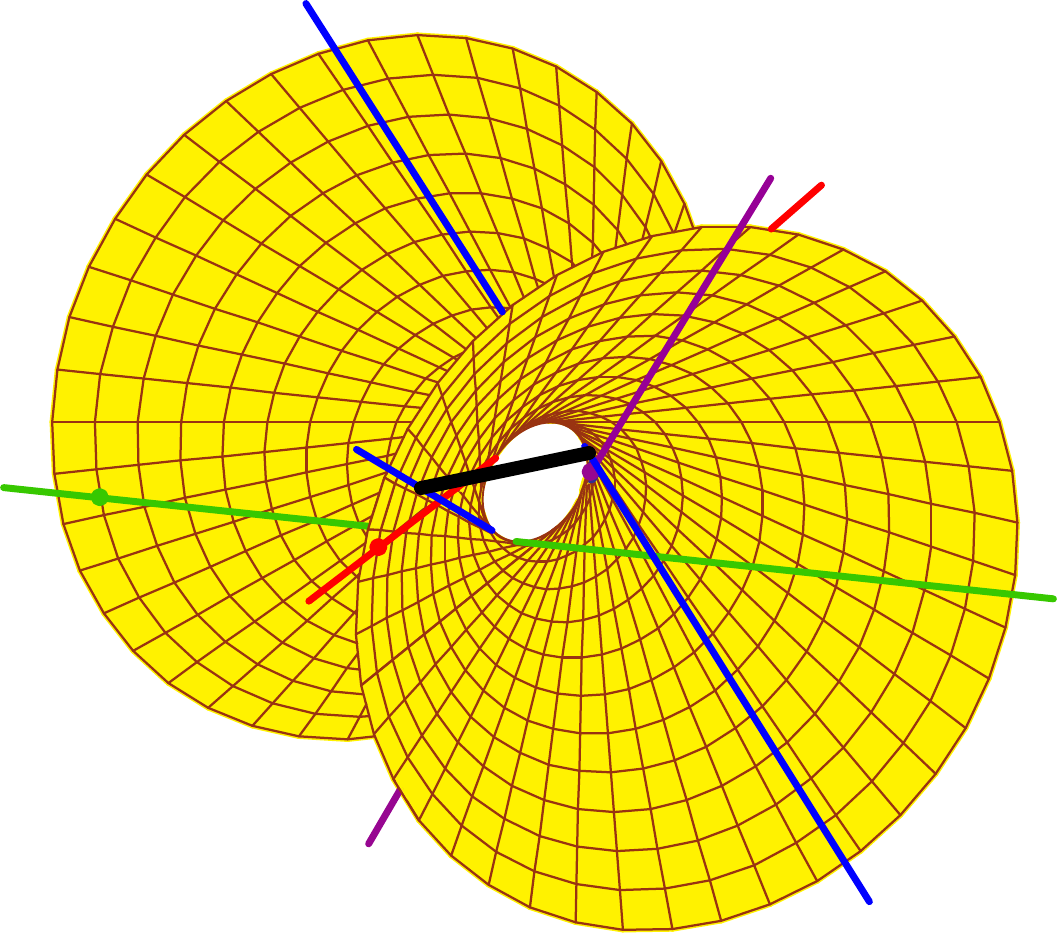}
  \caption{{\bf $\ell_4$} and {\color{4LinesBlue}two solution lines}.}
  \label{fig:solutions}
\end{figure}


\section{Conjectures with Computational Support}\label{sec:IWP}

Computer experimentation provided evidence in favor of the Shapiro Conjecture \cite{RS1998}, but further experimentation revealed that the most straightforward generalization to general flag varieties is false \cite{Sottile2000}.
After these computations, Eremenko and Gabrielov proved the Shapiro Conjecture for the Grassmannian of lines $\Gr(2,n)$ \cite{EG2002b}.
Mukhin, Tarasov, and Varchenko eventually proved the Shapiro Conjecture for all Grassmannians in type A \cite{MTV2009b}.

Computational experiments \cite{secant,monotone,lower,RSSS2006} suggested generalizations and variants of the Shapiro conjecture, some of which have been proven \cite{EGSV,mod4}.
We describe a variant of the problem which will be the focus of much of this thesis.

Recall the Wronski map from a Grassmannian to a projective space,
\[\Wr : \Gr(k,\C_n[t]) \longrightarrow \p\C_{k(n-k)+1}[t]\,.\]
Restricting the domain to the Grassmannian of polynomials with real coefficients, denoted by $\Gr(k,\R_n[t])$, gives the \emph{real Wronski map},
\[\Wr_\R : \Gr(k,\R\C_n[t]) \longrightarrow \p\R_{k(n-k)+1}[t]\,.\]
By Theorem \ref{thm:dimTrans}, the fibers of $\Wr$ are finite.
The problem of determining the number of points in a fiber of $\Wr$ is called the \emph{inverse Wronski problem}.

Since $\Wr_\R$ is a map between manifolds of the same dimension, it may have a topological degree, which gives a lower bound for the number of real points in the fiber $\Wr_\R^{-1}(f)$ over a general point $f\in\p\R_{k(n-k)+1}[t]$.
The main result of \cite{EG2002a} calculates this topological degree, finding nontrivial lower bounds on the number of points in a fiber $\Wr_\R^{-1}(f)$.
If $n$ is even, then either $\Gr(k,\R\C_n[t])$ or $\p\R_{k(n-k)+1}[t]$ is be orientable, and the topological degree, which we denote by $\sigma_{k,n}$, is zero.
If $n$ is odd, then neither $\Gr(k,\R\C_n[t])$ nor $\p\R_{k(n-k)+1}[t]$ is be orientable, and $\sigma_{k,n}$ may be nontrivial.
For $k\leq n-k$, we have
\begin{equation}\label{def:WronskiDegree}
 \sigma_{k,n} :=
  \frac{ 1! 2! \cdots (k{-}1)! (n{-}k{-}1)! (n{-}k{-}2)! \cdots (n{-}2k{+}1)! \left(\frac{k(n{-}k)}{2}\right)!}{(n{-}2k{+}2)!(n{-}2k{+}4)!\cdots(n{-}2)!\left(\frac{n{-}2k{+}1}{2}\right)!\left(\frac{n{-}2k{+}3}{2}\right)!\cdots\left(\frac{n{-}1}{2}\right)!}\,.
\end{equation}
If $n$ is odd and $k > n-k$, we have $\sigma_{k,n} = \sigma_{n-k,n}$.
This is a lower bound on the number of real points in the fiber $\Wr_\R^{-1}(f)$ over $f\in\p\R_{k(n-k)+1}[t]$ with $k(n-k)$ distinct roots.
We give the main theorem of \cite{EG2002a} in the language of Schubert calculus.
\begin{theorem}[Eremenko-Gabrielov]\label{Th:EG}
 Suppose $a \in (\p^1)^{k(n-k)}$ is a list of distinct points in $\p^1$.
 Furthermore, suppose $a$ is stable under complex conjugation, that is, $a_1,\dotsc,a_{k(n-k)}$ are the roots of a real polynomial.
 Then the real osculating Schubert problem
 \[X:=X_{\Is}(a_1) \cap \cdots \cap X_{\Is}(a_{k(n-k)})\]
 contains at least $\sigma_{k,n}$ real points.
\end{theorem}
The topological lower bounds of Eremenko and Gabrielov extend to Schubert problems of the form $\balpha=(\alpha,\It,\dotsc,\It)$ where $\alpha$ is an arbitrary Schubert condition.
Soprunova and Sottile extended these topological lower bounds to Schubert problems of the form $\balpha=(\alpha^1,\alpha^2,\It,\dotsc,\It)$ \cite{SS2006}, and we present their formula in Proposition \ref{prop:EGSS}.
When $k$ and $n$ are even, there is a choice of points in Theorem \ref{Th:EG} such that there are no real points in $X$, so the topological lower bound $\sigma_{k,n}=0$ is sharp \cite{EG2002a}.
For other cases, the lower bound given by $\sigma_{k,n}$ is not known to be sharp.

In the next two Chapters, we discuss a computational investigation of Eremenko-Gabrielov type lower bounds from the more general point of view of Schubert calculus and prove results inspired by the data.
We include a report on the observed sharpness of many of the bounds $\sigma_{k,n}$.
In this project, we keep track of whether each Schubert condition is associated to a real or nonreal osculation point to detect additional structure.

\chapter{\uppercase{Investigation of Lower Bounds}}
  \label{chapLower}
  The Mukhin-Tarasov-Varchenko Theorem \ref{Th:MTV} states that an instance of a Schubert problem on a Grassmannian involving Schubert varieties osculating a rational normal curve $\gamma$ at real points has all solutions real.
Intersections involving complex osculation points with the nonreal points coming in pairs may be real varieties, but they typically contain some nonreal points.
While the number of real points in such intersections may not be an invariant of the Schubert problem, there may be related invariants such as the topological lower bound on the number of real solutions in Theorem \ref{Th:EG} given by Eremenko and Gabrielov.
We study real instances of Schubert problems with the goal of understanding these invariants.
Thus we describe real osculating instances of Schubert problems as real enumerative problems and give a method for solving them.

\section{Real Osculating Instances}

We retain the conventions of Chapter \ref{chapSchubert}.
In particular, the results of this chapter depend on fixing a real parametrized rational normal curve $\gamma(t)$.
For a flag $\Fdot$ in $V$, we define the \emph{conjugate flag}
\[\barF : 0\subset \overline{F_1} \subset \cdots \subset \overline{F_n} = V\,.\]
\begin{proposition}
 Let $\alpha\in\tbinom{[n]}{k}$ be a Schubert condition and $\Fdot$ a flag in $V$.
 We have $\overline{X_\alpha \Fdot} = X_\alpha \barF$.
\end{proposition}
\begin{proof}
 For $H\in\Gr(k,V)$, we have the chain of equivalences
 \[
  \begin{array}{rcl}
   \overline{H}\in \overline{X_\alpha \Fdot}
    & \iff &  \dim(H\cap F_{\alpha_i}) \geq i\mbox{ for }i\in [k] \\
    & \iff & \dim\left(\overline{H\cap F_{\alpha_i}}\right) \geq i\mbox{ for }i\in [k] \\
    & \iff & \dim\left(\overline{H}\cap \overline{F_{\alpha_i}}\right) \geq i\mbox{ for }i\in [k] \\
    & \iff & \overline{H}\in X_\alpha \barF\,.
  \end{array}
 \]
\end{proof}
We note that $\overline{at^b} = a\overline{t}^b$ for any real number $a$ and any positive integer $b$.
\begin{corollary}\label{cor:SVbar}
 Let $\alpha\in\tbinom{[n]}{k}$ be a Schubert condition.
 We have $\overline{X_{\alpha}(t)} = X_{\alpha}(\overline{t})$.
\end{corollary}
An instance of the Schubert problem $\balpha=(\It\,,\dotsc,\It)$ is a real variety if the corresponding osculation points $t_1,\dotsc,t_m$ are the roots of a real polynomial of degree $m$.
We refine this condition to give criteria for an osculating instance of a general Schubert problem to be a real variety.
\begin{corollary}\label{cor:SVarsComeInPairs}
 Suppose $\balpha=(\alpha^1,\dotsc,\alpha^m)$ is a Schubert problem on $\Gr(k,V)$, and $|\alpha^i|>0$ for $i\in[m]$.
 Furthermore, suppose $t_1,\dotsc,t_m\in\p^1$ are distinct, and the instance
 \begin{equation}\label{eqn:svcip}
   X := X_{\alpha^1}(t_1) \cap \cdots \cap X_{\alpha^m}(t_m)
 \end{equation}
 of $\balpha$ is nonempty.
 Then $X$ is a real variety if and only if for $i\in[m]$ there exists $j\in[m]$ such that $X_{\alpha^i}(\overline{t_i}) = X_{\alpha^j}(t_j)$.
\end{corollary}
If $X$ from Equation (\ref{eqn:svcip}) is real, then we call it a \emph{real osculating instance of} $\balpha$.
\begin{proof}
 Suppose $X$ is a real variety and $H\in X$.
 Since $X = \overline{X}$, we have
 \[H\in \overline{X_{\alpha^1}(t_1)} \cap \cdots \cap \overline{X_{\alpha^m}(t_m)}\,.\]
 By Proposition \ref{prop:determinedSCandOP}, one may recover $\balpha$ and $t$ from $H$, so
 there is an involution
 \[(X_{\alpha^1}(t_1) ,\dotsc, X_{\alpha^m}(t_m))\longmapsto (\overline{X_{\alpha^1}(t_1)} ,\dotsc, \overline{X_{\alpha^m}(t_m)})\,.\]
 Applying Corollary \ref{cor:SVbar}, we have the forward implication.
 The reverse implication is elementary, since $\overline{X} = X$ implies $X$ is real.
\end{proof}
Let $\Re(a)$ and $\Im(a)$ denote the real and imaginary parts of a complex number.
Then the real part $\Re(f)$ or imaginary part $\Im(f)$ of a complex polynomial may be defined by taking the real or imaginary part respectively of the coefficients defining $f$.
\begin{proposition}\label{prop:realGenSet}
 The intersection $X_\alpha(t)\cap X_\alpha(\overline{t})$ of complex conjugate Schubert varieties is defined by the vanishing of the real and imaginary parts of the minors which define $X_\alpha(t)$.
\end{proposition}
When we have a real osculating instance of a Schubert problem, Proposition \ref{prop:realGenSet} gives us a real generating set for its ideal.
\begin{proof}
 Let $I$ be the ideal generated by the minors from Proposition \ref{prop:detConds} whose vanishing defines $X_\alpha(t)$, let $\overline{I}$ be the ideal with generators conjugate to those of $I$, and let $J$ be the ideal generated by the real and imaginary parts of the generators of $I$.
 Note that $\mathcal{V}(\overline{I})$ is $X_\alpha(\overline{t})$, and for a complex polynomial $f$, we have $\Re(f)=(f+\overline{f})/2$ and $\Im(f)=(f-\overline{f})/(2i)$.
 Since the generators of $J$ are complex linear combinations of the generators of $I$ and $\overline{I}$, we have $J \subset I+\overline{I}$.
 Similarly, the generators of $I$ are complex linear combinations of the generators of $J$, so $I\subset J$.
 By symmetry, $\overline{I}\subset J$, so $I+\overline{I} \subset J$.
 Therefore, $J=I+\overline{I}$, which implies that $J$ is the ideal of $X_\alpha(t)\cap X_\alpha(\overline{t})$.
\end{proof}

\section{Computations in Real Schubert Calculus}

We use modern software tools to study the real inverse Wronski problem as a problem in Schubert calculus.
Computation has been used to symbolically solve real osculating instances of Schubert problems with the more restrictive hypothesis that all osculation points are real \cite{secant,monotone,RSSS2006,Sottile2000}.
The framework used to interface with networks of computers for \cite{secant} (later adapted for \cite{monotone}) is described in \cite{framework}.
We adapt this framework to study Eremenko-Gabrielov type lower bounds for general Schubert problems (those involving more than two non-hypersurface Schubert varieties).

The data collected in \cite{secant,monotone} were gathered using the computer algebra system \url{Singular} with custom libraries.
We use the determinantal equations of Proposition \ref{prop:detConds} to formulate a real instance $X$ of a Schubert problem $(\alpha^1,\dotsc,\alpha^m)$ involving Schubert varieties which osculate the Veronese curve $\gamma(t)$ in local coordinates $\s(1,\dotsc,k)$, $\s_{\alpha^1}$, or $\s_{\alpha^1}^{\alpha^2}$ on $\Gr(k,V)$, $X_{\alpha^1}(0)$, or $X_{\alpha^1}(0) \cap X_{\alpha^2}(\infty)$ respectively.

We use the real generators from Proposition \ref{prop:realGenSet} to model the ideal of an instance $X$ of a Schubert problem, and we apply the tools of Section \ref{sec:prelim} to count the real solutions.
That is, we calculate an eliminant and then determine the number of real roots using a Sturm-Habicht sequence.
Our custom library calculates the real generators, and \url{schubert.lib} and the standard \url{Singular} libraries perform the remaining tasks.
This is structured with efficiency and repeatability in mind.
We use the software architecture of \cite{framework}, adapting the code of Hillar, et al.

We use a database hosted by the mathematics department at Texas A\&M University to keep track of instances of Schubert problems we wish to compute.
It also records results of computations, data needed to repeat the (pseudo)random computations, and enough information to recover from most errors.
The database is automatically backed up at regular intervals using \url{mysqldump}, and when otherwise unrecoverable errors occur, we use a \url{perl} script designed to repair the damaged part of the database using a recent backup.

Using algorithms based on the Littlewood-Richardson rule, we generated Schubert problems and determined the numbers of complex solutions to corresponding instances.
For each Schubert problem $\balpha$ studied, we run timing tests to compare computational efficiency subject to a choice of local coordinates, $\s(1,\dotsc,k),\s_{\alpha^1}$ or $\s_{\alpha^1}^{\alpha^2}$.
After making practical decisions, we assign a corresponding \emph{computation type} to $\balpha$, which denotes whether instances of $\balpha$ are to be solved using the coordinates $\s(1,\dotsc,k)$ for all osculation types or using $\s(1,\dotsc,k)$ for some types and $\s_{\alpha^1}$ or $\s_{\alpha^1}^{\alpha^2}$ for others.
We then load the Schubert problem $\balpha$ into the database.
This is automated by a script which generates an entry in a table used to keep track of pending requests to solve a reasonable number of random instances of $\balpha$.

This experiment is automated.
The scheduling program \url{crontab} periodically invokes scripts which check how many computations are running and submits job requests to a supercomputer.
Each job runs a \url{perl} wrapper which communicates with the database using standard \url{DBI::mysql} and \url{DBD::mysql} modules.
The main procedure queries the database for a computation request and then generates a \url{Singular} input file which models instances of the requested Schubert problem.
It then invokes \url{Singular} to run the input file.
The \url{Singular} process performs all tasks needed to count the number of real points in a randomly generated instance, and the \url{perl} wrapper records the results in the database.

This project continues to run on the brazos cluster at Texas A\&M University, a high-performance computing cluster.
We also benefited from the night-time use of the calclab, a Beowulf cluster of computers used by day for calculus instruction.

\section{Topological Lower Bounds and Congruences}

We denote the Schubert condition $\alpha$ in a visually appealing way by its Young diagram $d(\alpha)$, which is a northwest justified collection of boxes with $n-k+i-\alpha_i$ boxes in the $i$th row for $i=1,\dotsc,k$.
Immediately, one verifies that the number of boxes in $d(\alpha)$ is equal to $|\alpha|$, giving us immediate access the codimension of $X_\alpha\Fdot$.
\begin{example}
 The Schubert condition $\alpha = (3,6,8)$ on $\Gr(3,\C^8)$ has Young diagram
 \[d(\alpha) = {\IIIi}\,.\]
\end{example}
The shape above and to the left of the path $p(\alpha)$ from Definition \ref{def:path} is the same as the shape of $d(\alpha)$.
To express $\balpha$ in a compact way, we introduce \emph{exponential notation}.
Let $\widehat{\balpha} = (\widehat{\alpha}^1,\dotsc,\widehat{\alpha}^p)$ denote the distinct Schubert conditions comprising the Schubert problem $\balpha$, and let $\ba = (a_1,\dotsc,a_p)$ be an exponent vector.
Then $\widehat{\balpha}^\ba := ((\widehat{\alpha}^1)^{a_1},\dotsc,(\widehat{\alpha}^p)^{a_p})$ and $\balpha$ represent the same Schubert problem problem if $\balpha$ consists of exactly $a_i$ copies of $\widehat{\alpha}^i$ for $i\in[p]$.
We will often use $d(\alpha)$ in lieu of $\alpha$ when using exponential notation.
For an example, let $\widehat{\alpha}^1 = (5,6,9),\widehat{\alpha}^2 = (5,7,9),\widehat{\alpha}^3 = (6,8,9)\in\tbinom{[3]}{9}$.
The following represent the same Schubert problem,
\[\balpha := (\widehat{\alpha}^1,\widehat{\alpha}^1,\widehat{\alpha}^2,\widehat{\alpha}^2,\widehat{\alpha}^2,\widehat{\alpha}^3) \longleftrightarrow ({\IIiit}^2,{\IIit}^3,{\It}) =: \widehat{\balpha}^\ba\,.\]

Recall in a real osculating instance of a Schubert problem $\balpha$, some osculation points are real, while the rest come in complex conjugate pairs.
Given such an instance, we write $r_{\alpha}$ to denote the number of Schubert varieties involved with Schubert condition $\alpha$ osculating at a real point.

Suppose $\widehat{\balpha}$ and $\balpha$ represent the same Schubert problem.
If $\widehat{\alpha^j} = \alpha$, then $r_{\alpha} \equiv a_j \mod 2$.
We call $(r_\alpha\ |\ \alpha\in \widehat{\balpha})$ the \emph{osculation type} of the corresponding instance of $\balpha$.
\begin{example}
 The instance
 \[X_{\Is}(0) \cap X_{\Is}(\infty) \cap X_{\Is}(1) \,\cap\, X_{\IIs}(2) \cap X_{\IIs}(i) \cap X_{\IIs}(-i)\]
 in $\Gr(3,\C^6)$ has osculation type $(r_{\Is},r_{\IIs}) = (3,1)$.
\end{example}

The Mukhin-Tarasov-Varchenko Theorem \ref{Th:MTV}, asserts that a real instance of a Schubert problems with all osculation points real has all solutions real.
Eremenko and Gabrielov gave examples with other osculation types in which no solutions are real.
Thus the number of real solutions to a real osculating instance of a Schubert problem is sensitive to the osculation type, and we track this in our data.

Table \ref{tab:staircase} shows the observed frequency of real solutions after computing 400,000 random instances of $({\IIIIIt}\,,{\It}^7)$ in $\Gr(2,\C^8)$.
We leave a cell blank if there are no observed instances of the given type with the given number of real solutions.
Having tested 100,000 instances with exactly one pair of complex conjugate osculation points ($r_{\Is}=5$), none had only two real solutions, but 77,134 had exactly four real solutions.
We note that there are always six complex solutions to the Schubert problem, and the observed distribution in the $r_{\Is}=7$ column is forced by the Mukhin-Tarasov-Varchenko Theorem \ref{Th:MTV}, since all osculation points are real.
Collecting the data in Table \ref{tab:staircase} consumed 1.814 GHz-days of processing power.

\begin{table}[tp]
 \caption{Frequency table with inner border.}
  \begin{center}
   \begin{tabular}{|c||c|c|c|c||c|}\hline											
    {\bf $\#$ Real} & \multicolumn{4}{c||}{\rule{0pt}{12pt}${\IIIIIt}\quad{\It}^7$}               & \multirow{2}{*}{\bf Total} \\ \cline{2-5}
    {\bf Solutions} & $r_{\Is} = 7$ & $r_{\Is} = 5$ & $r_{\Is} = 3$ & $r_{\Is} = 1$ &                             \\ \hline \hline
    {\bf 0}         &              &              &              & 8964         & {\bf 8964}                        \\ \hline
    {\bf 2}         &              &              & 47138        & 67581        & {\bf 114719}                      \\ \hline
    {\bf 4}         &              & 77134        & 47044        & 22105        & {\bf 146283}                      \\ \hline
    {\bf 6}         & 100000       & 22866        & 5818         & 1350         & {\bf 130034}                      \\ \hline \hline
    {\bf Total}     & {\bf 100000 }& {\bf 100000 }& {\bf 100000 }& {\bf 100000 }& {\bf 400000}                      \\ \hline
   \end{tabular}											
  \end{center}
 \label{tab:staircase}											
\end{table}

In \cite{EG2002a}, Eremenko and Gabrielov gave lower bounds on the number of real solutions to a real osculating instance of a Schubert problem involving at most one Schubert variety not given by $\It$.
In \cite{SS2006}, Soprunova and Sottile extended these lower bounds to Schubert problem involving two nonhypersurface Schubert varieties.
We refer to these as topological lower bounds.
The following definitions allow us to calculate topological lower bounds.
\begin{definition}
 Let $\alpha\in\tbinom{[n]}{k}$ be a Schubert condition.
 The \emph{complementary Schubert condition $\alpha' \in\tbinom{[n]}{k}$} is
 \[\alpha'_i := n+1-\alpha_{k+1-i}\,,\quad\mbox{for}\quad i=1,\dotsc,k\,.\]
\end{definition}
It is illustrative to draw the Young diagrams of $\alpha$ and $\alpha'$ inside the diagram $d(1,\dotsc,k)$.
For example, if $k=3,n=4,\alpha=(2,5,7)$, then
\[d(\alpha) = {\IIIiGrIIIvii}\qquad\mbox{and}\qquad d(\alpha') = {\IIIIiiiIGrIIIvii}\,.\]

Recall that the Bruhat order on Schubert conditions $\alpha\leq \beta$ is given by $\alpha\leq \beta$ for $i\in[k]$.
This induces an order on diagrams so that $d(\alpha)\leq d(\beta)$ if $d(\alpha)$ fits inside $\d(\beta)$.
For example, $(1,3,6)\leq (3,5,7)$ in $\tbinom{[7]}{3}$, so
\[\IIIIiiiI \leq\, \IIi\,.\]
\begin{definition}
 Given $\alpha,\beta\in\tbinom{[n]}{k}$ with $\alpha \leq \beta$ in the Bruhat order, the \emph{skew Young diagram} $d(\alpha/\beta):=d(\alpha)/d(\beta)$ is the diagram $d(\alpha)$ with the boxes of $d(\beta)$ removed.
\end{definition}
For $\Gr(3,\C^7)$,
\begin{equation}\label{def:skewShape}
  \lambda := d((1,3,6)/(3,5,7)) = \IIIIiiiI\bigg/\IIi = {\IIIIiiiIminusIIi}\,.
\end{equation}
A \emph{standard Young tableau of shape $d(\alpha/\beta)$} is an association between the boxes of a skew Young diagram $d(\alpha/\beta)$ with $N$ boxes and the set $[N]$ which is increasing in each row from left to right and increasing in each column from top to bottom.
We give examples of standard Young tableaux of shape $\lambda$ defined in Equation (\ref{def:skewShape}) ,
\[
  \begin{picture}(127.5,36)
   \put (0,0){\IIIIiiiIminusIIib}
   \put (80,0){\IIIIiiiIminusIIib}

   \put (26.5,26){1}
   \put (38,26){2}
   \put (15,14.5){3}
   \put (26.5,14.5){4}
   \put (3.5,3){5}

   \put (106.5,26){1}
   \put (118,26){3}
   \put (95,14.5){2}
   \put (106.5,14.5){5}
   \put (83.5,3){4}

  \end{picture}
\]
The standard Young tableau of shape $d(\alpha/\beta)$ which associates the boxes of $d(\alpha/\beta)$ to the set $[N]$ in order from left to right starting with the top and working down is called the \emph{standard filling of $d(\alpha/\beta)$}.
The tableau on the left pictured above is the standard filling of $\lambda$.
The set of standard Young tableaux of shape $d(\alpha/\beta)$ is denoted $\SYT(d(\alpha/\beta))$.

The diagram $\IIiit$ has two standard fillings,
\[
 \begin{picture}(84.5,24.5)
   \put (0,0){\IIiib}
   \put (60,0){\IIiib}

   \put (3.5,14){1}
   \put (15,14){2}
   \put (3.5,2.5){3}
   \put (15,2.5){4}

   \put (63.5,14){1}
   \put (75,14){3}
   \put (63.5,2.5){2}
   \put (75,2.5){4}

  \end{picture}\,,
\]
so $\#\left(\SYT\left(\IIiit\right)\right)=2$.

A tableau of shape $\IIIiiiIminusIit$ may have $i$ in the southwest box for $i\in [5]$.
The order of the other boxes given by their entries $(1,\dotsc,\widehat{i},\dotsc,5)$ is the same as the order in one of the standard tableaux of shape $\IIiit$, so
\[\#\left(\SYT\left(\IIIiiiIminusIi\right)\right) = 5\cdot \#\left(\SYT\left(\IIii\right)\right)=10\,.\]

Every standard Young tableau $T$ of shape $d(\alpha/\beta)$ has a parity, $\sign(T)=\pm 1$, which is the parity of the permutation mapping the standard filling to $T$.
\begin{definition}
 Suppose $\alpha,\beta\in \tbinom{[n]}{k}$, and $\alpha'\leq \beta$.
 The \emph{sign imbalance of $\alpha'/\beta$} is
 \[\Sigma(\alpha,\beta):=\left| \sum_{T\in \SYT(\alpha'/\beta)}\sign(T) \right|\,.\]
\end{definition}
\begin{proposition}[Soprunova-Sottile]\label{prop:EGSS}
 Suppose $\alpha,\beta\in \tbinom{[n]}{k}$, $\alpha'\leq \beta$, and
 \[X:=X_\alpha(t_1) \cap X_\beta(t_2) \cap X_{\Is}(t_3) \cap \cdots \cap X_{\Is}(t_m)\]
 is a real osculating instance of a Schubert problem with $t_1,t_2\in\R\p^1$.
 Then $X$ contains at least $\Sigma(\alpha,\beta)$ real points.
\end{proposition}
The lower bound $\Sigma(\alpha,\beta)$ is obtained by calculating the topological degree of a map, and we it a \emph{topological lower bound}.
 If $\alpha=\beta=\It$, then $\Sigma(\alpha,\beta)=\sigma(k,n)$ from definition (\ref{def:WronskiDegree}).

Of the 756 Schubert problems we have studied so far, 267 of them have associated topological lower bounds $\Sigma(\alpha,\beta)$ for the numbers of real solutions, and the other 489 involve intersections of more than two hypersurfaces.
We calculated sign imbalances and tested the sharpness of topological lower bounds $\Sigma(\alpha,\beta)$.
In cases where $k$ and $n$ are even, Eremenko and Gabrielov showed that their lower bound $\Sigma(\It\,,\It)=0$ is sharp.
This applies to three of the 267 Schubert problems we studied with $k=2$ and $n=4$, $6$, or $8$.
Our symbolic computations verified sharpness for 258 of the remaining 264 cases tested.
We do not give witnesses to these verifications here, but our stored data are sufficient for repeating these calculations.

Our data suggest that the other six lower bounds may be improved.
Table \ref{tab:sharpness1} gives frequency tables associated to two of these Schubert problems, $(\IIIiiit\,,\IiIt\,,\It^7)$ and $(\IIiiIIt\,,\IIIt\,,\It^7)$, each with 35 solutions in $\Gr(4,\C^8)$.
Theorem \ref{Th:MTV} asserts that if $r_{\Is} = 7$ then all 35 solutions are real.
We verified this fact 200,000 times for each problem, but we omit the data from the frequency table.
Since nonreal solutions come in pairs, we expect expect only odd numbers of real solutions, so we omit rows corresponding to even numbers of real solutions.

\begin{table}[tp]
 \caption{Topological lower bound 1, but observed lower bound 3.}
  \begin{center}
   \begin{tabular}{| c || c|c|c || c|c|c |}\hline
    {\bf $\#$ Real} &
    \multicolumn{3}{c||}{\rule{0pt}{12pt}$\IIIiiit \quad \IiIt \quad \It^7$} &
    \multicolumn{3}{c|}{$\IIiiIIt \quad \IIIt \quad \It^7$} \\ \cline{2-7}
    {\bf Solutions} & $r_{\Is}=5$ & $r_{\Is}=3$ & $r_{\Is}=1$ 
                  & $r_{\Is}=5$ & $r_{\Is}=3$ & $r_{\Is}=1$ \\ \hline \hline
    {\bf 1 } & {{	 	}} & {{     }} & {{  	}} & {{	 	}} & {{	 	}} & {{	 	}} \\ \hline	
    {\bf 3 } & {{	 	}} & {{16038}} & {{24070	}} & {{	 	}} & {{	16033	}} & {{	24184	}} \\ \hline	
    {\bf 5 } & {{5278	}} & {{34048}} & {{51572	}} & {{	5224	}} & {{	34096	}} & {{	51017	}} \\ \hline	
    {\bf 7 } & {{15817	}} & {{30992}} & {{28808	}} & {{	15769	}} & {{	30943	}} & {{	29449	}} \\ \hline	
    {\bf 9 } & {{41717	}} & {{34231}} & {{48405	}} & {{	41872	}} & {{	33992	}} & {{	48248	}} \\ \hline	
    {\bf 11} & {{17368	}} & {{24601}} & {{23458	}} & {{	17465	}} & {{	24839	}} & {{	23756	}} \\ \hline	
    {\bf 13} & {{15011	}} & {{14761}} & {{8559	}} & {{	14829	}} & {{	14805	}} & {{	8560	}} \\ \hline	
    {\bf 15} & {{13556	}} & {{10197}} & {{4686	}} & {{	13471	}} & {{	10478	}} & {{	4635	}} \\ \hline	
    {\bf 17} & {{7589	}} & {{6255}}  & {{2788	}} & {{	7650	}} & {{	6202	}} & {{	2816	}} \\ \hline	
    {\bf 19} & {{13462	}} & {{9744}}  & {{3329	}} & {{	13295	}} & {{	9670	}} & {{	3081	}} \\ \hline	
    {\bf 21} & {{5244	}} & {{3071}}  & {{1156	}} & {{	5337	}} & {{	3093	}} & {{	1060	}} \\ \hline	
    {\bf 23} & {{4785	}} & {{2256}}  & {{581	}} & {{	4816	}} & {{	2169	}} & {{	605	}} \\ \hline	
    {\bf 25} & {{17219	}} & {{5535}}  & {{1259	}} & {{	17335	}} & {{	5586	}} & {{	1262	}} \\ \hline	
    {\bf 27} & {{1587	}} & {{834}}   & {{176	}} & {{	1530	}} & {{	814	}} & {{	184	}} \\ \hline	
    {\bf 29} & {{3946	}} & {{1236}}  & {{235	}} & {{	4037	}} & {{	1242	}} & {{	289	}} \\ \hline	
    {\bf 31} & {{3558	}} & {{892}}   & {{159	}} & {{	3498	}} & {{	876	}} & {{	157	}} \\ \hline	
    {\bf 33} & {{711	}} & {{307}}   & {{73	}} & {{	631	}} & {{	262	}} & {{	75	}} \\ \hline	
    {\bf 35} & {{33152	}} & {{5002}}  & {{686	}} & {{	33241	}} & {{	4900	}} & {{	622	}} \\ \hline	\hline
    {\bf Total} & {{\bf 200000 }} & {{\bf 200000 }} & {{\bf 200000 }} & {{\bf 200000 }} & {{\bf 200000 }} & {{\bf 200000 }}  \\ \hline	
   \end{tabular}
  \end{center}
 \label{tab:sharpness1}
\end{table}

The problems given in Table \ref{tab:sharpness1} are dual to each other.
It is a consequence of the duality studied in Chapter \ref{chapMod4} that for a fixed set of osculation points these problems have the same number of real solutions.
This explains the remarkable similarity between the two distributions in Table \ref{tab:sharpness1}, and it implies that they have the same lower bounds.

The Schubert problems $(\IIIiiit\,,\IIiit\,,\It^6)$ and $(\IIiiIIt\,,\IIiit\,,\It^6)$ with 30 solutions in $\Gr(4,\C^8)$ are also dual to each other, and their frequency tables bear remarkable similarity.
They have topological lower bound $\Sigma = 0$, but after calculating 1.6 million instances of each we never observed less than $2$ real solutions.

\begin{table}[tp]
 \caption{Congruence modulo four.}
  \begin{center}
 \begin{tabular}{|c||c|c|c|c||c|}
\hline	{\bf $\#$ Real}	&	\multicolumn{4}{c||}{\rule{0pt}{12pt}$\It^9$} & \multirow{2}{*}{\bf Total} \\ \cline{2-5}	
{\bf Solutions}	& {\bf 	$r_{\Is}=7$	} & {\bf 	$r_{\Is}=5$	} & {\bf 	$r_{\Is}=3$	} & {\bf 	$r_{\Is}=1$	} &	\\ \hline \hline
{\bf 2  }    	 	&	1843	&	30223	&	34314	&	 	&	{\bf 66380}	\\ \hline
{\bf 6  }    	 	&	13286	&	51802	&	93732	&	151847	&	{\bf 310667}	\\ \hline
{\bf 10 }    	 	&	69319	&	57040	&	47142	&	35220	&	{\bf 208721}	\\ \hline
{\bf 14 }    	 	&	18045	&	17100	&	10213	&	6416	&	{\bf 51774}	\\ \hline
{\bf 18 }    	 	&	13998	&	12063	&	5532	&	2931	&	{\bf 34524}	\\ \hline
{\bf 22 }    	 	&	22883	&	15220	&	5492	&	2345	&	{\bf 45940}	\\ \hline
{\bf 26 }    	 	&	4592	&	2767	&	839	&	362	&	{\bf 8560}	\\ \hline
{\bf 30 }    	 	&	11603	&	4634	&	1194	&	450	&	{\bf 17881}	\\ \hline
{\bf 34 }    	 	&	3891	&	2056	&	504	&	181	&	{\bf 6632}	\\ \hline
{\bf 38 }    	 	&	473	&	211	&	65	&	22	&	{\bf 771}	\\ \hline
{\bf 42 }        	&	40067	&	6884	&	973	&	226	&  {\bf 48150}	\\ \hline \hline
{\bf Total } & {\bf 200000 } & {\bf 200000 } & {\bf 200000 } & {\bf 200000 } & {\bf 800000 } \\ \hline
 \end{tabular}
  \end{center}
 \label{tab:sharpness2}
\end{table}

Table \ref{tab:sharpness2} shows that we always observed at least two real solutions to $(\It^9)$, but $\Sigma(\It,\It)=0$, so its topological bound is apparently not sharp.

More strikingly, while the number of real solutions to a real instance of this problem must be congruent to 42 mod 2, we only observed instances with 42 mod 4 real solutions.
The stronger congruence modulo four in the number of real solutions is due to a geometric involution which we explain in Chapter \ref{chapMod4}.
Thus we will prove that $\Sigma(\It\,,\It)$ is not a sharp lower bound for the number of real solutions to a real osculating instance of $(\It^9)$.

The sixth and final topological lower bound which we did not find to be sharp is $\Sigma(\IIIiiiIIt,\It) = 0$ for $(\IIIiiiIIt\,,\It^8)$ having $90$ complex solutions in $\Gr(4,\C^8)$.
We omit the rather large frequency table but note that we observed the number of real solutions to be congruent to $90$ modulo four.
This congruence is related to that in Table \ref{tab:sharpness2}.
In \ref{chapMod4}, we see that two is the sharp lower bound for real osculating instances of this Schubert problem.

\section{Lower Bounds via Factorization}

For each Grassmannian, we describe a special Schubert problem $\balpha$, and following joint work with Hauenstein and Sottile \cite{lower}, we show that the number of real solutions to an instance of $\balpha$ has a lower bound depending only on osculation type.
In particular, we explain the inner border in Table \ref{tab:staircase} related to the Schubert problem $({\IIIIIt}\,,{\It}^7)$ in $\Gr(2,\C^8)$.

\begin{definition}\label{def:omega}
 Let $\omat:=(2,\dotsc,k,n)\in\tbinom{[n]}{k}$.
 Note that the diagram $d(\omat)$ has $n-k-1$ boxes in the $i$th row for $i<k$ and no boxes in the $k$th row.
 Equivalently, $d(\omat)$ has $k-1$ boxes in the $j$th column for $j<n-k$ and no boxes in the $(n-k)$th column.
\end{definition}
For $\Gr(2,8)$, $d(\omat) = {\IIIIIt}$, and for $\Gr(4,8)$, $d(\omat) = \IIIiiiIIIt$.
There are local coordinates similar to $\s_{\omas}$ for $X_{\omas}(\infty)$, given by matrices $M$ of the form
\[M:=
  \left(
    \begin{array}{ccccccccccc}
      c_1 & c_2a_2 & \cdots & c_{n-k}a_{n-k} & \frac{b_k}{b_{k-1}} & 0                       & 0 & 0 & \cdots & 0  \\
      0 & 0   & \cdots & 0       & \frac{1}{n-k+2}                   & \frac{b_{k-1}}{b_{k-2}} & 0 & 0 & \cdots & 0  \\
      0 & 0   & \cdots & 0       & 0                   & \frac{2}{n-k+3}                       & \frac{b_{k-2}}{b_{k-3}} & 0 & \cdots & 0 \\
      \vdots & \vdots & & \vdots & \vdots & \ddots & \ddots & \ddots & \ddots & \vdots \\
      0 & 0 & \cdots & 0 & 0 & \cdots & 0 & \frac{k-2}{n+3} & \frac{b_{2}}{b_{1}} & 0 \\
      0 & 0 & \cdots & 0 & 0 & \cdots & 0 & 0 & \frac{k-1}{n+2} & b_1 \\
    \end{array}
  \right)\,,
\]
where $a_2,\dotsc,a_{n-k},b_1,\dotsc,b_k$ are coordinates, $b_1,\dotsc,b_k$ are nonzero, and $c_i$ are constants,
\[c_i:=(-1)^n \frac{(n-k-i)!(i-1)!}{1!2!\cdots (n-k-1)!(n-k+1)!}\,.\]
The constants $c_i$ are introduced to simplify further calculations.
Consider the Schubert problem $\balpha=(\omat,{\It}^{n-1})$ and distinct points $t_1,\dotsc,t_{n}\in\p^1$ with $t_1=\infty$.
The intersection
\begin{equation}\label{eqn:LBfamily}
 X:=X_{\omas}(\infty) \cap X_{\Is}(t_2) \cap \cdots \cap X_{\Is}(t_{n})
\end{equation}
is a real osculating instance of $\balpha$.
We examine the determinantal conditions defining $X_{\Is}(t)$ for $t=t_2,\dotsc,t_{n}$ in the local coordinates $M$.
We write $M_\beta$ to denote the maximal minor of $M$ involving columns $\beta$, and we write $(F_{n-k}(t))_{\beta^c}$ to denote the maximal minor of $F_{n-k}(t)$ involving columns $\beta^c$.
Expanding along the rows of $M$ gives
\begin{equation}\label{eqn:LaPlace}
\det
  \left(
    \begin{matrix}
      M \\
      F_{n-k}(t)
    \end{matrix}
  \right) = 
  (-1)^{k(n-k)} \sum_{\beta \in \tbinom{[n]}{k}} (-1)^{|\beta|} M_\beta (F_{n-k}(t))_{\beta^c}\,.
\end{equation}
The nonzero maximal minors of $M$ involve at most one of the first $n-k$ columns, so they are indexed by $k$-tuples of the form
\[[i,\widehat{j}]:=(i,n-k+1,\dotsc,\widehat{n{-}k{+}j},\dotsc,n)\,,\]
with $i\in [n-k]$ and $j\in[k]$, or of the form $[n-k]^c:=(n-k+1,\dotsc,n)$.
Defining $a_1:=1$ and $b_0:=1$, we have
\[M_{[i,\widehat{j}]} = \frac{1}{\tbinom{n-k+j}{j-1}} c_ia_ib_{k-j}\qquad\mbox{and}\qquad M_{[n-k]^c} = b_k\,.\]
For any $\beta\in\tbinom{[n]}{n-k}$ we have
\[
  {\displaystyle{ \begin{array}{rl}
  {\displaystyle{(F_{n-k}(t))_{\beta} }}& {\displaystyle{ = \det
  \left(
    \begin{matrix}
      t^{\beta_1-1} & \cdots & t^{\beta_{n-k}-1} \\
      (\beta_1-1) t^{\beta_1-2} & \cdots & (\beta_{n-k}-1) t^{\beta_{n-k}-2} \\
      \vdots & & \vdots \\
      \frac{(\beta_1-1 )!}{(\beta_{1}-n+k)!} t^{\beta_{1}-n+k} & \cdots & \frac{(\beta_{n-k}-1) !}{(\beta_{n-k}-n+k)!} t^{\beta_{n-k}-n+k} \\
    \end{matrix}
  \right) }}\\ \\
  & {\displaystyle{ = t^{||\beta||}\det
  \left(
    \begin{matrix}
      1 & \cdots & 1 \\
      \beta_1-1 & \cdots & \beta_{n-k}-1 \\
      \vdots & & \vdots \\
      \frac{(\beta_1-1) !}{(\beta_{1}-n+k)!} & \cdots & \frac{(\beta_{n-k}-1) !}{(\beta_{n-k}-n+k)!} \\
    \end{matrix}
  \right)\,, }}
  \end{array}}}
\]
where $||\beta||:=k(n-k)-|\beta|$ is the dimension of $X_\beta(t)\subset \Gr(n - k,\C^n)$.
So $(F_{n-k}(t))_{\beta}$ is $t^{||\beta||}$ times the Van der Monde determinant,
\[
  (F_{n-k}(t))_\beta = t^{||\beta||}\prod_{i<j}((\beta_j-1)-(\beta_i-1)) = t^{||\beta||}\prod_{i<j}(\beta_j-\beta_i)\,.
\]
Thus the determinant $(F_{n-k}(t))_{[i,\widehat{j}]^c}$ is
\[t^{n-k+j-i}1!2!\cdots (i-2)! \frac{i!}{1} \frac{(i+1)!}{2}\cdots \frac{(n-k-1)!}{n-k-i}\frac{(n-k+j)!}{(j-1)!}\frac{1}{n-k+j-i}\,,\]
which implies
\[
  M_{[i,\widehat{j}]} (F_{n-k}(t))_{[i,\widehat{j}]^c} = (-1)^n t^{n-k+j-i} \frac{1}{n-k+j-i} a_ib_{k-j}\,.
\]
Referring back to (\ref{eqn:LaPlace}), we see
\[{\displaystyle{
 \begin{array}{rl}
  \det
  \left(
    \begin{matrix}
      M \\
      F_{n-k}(t)
    \end{matrix}
  \right) = & {\displaystyle{
  (-1)^{k(n-k)} (-1)^{|[n-k]^c|} M_{[n-k]^c} (F_{n-k}(t))_{[n-k]} }} \\
  & {\displaystyle{ + (-1)^{k(n-k)}
  \sum_{i=1}^{n-k} \sum_{j=1}^{k} (-1)^{|[i,\widehat{j}]|} M_{[i,\widehat{j}]} (F_{n-k}(t))_{[i,\widehat{j}]^c}}} \\
  = & {\displaystyle{ b_k + (-1)^n \sum_{i=1}^{n-k} \sum_{j=1}^{k}}} (-t)^{n-k+j-i} \frac{1}{n-k+j-i} a_ib_{k-j} =: f(t)\,.
 \end{array}}}
\]
Taking the derivative of $f(t)$ yields
\[(-1)^n \sum_{i=1}^{n-k} \sum_{j=1}^{k} (-1)(-t)^{n-k+j-i-1} a_ib_{k-j}\,,
\]
which may be factored,
\begin{equation}\label{eqn:FeqAB}
 f'(t) = (-1)^n \left(\sum_{i=1}^{n-k} (-t)^{n-k-i} a_i\right)\left(\sum_{j=1}^{k} (-t)^{j-1}b_{k-j}\right) =: A(t)B(t)\,,
\end{equation}
so that $A(t)$ and $B(t)$ are uniquely defined monic polynomials with coefficients $(-1)^{i-1}a_i$ and $(-1)^ib_i$ respectively.
The coefficients $a_2,\dotsc,a_{n-k},b_1,\dotsc,b_k$ are coordinates of a solution to the instance $X$ of $\balpha$ from Equation (\ref{eqn:LBfamily}) if and only if $f(t)$ has roots at $t_2,\dotsc,t_{n}$.
There are no other roots, because $\#\{t_2,\dotsc,t_{n}\}=k(n-k)-|\omat|=n-1=\deg f(t)$.
A solution to an instance of $X$ corresponds to a real polynomial $f(t)$ with $r_{\Is}$ real roots.
Applying Rolle's Theorem, we see that $f'(t)$ has at least $r_{\Is}-1$ real roots.

The polynomials $A(t)$ and $B(t)$ have all coefficients real if and only if $a_i,b_j\in \R$ for all $i,j$ ($b_k$ is real because it is an integer multiple of $f(0)$).
This implies the following.
\begin{theorem}\label{thm:factorLB}
 Let $X$ be the real osculating instance of $\balpha$ from Equation (\ref{eqn:LBfamily}).
 The number of real points in $X$ is equal to the number of factorizations $f'(t)=A(t)B(t)$ from Equation (\ref{eqn:FeqAB}), such that $A(t),B(t)$ are monic real polynomials of degree $n-k-1,k-1$ respectively.
\end{theorem}
Applying the action of $\GL(2,\R)$ to the osculation points induces a real action on the solutions, so our discussion involving $t_1=\infty$ is general, and we have proven Theorem \ref{thm:factorLB}.
Soprunova and Sottile \cite{SS2006} discovered the use of an auxiliary factorization problem to rule out possible numbers of real solutions to geometric problems.

The polynomial $f'(t)$ in Theorem \ref{thm:factorLB} has degree $n-2$, and it has at least $r_{\Is}-1$ real roots, where $(1,r_{\Is})$ is the osculation type of $X$.
Increasing the number of real roots of a polynomial of fixed degree cannot decrease its number of real factorizations, so we have a lower bound on the number of real solutions to $X$.
\begin{corollary}\label{cor:factorLB}
 Let $X$ be the real osculating instance of $\balpha$ from Equation (\ref{eqn:LBfamily}).
 If $X$ has osculation type $(1,r_{\Is})$, then the number of real points in $X$ is at least the number of factorizations of a monic real polynomial $g(t)=a(t)b(t)$ of degree $n-2$ with $r_{\Is}-1$ real roots, such that $a(t),b(t)$ are monic real polynomials of degree $n-k-1,k-1$ respectively.
 
 Furthermore, if $k = 2p+1$ is odd and $n = 2p+2q+2$ is even, then the number of real points in $X$ is at least $\tbinom{p+q}{p}$, regardless of osculation type.
\end{corollary}
\begin{proof}
 We have already proven the first statement.
 For the second statement, we observe that if $f'(t)$ has no real roots, then it is a product of $p+q$ complex conjugate pairs of linear factors.
 The factorization $f'(t)=A(t)B(t)$ from Equation (\ref{eqn:FeqAB}) is real if and only if $B(t)$ is the product of $p$ complex conjugate pairs of linear factors.
 This gives the stated lower bound.
\end{proof}
If $k$ is even or $n$ is odd, then there may be no real factorizations of $f'(t)$, which implies the trivial lower bound on the number of real points in $X$.
\begin{example}
 Consider the Schubert problem $(\IIIIIt\,,\It^7)$ in $\Gr(2,\C^8)$, given in Table \ref{tab:staircase}.
 The lower bounds in the table are given by counting factorizations of a monic real degree-six polynomial $f'(t)$ into a monic real degree-five polynomial $A(t)$ and a monic real degree-one polynomial $B(t)$.
 Since $f'(t)$ has at least $r_{\Is}-1$ real factors, there are $\tbinom{r_{\Iss}-1}{1}=r_{\Is}-1$ ways to factor $f'(t)=A(t)B(t)$ with $A,B$ real.
\end{example}
\begin{example}
 Consider the Schubert problem $({\IIIiiiIIIt}\,,{\It}^7)$ in $\Gr(4,\C^8)$, given in Table \ref{tab:gap12}.
 The observed lower bounds are obtained by counting real factorizations of the degree six polynomial $f'$ into two monic degree three polynomials $A,B$.
 If $r_{\Is}=1$, then $f'$ may have no real factors and thus no real cubic factor $A$, so the lower bound is zero.

 If $r_{\Is}=3$, then $f'$ has at least two real factors $w,x$ and two pairs of complex factors $(y,\overline{y})$ and $(z,\overline{z})$.
 So $f'$ has real factorizations given by $A=wy\overline{y},wz\overline{z},xy\overline{y},xz\overline{z}$, so the lower bound is four.

 Similar arguments show that $r_{\Is}=5$ or $7$ impose lower bounds $8$ and $20$.
\end{example}
\begin{table}[tp]
 \caption{Irregular gaps.}											
  \begin{center}
 \begin{tabular}{|c||c|c|c|c||c|}											
  \hline											
  {\bf $\#$ Real} & \multicolumn{4}{c||}{\rule{0pt}{12pt}$\IIIiiiIIIt \quad \It^7$}               & \multirow{2}{*}{\bf Total} \\ \cline{2-5}
  {\bf Solutions} & {\bf 	$r_{\Is}=7$	} & {\bf 	$r_{\Is}=5$	} & {\bf 	$r_{\Is}=3$	} & {\bf 	$r_{\Is}=1$	}	&		\\ \hline \hline
{\bf 	0	} &	 	&	 	&	 	&	37074	& {\bf 	37074	} \\ \hline
{\bf 	2	} &	 	&	 	&	 	&	 	& {\bf 	0	} \\ \hline
{\bf 	4	} &	 	&	 	&	66825	&	47271	& {\bf 	114096	} \\ \hline
{\bf 	6	} &	 	&	 	&	 	&	 	& {\bf 	0	} \\ \hline
{\bf 	8	} &	 	&	85080	&	30232	&	14517	& {\bf 	129829	} \\ \hline
{\bf 	10	} &	 	&	 	&	 	&	 	& {\bf 	0	} \\ \hline
{\bf 	12	} &	 	&	 	&	 	&	 	& {\bf 	0	} \\ \hline
{\bf 	14	} &	 	&	 	&	 	&	 	& {\bf 	0	} \\ \hline
{\bf 	16	} &	 	&	 	&	 	&	 	& {\bf 	0	} \\ \hline
{\bf 	18	} &	 	&	 	&	 	&	 	& {\bf 	0	} \\ \hline
{\bf 	20	} &	100000	&	14920	&	2943	&	1138	& {\bf 	119001	} \\ \hline
{\bf Total } & {\bf 100000} & {\bf 100000} & {\bf 100000} & {\bf 100000} & {\bf 400000} \\ \hline
 \end{tabular}
  \end{center}
 \label{tab:gap12}
\end{table}
\begin{example}
 Corollary \ref{cor:factorLB} asserts that every real osculating instance of $({\IIiit}\,,{\It}^5)$ in $\Gr(3,\C^6)$ has at least two real solutions, and every real osculating instance of $({\IIIIiiiit}\,,{\It}^7)$ in $\Gr(3,\C^8)$ has at least three real solutions.
 Indeed, we observe this in Tables \ref{tab:22.1e5} and \ref{tab:44.1e7} respectively.
\end{example}
\begin{table}[tp]
 \caption{Nontrivial lower bound.}
  \begin{center}
 \begin{tabular}{|c||c|c|c||c|}											
  \hline											
  {\bf $\#$ Real} & \multicolumn{3}{c||}{\rule{0pt}{12pt}$\IIiit \quad \It^5$} & \multirow{2}{*}{\bf Total} \\ \cline{2-4}
  {\bf Solutions} & {\bf 	$r_{\Is}=5$	} & {\bf 	$r_{\Is}=3$	} & {\bf 	$r_{\Is}=1$	}	&		\\ \hline \hline
{\bf	 0    } &              &              &              & {\bf 0}      \\ \hline
{\bf	 2    } &              & 64775        & 87783        & {\bf 152558} \\ \hline
{\bf	 4    } &              &              &              & {\bf 0}      \\ \hline
{\bf	 6    } & 100000       & 35225        & 12217	    & {\bf 147442} \\ \hline \hline
{\bf	 Total} & {\bf 100000} & {\bf 100000} & {\bf 100000} & {\bf 300000} \\ \hline
 \end{tabular}
  \end{center}
 \label{tab:22.1e5}
\end{table}
\begin{table}[tp]
 \caption{Another nontrivial lower bound.}											
  \begin{center}
 \begin{tabular}{|c||c|c|c|c||c|}											
  \hline											
  {\bf $\#$ Real} & \multicolumn{4}{c||}{\rule{0pt}{12pt}$\IIIIiiiit \quad \It^7$}               & \multirow{2}{*}{\bf Total} \\ \cline{2-5}
  {\bf Solutions} & {\bf 	$r_{\Is}=7$	} & {\bf 	$r_{\Is}=5$	} & {\bf 	$r_{\Is}=3$	} & {\bf 	$r_{\Is}=1$	}	&		\\ \hline \hline
{\bf	1	} &	 	&	 	&	 	&	 	& {\bf	0	} \\ \hline
{\bf	3	} &	 	&	 	&	47274	&	76702	& {\bf	123976	} \\ \hline
{\bf	5	} &	 	&	 	&	 	&	 	& {\bf	0	} \\ \hline
{\bf	7	} &	 	&	77116	&	46912	&	21909	& {\bf	145937	} \\ \hline
{\bf	9	} &	 	&	 	&	 	&	 	& {\bf	0	} \\ \hline
{\bf	11	} &	 	&	 	&	 	&	 	& {\bf	0	} \\ \hline
{\bf	13	} &	 	&	 	&	 	&	 	& {\bf	0	} \\ \hline
{\bf	15	} &	100000	&	22884	&	5814	&	1389	& {\bf	130087	} \\ \hline \hline
{\bf	Total	} & {\bf	100000	} & {\bf	100000	} & {\bf	100000	} & {\bf	100000	} & {\bf	400000	} \\ \hline
 \end{tabular}
  \end{center}
 \label{tab:44.1e7}
\end{table}

\section{Gaps}

Several of the Schubert problems we studied had unexpected gaps in the possible numbers of real solutions.
One may see this in Tables \ref{tab:gap12}--\ref{tab:44.1e7}.
That we never have $12$ or $16$ solutions in Table \ref{tab:gap12} is particularly unexpected, as this problem satisfies a congruence modulo four, and $12\equiv 16\equiv 20 \mod 4$.
The gaps in all three of these tables may be fully explained by Theorem \ref{thm:factorLB}.
Proposition \ref{prop:mod4tech} from Chapter \ref{chapMod4} gives an alternative explanation for the congruence modulo four observed in Tables \ref{tab:gap12} and \ref{tab:22.1e5} (but not for the congruence in Table \ref{tab:44.1e7}).

A solution to an instance of the Schubert problem given in Table \ref{tab:gap12} is real if and only if its coordinates are given by a real factorization $f'(t)=A(t)B(t)$ as in Theorem \ref{thm:factorLB}.
Since the number $r$ of real factors of $f'(t)$ is at least $r_{\Is}-1$, we have $r=r_{\Is}-1,r_{\Is}+1,\dotsc,6$.
For each of these $r$-values, the number of real solutions to $X$ is exactly the lower bound of Corollary \ref{cor:factorLB} associated to the osculation type $(1,r)$.
Thus the set of lower bounds given by Corollary \ref{cor:factorLB} is the set of possible numbers of real points in $X$, given by Theorem \ref{thm:factorLB}.

Similar analysis explains the gaps found in Tables \ref{tab:22.1e5} and \ref{tab:44.1e7}.
We give an example in a Grassmannian of higher dimension.
\begin{example}
 Consider the Schubert condition $\omat$ for $\Gr(5,10)$.
 The lower bounds of Corollary \ref{cor:factorLB} corresponding to $\balpha = (\omat,\It^{n-1})$ are $6,6,14,30$, or $70$.
 Thus any real osculating instance of $\balpha$ has exactly $6,14,30$, or $70$ real solutions.
 The lower bound of $6$ which is independent of osculation type is an example of the nontrivial lower bound given by Corollary \ref{cor:factorLB}.
 This lower bound is the topological lower bound $\Sigma(\omat,\It)$.
\end{example}

Let
\[ 
  \begin{picture}(46.5,46.5)
   \put (0,0){\includegraphics{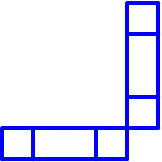}}
   \put (11.5,2.1){$\cdots$}
   \put (39.5,22.5){$\vdots$}
  \end{picture}\ \,\raisebox{24pt}{$ =: \lambda(p,q)$}
\]
be the skew-diagram with $p$ boxes in the rightmost column and $q$ boxes in the bottom row.
\begin{proposition}\label{prop:LBareEQ}
 Let $p=k-1$ and $q=n-k-1$.
 If $p\geq 2$ or $q\geq 2$, then
 \[\sum_{T\in\SYT(\lambda(p,q))}\sign(T) = \sum_{T\in\SYT(\lambda(p,q-2))}\sign(T) + \sum_{T\in\SYT(\lambda(p-2,q))}\sign(T)\,.\]
\end{proposition}
\begin{proof}
 We first observe that both sides of the equation are zero unless $p$ and $q$ are even.
 Thus we only need to prove one base case $p=q=2$ to use induction.
 For the base case, we observe
 \[\sum_{T\in\SYT(\lambda(2,0))}\sign(T) + \sum_{T\in\SYT(\lambda(0,2))}\sign(T) = 1 + 1 = 2\,.\]
 To calculate $\sum_{T\in\SYT(\lambda(2,2))}\sign(T)$, we compare the standard Young tableaux of the shape $\IIIiiiIIminusIIiit$, of which there are six.
 \[
  \begin{picture}(386,36)
   \put (0,0){\IIIiiiIIminusIIiib}
   \put (70,0){\IIIiiiIIminusIIiib}
   \put (140,0){\IIIiiiIIminusIIiib}
   \put (210,0){\IIIiiiIIminusIIiib}
   \put (280,0){\IIIiiiIIminusIIiib}
   \put (350,0){\IIIiiiIIminusIIiib}

   \put (3.5,2.5){3}
   \put (15,2.5){4}
   \put (26.5,14){2}
   \put (26.5,25.5){1}

   \put (73.5,2.5){2}
   \put (85,2.5){4}
   \put (96.5,14){3}
   \put (96.5,25.5){1}

   \put (143.5,2.5){2}
   \put (155,2.5){3}
   \put (166.5,14){4}
   \put (166.5,25.5){1}

   \put (213.5,2.5){1}
   \put (225,2.5){4}
   \put (236.5,14){3}
   \put (236.5,25.5){2}

   \put (283.5,2.5){1}
   \put (295,2.5){3}
   \put (306.5,14){4}
   \put (306.5,25.5){2}

   \put (353.5,2.5){1}
   \put (365,2.5){2}
   \put (376.5,14){4}
   \put (376.5,25.5){3}

  \end{picture}
 \]
 The first tableau pictured above has sign $+1$ by definition, and the others have signs $-1$, $+1$, $+1$, $-1$, and $+1$ respectively.
 These sum to $+2$, proving the base case.
 
 For the inductive step, we denote the two largest numbers appearing in a standard Young tableau of shape $\lambda(p,q)$ by $\mbox{Y}:=p+q-1$ and $\mbox{Z}:=p+q$.
 Observe that $\mbox{Z}$ occurs in the last box of the single row of $\lambda(p,q)$ or in the last box of the single column of $\lambda(p,q)$.
 Similarly, $\mbox{Y}$ occurs in one of the last two boxes of either the single row or the single column.
 We draw the four possible configurations of the numbers $\mbox{Y},\mbox{Z}$ is a standard Young tableau of shape $\lambda(p,q)$.
 \[
  \begin{picture}(256,36)
   \put (0,0){\IIIiiiIIminusIIiib}
   \put (70,0){\IIIiiiIIminusIIiib}
   \put (140,0){\IIIiiiIIminusIIiib}
   \put (210,0){\IIIiiiIIminusIIiib}

   \put (3.5,2.5){}
   \put (14.3,2.5){Z}
   \put (25,14){Y}
   \put (26.5,25.5){}

   \put (73.5,2.5){}
   \put (83.5,2.5){Y}
   \put (95.8,14){Z}
   \put (96.5,25.5){}

   \put (142,2.5){Y}
   \put (154.3,2.5){Z}
   \put (166.5,14){}
   \put (166.5,25.5){}

   \put (213.5,2.5){}
   \put (225,2.5){}
   \put (235.8,14){Z}
   \put (235,25.5){Y}

  \end{picture}
 \]

 Let $\mathcal{T}_1\subset\SYT(\lambda(p,q))$ be the set of standard Young tableaux of the first type ($\mbox{Y}$ appears in the single column, and $\mbox{Z}$ appears in the single row).
 Similarly, let $\mathcal{T}_2,\mathcal{T}_3,\mathcal{T}_4\subset\SYT(\lambda)$ denote the sets of tableaux of the second, third, and fourth types respectively.
 Since the tableaux of $\mathcal{T}_2$ are obtained by applying the transposition $\mbox{Y} \leftrightarrow \mbox{Z}$ to $\mathcal{T}_1$, we have
 \[\sum_{T\in\mathcal{T}_1}\sign(T) + \sum_{T\in\mathcal{T}_2}\sign(T) = 0\,.\]
 Therefore, we need only consider the parity of tableaux of the third and fourth types,
 \[\sum_{T\in \SYT(\lambda(p,q))} \sign(T) = \sum_{T\in \mathcal{T}_3} \sign(T) + \sum_{T\in \mathcal{T}_4} \sign(T)\,.\]
 Deleting the columns of tableaux in $\mathcal{T}_3$ which contain $\mbox{Y},\mbox{Z}$ gives a bijection
 \[\pi_3:\SYT(\lambda(p,q))\rightarrow\SYT(\lambda(p,q-2))\,,\]
 which one immediately verifies has the property $\sign(\pi_3(T)) = \sign(T)$ for $T\in \SYT(\lambda(p,q))$.
 Therefore,
 \[\sum_{T\in \SYT(\lambda(p,q))} \sign(T) = \sum_{T\in \SYT(\lambda(p,q-2))} \sign(T) + \sum_{T\in \mathcal{T}_4} \sign(T)\,.\]
 
 Deleting the rows of $T\in \mathcal{T}_4$ which contain $\mbox{Y},\mbox{Z}$ gives a bijection $\pi_4:\SYT(\lambda(p,q))\rightarrow\SYT(\lambda(p-2,q))$.
 The standard filling of $\lambda(p,q)$ has the same parity as the tableaux given by assigning $(1,2,\dotsc,p-2,\mbox{Y},\mbox{Z})$ to the column and $(p-1,p,\dotsc,p+q-1)$ to the row.
 Thus, $\sign(\pi_4(T)) = \sign(T)$ for $T\in \SYT(\lambda(p,q))$, and we have our result.
\end{proof}
The skew diagram $\lambda(p,q)$ as a product of two chains, one of length $p$ and one of length $q$.
Sottile and Soprunova studied products of chains and showed a connection between lower bounds obtained by factorization and topological lower bounds \cite{SS2006}.
The following is a corollary to Proposition \ref{prop:LBareEQ}.
\begin{theorem}
 Suppose $k=2p+1$ is odd, $n=2p+2q+2$ is even, and $X$ is the real osculating instance of $\balpha$ from Equation (\ref{eqn:LBfamily}).
 The lower bound $\tbinom{p+q}{p}$ on the number of real points in $X$ coincides with the topological lower bound $\Sigma(\omat,\It)$.
\end{theorem}

There are Schubert problems $\balpha$ not of the form $(\omat,\It^{n-1})$ whose frequency tables exhibit gaps, sometimes apparently due to unexpected upper bounds on the numbers of real solutions for instances of $\balpha$ with certain osculation types.
Table \ref{tab:upperGap} reporting on the Schubert problem $(\IIIiIt\,,\IIiit\,,\IIit\,,\IIit\,,\It)$ in $\Gr(4,\C^8)$ exhibits such behavior.

There are also problems with no gaps which have apparent upper bounds lower than the number of complex solutions for certain osculation types.
Table \ref{tab:upperLower} exhibits remarkable upper bounds for $(\IIit\,,\IIit\,,\It\,,\It\,,\It)$ in $\Gr(3,\C^6)$.
Unexpected upper bounds were far less common than nontrivial lower bounds in the problems we tested.

\begin{table}[tp]
 \caption{Unexpected upper bound.}
  \begin{center}
 \begin{tabular}{|c||c|c||c|}											
  \hline											
  {\bf $\#$ Real} & \multicolumn{2}{c||}{\rule{0pt}{12pt}$\IIIiIt \quad \IIiit \quad \IIit^2 \quad \It$} & \multirow{2}{*}{\bf Total} \\ \cline{2-3}
  {\bf Solutions} 		& $r_{\IIis}=2$ & $r_{\IIis}=0$ &		\\ \hline \hline
{\bf 	0	} & 	 	&	148450	& {\bf 148450} \\ \hline
{\bf 	2	} & 	 	&	64662	&	 {\bf 64662}	\\ \hline
{\bf 	4	} & 	 	&	99465	&	 {\bf 99465}\\ \hline
{\bf 	6	} & 	 	&	59	& {\bf 59}	\\ \hline
{\bf 	8	} & 	 	&	87364	& {\bf 87364}	\\ \hline
{\bf 	10	} & 	 	&	 	& {\bf 0}	\\ \hline
{\bf 	12	} & 	 	&	 	& {\bf 0}	\\ \hline
{\bf 	14	} & 	 	&	 	& {\bf 0}\\ \hline
{\bf 	16	} & 	400000	&	 	&	 {\bf 400000} \\ \hline \hline
{\bf 	Total	} & {\bf 	400000	} & {\bf 	400000	} & {\bf 800000}\\ \hline
 \end{tabular}											
  \end{center}
 \label{tab:upperGap}
\end{table}
\begin{table}[tp]
 \caption{Unexpected lower bounds and upper bounds.}
  \begin{center}
 \begin{tabular}{|c||c|c|c|c||c|}											
  \hline											
  \multirow{2}{*}{\bf $\#$ Real} & \multicolumn{4}{c||}{\rule{0pt}{12pt}$\IIit^2 \quad \It^3$} &  \\ \cline{2-5}
  \multirow{2}{*}{\bf Solutions} 		& $r_{\IIis}=2$& $r_{\IIis}=2$& $r_{\IIis}=0$& $r_{\IIis}=0$ &  {\bf Total} \\ \cline{2-5}
 		& $r_{\Is}=3$	& $r_{\Is}=1$	& $r_{\Is}=3$	& $r_{\Is}=1$ &  \\ \hline \hline
{\bf 	0	} &	 	&	27855	&	17424	&	 	& {\bf 	45279	} \\ \hline
{\bf 	2	} &	 	&	11739	&	82576	&	100000	& {\bf 	194315	} \\ \hline
{\bf 	4	} &	 	&	22935	&	 	&	 	& {\bf 	22935	} \\ \hline
{\bf 	6	} &	100000	&	37471	&	 	&	 	& {\bf 	137471	} \\ \hline \hline
{\bf 	Total	} & {\bf 100000 } & {\bf 100000 } & {\bf 100000 } & {\bf 100000 } & {\bf 	400000	} \\ \hline
 \end{tabular}											
  \end{center}
 \label{tab:upperLower}
\end{table}

\chapter{\uppercase{A Congruence Modulo Four}}
  \label{chapMod4}
  Our report in Chapter \ref{chapLower} includes Schubert problems for which the number of real solutions is fixed modulo four.
Table \ref{tab:sharpness2} representing the Schubert problem $(\It^9)$ in $\Gr(3,\C^6)$ and Table \ref{tab:gap12} representing $(\IIIiiiIIIt,\It^7)$ in $\Gr(4,\C^8)$ exhibit this congruence.
We observed this phenomenon in several other problems sharing the property that each defining Schubert condition $\alpha$ satisfies $\alpha = \alpha^\perp$.
We prove this congruence modulo four and thereby find new invariants in enumerative real algebraic geometry.
The proof uses a geometric involution that fixes Schubert varieties $X_\alpha (t)\subset \Gr(k,2k)$ with $\alpha = \alpha^\perp$.
This chapter follows joint work with Sottile and Zelenko \cite{mod4}.

\section{The Lagrangian Grassmannian}

We retain the notation of Chapters \ref{chapSchubert} and \ref{chapLower}.
Throughout this chapter, $n=2k$.
We denote the real points of a variety $X$ by $X(\R)$.

Let $J$ be a non-degenerate skew-symmetric $2k\times 2k$ matrix with determinant $1$.
The matrix $J$ gives an isomorphism $J:V^*\rightarrow V$ defined by $v\mapsto (Jv)^T$.
The \emph{symplectic group} $\Sp(V)_{J}$ is the set of all elements $h$ of $\SL(V)$ which satisfy $J = hJh^T$.
Let $\langle\cdot\,,\cdot\rangle_{J}$ denote a nondegenerate alternating form on $V$, called a \emph{symplectic form},
\[\langle u,v\rangle_{J} := uJv^T \quad \mbox{for} \quad u,v\in V\,.\]

Given an $l$-plane $H\in\Gr(l,\C^{2k})$, let $H^\angle\in \Gr(2k-l,V)$ denote its \emph{skew-orthogonal complement (with respect to $J$)} in $V$,
\[H^\angle := JH^\perp\,.\]
Since $\langle\cdot\,,\cdot\rangle$ is nondegenerate, $\dim(H)+\dim(H^\angle)=2k$ and $(H^\angle)^\angle=H$ for any linear subspace $H\subset V$.
We call a flag $\Fdot$ in $V$ \emph{isotropic (with respect to $J$)} if for $F_i^\angle = F_{2k-i}^\angle$ for $i < 2k$.

Since $n=2k$, there is an involution
\[\angle:\Gr(k,V)\longrightarrow \Gr(k,V)\,,\]
given by $H\mapsto H^\angle$, i.e.\ $\angle = J \circ {\perp}$.
A $k$-plane $H\in\Gr(k,V)$ is \emph{Lagrangian (with respect to $J$)} if $H=H^\angle$.
We note that if $\Fdot$ is an isotropic flag in $V$, then $F_k$ is Lagrangian.

Consider skew-symmetric $2k\times 2k$ matrix
\[\widetilde{J} =
  \left(
    \begin{matrix}
      0      & \cdots  & 0      & 0     &         & -1 \\
      \vdots &         & \vdots &       & \iddots &     \\
      0      & \cdots  & 0      & -1    &         & 0    \\
      0      &         & 1      & 0     & \cdots  & 0     \\
             & \iddots &        &\vdots &         & \vdots \\
      1      &         & 0      & 0     & \cdots  & 0       \\
    \end{matrix}
  \right)
\]
with determinant $1$ and the real parametrized rational normal curve
\begin{equation}\label{eqn:normalGamma}
 \widetilde{\gamma}(t)=\left(1,t,\frac{t^2}{2!},\dotsc,\frac{t^k}{k!},\frac{(-1)^1t^{k+1}}{(k+1)!},\dotsc,\frac{(-1)^{k-1}t^{2k-1}}{(2k-1)!}\right)\,.
\end{equation}
The flag $\Fdot(t)$ osculating $\widetilde{\gamma}$ at $\widetilde{\gamma}(t)$ has basis $(\widetilde{\gamma}(t),\widetilde{\gamma}'(t),\dotsc,\widetilde{\gamma}^{(2k-1)})$.
A calculation using Equation (\ref{eqn:normalGamma}) shows that if $i<j$, then 
\[\langle \widetilde{\gamma}^{(i-1)}(t) , \widetilde{\gamma}^{(j-1)}(t)\rangle_{\widetilde{J}} = \sum_{l=0}^{2k+1-i-j} \frac{(-1)^{i+j+l}}{(2k+1-i-j)!}\tbinom{2k+1-i-j}{l}\,.\]
If $i+j\neq 2k+1$ then the dot product is zero.
Thus $\Fdot(t)$ is isotropic with respect to $\widetilde{J}$.

One may obtain any parametrized rational normal curve $\gamma(t)$ by applying the right action of $g\in\SL(V)$ to $\widetilde{\gamma}(t)$, given by $\gamma(t) = \widetilde{\gamma}(t)g$.
If $\gamma(t)$ is real, then $g$ may be chosen to be real.
The skew-symmetric matrix $J:=g^{-1}\widetilde{J}(g^{-1})^T$ with determinant $1$ gives an isomorphism $J:V^*\rightarrow V$.
All symplectic groups are conjugate by this action of $\SL(V)$.

Henceforth, we fix a real $g\in \SL(V)$, thus fixing a real curve $\gamma(t)$ and a real matrix $J$ which identifies $V^*$ with $V$.
Thus we omit subscripts of $J$, writing $\langle\cdot\,,\cdot\rangle$ and $\Sp(V)$.
To facilitate proofs, we may use the real action of $\SL(V)$ to give a particular curve $\gamma(t)$ and a corresponding matrix $J$.
\begin{proposition}\label{prop:isotropicOsculatingFlags}
 Osculating flags are isotropic.
\end{proposition}
\begin{proof}
 Flags osculating the rational normal curve $\gamma(t)$ are isotropic with respect to $J$, because
 \[\langle ug , vg \rangle = ugJ(vg)^T = u\widetilde{J}v^T = \langle u,v \rangle_{\widetilde{J}}.\]
 The rest follows from our discussion above.
\end{proof}

For $X\subset \Gr(k,V)$, let $X_\angle$ denote the Lagrangian points of $X$, that is, the points fixed by $\angle$.
The \emph{Lagrangian Grassmannian $\LG(V)$} is the subset of the Grassmannian consisting of Lagrangian $k$-planes,
\[\LG(V):=\Gr(k,V)_\angle = \left\{H\in\Gr\left(k,V\right)\ |\ H=H^\angle \right\}\,.\]
We have observed that $\Gr(k,V)$ is a homogeneous space for $\GL(V)$.
Since scaling generators does not affect their span, $\Gr(k,V)$ is a homogeneous space for $\SL(V)$ as well.
The Lagrangian Grassmannian is a homogeneous space of $\Sp(V)$.

\begin{proposition}\label{prop:LGsubvar}
 The Lagrangian Grassmannian $\LG(V)\subset \Gr(k,V)$ is a subvariety of dimension $\tbinom{k+1}{2}$.
\end{proposition}
\begin{proof}
Without loss of generality, suppose
\begin{equation}\label{def:sympMat}
 J = 
  \left(\begin{matrix}
   0      & -\Id_k \\
   \Id_k & 0     \\
  \end{matrix}\right)\,.
\end{equation}
Recall definition (\ref{def:Galpha}) of the open cover $\mathcal G$ of the Grassmannian.
For any $\alpha\in\tbinom{[2k]}{k}$, the matrices of $\s(\alpha)$ give coordinates for the dense open set $G_\alpha\subset \Gr(k,V)$.
A matrix $M\in\s(\alpha)$ gives coordinates for a point in $\LG(V)$ if an only if for any two rows $u,v$ of $M$ we have $\langle u,v\rangle = 0$.
This gives $k^2$ polynomial equations which establish $\LG(V)\subset \Gr(k,V)$ as a subvariety.

One may choose $\alpha$ so that all of the equations given above are linear ($\alpha=[k]$, for example).
Since $\langle u,v\rangle = - \langle v , u \rangle$ there are $\tbinom{k}{2}$ linearly independent equations defining $\LG(V)$.
Since $\dim(\s(\alpha))=k^2$, we have $\dim(\LG(V)) = k^2-\tbinom{k}{2} = \tbinom{k+1}{2}$.
%
\end{proof}
\begin{remark}
 One may choose a standard basis of $V$ in such a way that the $k\times 2k$ matrix $[\Id_k|M]$ of parameters with $M$ symmetric give local coordinates for $\LG(V)$.
\end{remark}

Recall the definition of $\alpha^\perp$, and that the rows of $d(\alpha)$ are the columns of $d(\alpha^\perp)$.
Noting that $\angle = J\circ {\perp}$, Proposition \ref{prop:dualSV} implies the following.
\begin{proposition}\label{prop:skewDualSV}
 Suppose $\Fdot$ is isotropic and $\alpha\in\tbinom{[2k]}{k}$.
 Then $\angle(X_\alpha \Fdot) = X_{\alpha^\perp} \Fdot$.
\end{proposition}
\begin{definition}
 The Schubert condition $\alpha\in\tbinom{[2k]}{k}$ is \emph{symmetric} if $\alpha=\alpha^\perp$.
 A Schubert problem $\balpha$ is \emph{symmetric} if each Schubert condition in $\balpha$ is symmetric.
\end{definition}
For $\Gr(3,\C^6)$, we give diagrams of some symmetric Schubert conditions
 \[d(3,5,6) = \I\,,\qquad d(2,4,6) = \IIi\,,\qquad d(2,3,6) = \IIii\,,\]
and some non-symmetric Schubert conditions
 \[d(1,5,6) = \III \neq \IiI\,,\qquad d(2,4,5) = \IIiI \neq \IIIi\,.\]
We give the key to proving the main theorems of this chapter.
\begin{corollary}\label{cor:SPfixedbyLag}
 If $\balpha$ is a symmetric Schubert problem, and $X$ is an osculating instance of $\balpha$, then $X$ is stable under the Lagrangian involution, $X^\angle=X$.
\end{corollary}
\begin{proof}
 Proposition \ref{prop:isotropicOsculatingFlags} asserts that the flags giving the instance $X$ are isotropic.
 Thus Proposition \ref{prop:skewDualSV} establishes $X$ as an intersection of Schubert varieties which are stable under $\angle$.
 Therefore, $X$ is stable under $\angle$.
\end{proof}
Proposition \ref{prop:skewDualSV} allows us to define Schubert varieties for the Lagrangian Grassmannian.
\begin{definition}
 Suppose $\Fdot$ is isotropic and $\alpha\in\tbinom{[2k]}{k}$ is symmetric.
 Then $Y_\alpha \Fdot:= X_\alpha \Fdot \cap \LG(V)$ is a \emph{Lagrangian Schubert variety}.
\end{definition}
The \emph{length} $\ell(\alpha)$ of a Schubert condition $\alpha\in\tbinom{[2k]}{k}$ is the number of entries in $\alpha$ no greater than $k$,
\[\ell(\alpha):=\#\{\alpha_i \in \alpha\ |\ \alpha_i\leq k\}\,.\]
The length $\ell(\alpha)$ is the number of boxes in the main diagonal of the Young diagram $d(\alpha)$, and we may write $\ell(d(\alpha))$ to denote $\ell(\alpha)$.
We illustrate this by giving Young diagrams with their main diagonals shaded:
\[\ell\left(\lenIIi\right) = 1\,,\qquad \ell\left(\lenIIii\right)=2\,,\qquad \ell\left(\lenIIIIiiiI\right)=2\,.\]
\begin{proposition}\label{prop:LGSVdim}
 Let $\alpha$ be a symmetric Schubert condition and $\Fdot$ an isotropic flag.
 The codimension of $Y_\alpha \Fdot$ in $\LG(V)$ is
 \[\| \alpha \| := \frac{|\alpha| + \ell(\alpha)}{2}\,.\]
\end{proposition}
\begin{proof}
 Without loss of generality, we use the symplectic form defined by
 \begin{equation}\label{def:sympMat}
  J = 
  \left(
   \begin{matrix}
    0      & -\Id_k \\
    \Id_k & 0     \\
   \end{matrix}
  \right)\,,
 \end{equation}
 and a corresponding isotropic flag $\Fdot$ which makes the codimension of $Y_\alpha \Fdot$ apparent.
 Consider the path $p(\alpha)$ defined in Proposition \ref{prop:perpSameCodim} (the northeast to southwest path defining the lower border of the Young diagram $d(\alpha)$).
 We label the vertical edges of $p(\alpha)$ with elements of the standard basis $e_1,\dotsc,e_k$ from top to bottom, and we label the horizontal edges with $e_{k+1},\dotsc,e_{2k}$ from left to right.
 Reading the labels along the path $p(\alpha)$ gives a basis $\bff$ for a flag $\Fdot$.

 As an example, we draw the labeled path associated to the symmetric Schubert condition $(2,5,6,8)\in\tbinom{[8]}{4}$.
 \[
  \begin{picture}(93,82)
   \put (11,0){\IIIiIp}

   \put (0,8){$e_4$}
   \put (17.5,13.5){$e_5$}
   \put (20.5,28){$e_3$}
   \put (20.5,48){$e_2$}
   \put (37.5,54){$e_6$}
   \put (58,54){$e_7$}
   \put (61,68){$e_1$}
   \put (78.5,74){$e_8$}

  \end{picture}
 \]
 Since $\alpha$ is symmetric, the $i$th vertical edge counting from the top is transposed with the $i$th horizontal edge counting from the left, so reflecting the path along the antidiagonal transposes $e_i$ with $e_{i+k}$ for $i\in[k]$.
 Equivalently, if $f_i=e_j$ and $f_{2k-i+1}=e_l$ then $|j-l|=k$.
 Thus $\langle f_i\,,f_j\rangle = \delta_{j,2k-i+1}$, and $\Fdot$ is isotropic.

 Furthermore, $Y_\alpha \Fdot$ has local coordinates like the Stiefel coordinates given by the matrix $[\Id_k|M]$ of parameters with $M$ symmetric, such that the entries of $M$ satisfy the equations
 \[m_{ij} = 0\mbox{ if } j\leq d(\alpha)_i\,,\]
 where $d(\alpha)_i$ is the number of boxes in the $i$th row of the diagram $d(\alpha)$.
 Since $M$ is symmetric, we have $M_{ij}=M_{ji}$ for all $i$ and $j$.
 A calculation shows that these Stiefel-like coordinates have $\tbinom{k+1}{2}-\frac{|\alpha| + \ell(\alpha)}{2}$ independent parameters.
 Since $\dim(LG(V))=\tbinom{k+1}{2}$, we have $\| \alpha \| = \frac{|\alpha| + \ell(\alpha)}{2}$.
\end{proof}
We illustrate the coordinates defined in the proof of Proposition \ref{prop:LGSVdim}.
The Schubert condition $\alpha=(2,5,6,8)$ has Young diagram
\[d(\alpha) = \IIIiI\,,\]
and the Lagrangian Schubert variety $Y_\alpha \Fdot$ has local coordinates
\[
  \left(
   \begin{matrix}
     1 & 0 & 0 & 0 & \I     & \I     & \I     & m_{14} \\
     0 & 1 & 0 & 0 & \I     & m_{22} & m_{23} & m_{24} \\
     0 & 0 & 1 & 0 & \I     & m_{23} & m_{33} & m_{34} \\
     0 & 0 & 0 & 1 & m_{14} & m_{24} & m_{34} & m_{44} \\
   \end{matrix}
  \right)\,,
\]
where ${\It}$ denotes a coordinate in $\s(1,2,3,4)$ on $\LG(V)$ which is identically zero on $Y_\alpha\Fdot$.
 With these coordinates, one may see that $\dim(Y_\alpha \Fdot) = 7$ and $\| \alpha \| = 3$.
 In general, $\|\alpha\|$ is the number of boxes in and above the diagonal of $\d(\alpha)$.

Given a list $\balpha = (\alpha^1,\dotsc,\alpha^m)$ of symmetric Schubert conditions for $\LG(V)$, we define $\|\balpha\| := \| \alpha^1 \|+\cdots+\| \alpha^m \|$.
Kleiman's Theorem (Proposition \ref{prop:genTrans}) applies to Lagrangian Schubert problems \cite{Kleiman}.
\begin{proposition}[Lagrangian General Transversality]\label{prop:LGgenTrans}
 Let $\balpha = (\alpha^1,\dotsc,\alpha^m)$ be a list of symmetric Schubert conditions for $\LG(V)$.
 If $\Fdot^1,\dotsc,\Fdot^m$ are general isotropic flags, then the intersection
 \begin{equation}\label{eqn:LGGeneralIntersection}
  Y:=Y_{\alpha^1} \Fdot^1 \cap \cdots \cap Y_{\alpha^m} \Fdot^m
 \end{equation}
 in $\LG(V)$ is generically transverse.
 In particular, if $Y$ is nonempty, then $\codim (Y) = \|\balpha\|$.
\end{proposition}
If $\balpha$ is a list of symmetric Schubert problems with $\| \balpha \| = \tbinom{k+1}{2}$, then $\balpha$ is called a \emph{Lagrangian Schubert problem}.

\section{Congruence Modulo Four via Independent Involutions}

We find it useful to discuss sets of fixed points of the Grassmannian under different involutions.
The set of points fixed by complex conjugation in $\Gr(k,V)$ is the \emph{real Grassmannian}
\[\RGr(k,V) := \Gr(k,V)(\R) = \Gr(k,V(\R))\,.\]
Using the local coordinates of Proposition \ref{prop:LGSVdim}, $\RGr(k,V)$ is given by $k\times k$ matrices $M$ with the restriction that all entries are real.

Composing complex conjugation with $\angle$ gives another involution on $\Gr(k,V)$, 
and we call its set of fixed points the \emph{Hermitian Grassmannian},
\[\HG(V):=\left\lbrace H\in \Gr(k,V)\ |\ H=\overline{H}^\angle\right\rbrace\,.\]
We could alternatively define the Hermitian Grassmannian as the set of $k$-planes $H$ with $H=\overline{H^\angle}$, since $\angle$ commutes with complex conjugation.
The Hermitian Grassmannian has local coordinates given by $k\times k$ Hermitian matrices.

The real Lagrangian $k$-planes are fixed by both complex conjugation and the Lagrangian involution.
They form the \emph{real Lagrangian Grassmannian}
\[\RLG(V) = \RGr(k,V)_\angle\,,\]
which has local coordinates given by $k\times k$ real symmetric matrices.
We observe that the real Lagrangian Grassmannian may be defined in several equivalent ways,
\[\RLG(V) = \RGr(k,V) \cap \LG(V) = \RGr(k,V) \cap \HG(V) = \LG(V) \cap \HG(V)\,.\]

Suppose $X$ and $Z$ are irreducible varieties of the same dimension, and $f:X\rightarrow Z$ is a dominant map of degree $d$.
The number of complex points in the fiber $f^{-1}(z)$ over a general point $z\in Z$ is $d$.
Furthermore, if $X$ and $Z$ are real varieties and $f$ is real, then the fiber $f^{-1}(z)$ over a real point $z\in Z(\R)$ is a real variety, and for general $z\in Z(\R)$ the fiber $f^{-1}(z)(\R)$ satisfies the congruence
\[\#(f^{-1}(z)(\R))\equiv \#(f^{-1}(z)) \mod 2\,,\]
since nonreal points come in conjugate pairs.
By degenerating to special fibers and counting multiplicities, we see that this congruence holds for all $z\in Z(\R)$.

If $X$ is equipped with an involution $\angle$ such that $f\circ \angle = f$, then the points of $f^{-1}(z)$ not fixed by $\angle$ satisfy another congruence modulo two.
We give a nondegeneracy condition which implies that these two involutions are independent, giving a stronger congruence modulo four.
\begin{proposition}\label{prop:simpleLemma}
 Suppose $X$ is an irreducible real variety with a real involution $\angle$, $Z$ is a real variety of the same dimension, and $f:X\rightarrow Z$ is a dominant real map such that $f\circ\angle = f$ and $\codim_Z f(X_\angle)\geq 2$.
 If $y,z\in Z(\R)$ are general points in the same connected component of $Z(\R)$, then
 \[\#(f^{-1}(y)(\R))\equiv \#(f^{-1}(z)(\R)) \mod 4\,.\]
\end{proposition}
\begin{proof}
 We prove this for sufficiently general points $y,z\in Z(\R)$.
 By degenerating to special fibers and counting multiplicities, this congruence holds for all real points $y,z\in Z(\R)$ in the same connected component.

 Since $\codim_Z f(X_\angle)\geq 2$, there is a path $\Gamma:[0,1]\rightarrow Z(\R)$ with $\Gamma(0)=y$ and $\Gamma(1)=z$ having a finite set of critical values $\{c_1,\dotsc,c_r\}=:C\subset (0,1)$, such that $\Gamma$ does not meet $f(X_\angle)$.
 Taking the closure, $X_\Gamma := \closure(f^{-1}(\Gamma([0,1]\setminus C)))$, we obtain a map $f_\Gamma : X_\Gamma \rightarrow [0,1]$ having all fibers finite and stable under conjugation.
 Let $w\in \Gamma([0,1])$ and $x\in f_\Gamma^{-1}(w)$.
 Since the fiber $f_\Gamma^{-1}(w)$ is real and stable under $\angle$, we have that $A_x:=\{x,\overline{x},x^\angle,\overline{x}^\angle\}\subset f^{-1}(w)$.
 Thus $x$ may be grouped in one of the following ways:
 \begin{itemize}
  \item[(1)] $A_x=\{x,\overline{x},x^\angle,\overline{x}^\angle\}$ contains four distinct points,
  \item[(2)] $A_x=\{x = \overline{x},x^\angle = \overline{x}^\angle\}$ contains two distinct real points,
  \item[(3)] $A_x=\{x = \overline{x}^\angle,\overline{x} = x^\angle\}$ contains two distinct Hermitian points,
  \item[(4)] $A_x=\{x = x^\angle,\overline{x} = \overline{x}^\angle\}$ contains two distinct Lagrangian points, or
  \item[(5)] $A_x=\{x = \overline{x} = x^\angle = \overline{x}^\angle\}$ contains one real Lagrangian point.
 \end{itemize}  
 Since the fibers over $\Gamma$ contain no Lagrangian points, types (4) and (5) do not occur.
 The number of points $x$ of type (1), (2), or (3) respectively is locally constant on $\Gamma([0,1]\setminus C)$.
 So as we vary $w$ continuously, the number of real points in the fiber $f^{-1}(w)$ may only change at a critical value $c$.
 It is forbidden that a Hermitian pair collides to a real point $x_0$ with multiplicity 2 at $c_i$, because such an $x_0$ would be Lagrangian, and $\Gamma([0,1])$ contains no Lagrangian points.
 Thus points in the fiber may not pass from type (3) to type (2) or vice versa, and we see the only way for the number of real points to change at $c$ is for points in the fiber to change from type (1) to type (2) or vice versa.

 Suppose $x\in f^{-1}(c)$ is a real point in the fiber of the critical value $c$.
 Let $a_1(t),\dotsc,$ $a_p(t)$ be the nonreal points of $f^{-1}(t)$ which collide to $x$ as $t$ approaches $c$ from below.
 Let $b_1(t),\dotsc,b_q(t)$ be the nonreal points of $f^{-1}(t)$ which collide to $x$ as $t$ approaches $c$ from above.
 Since the points $a_i(t)$ for $i\in[p]$ come in pairs, $p$ is even.
 Similarly $q$ is even, so $q-p$ is even.
 
 On the other hand, $x^\angle = x$ is in $f^{-1}(c)$.
 We have that $a^\angle_1(t),\dotsc,a^\angle_p(t)$ are the nonreal points approaching $x$ as $t$ increases to $c$, and $b^\angle_1(t),\dotsc,b^\angle_q(t)$ are the nonreal points approaching $x$ as $t$ decreases to $c$.
 Since $2(q-p)$ is a multiple of four, we have that the number of points changing from type (1) to type (2) or vice versa is a multiple of four.
\end{proof}

\section{A Congruence Modulo Four in Real Schubert Calculus}

Consider the Schubert problem $\balpha = \ik$ in $\Gr(k,V)$ involving the intersection of $k^2$ hypersurface Schubert varieties.
Schubert \cite{Schubert1886} calculated the number of complex points in an instance of $\ik$,
\[\#_k := \frac{(k^2)!1!\cdots(k-1)!}{k!(k+1)!\cdots(2k-1)!}\,,\]
and Theorem \ref{thm:dimTrans} implies that this is the number of complex points counting multiplicity in a instance of $\ik$ if the flags involved osculate a rational normal curve at distinct points.

More generally, we write $\#(\balpha)$ to denote the number of complex points in an instance of a Schubert problem $\balpha$.
We give one of the main results of this chapter.
\begin{theorem}\label{thm:WrMod4}
 Suppose $k\geq 3$ and $n=2k$.
 Given a set of distinct points $(t_1,\dotsc,t_{k^2})$ in $\p^1$, stable under complex conjugation, the number of real points of the instance
 \[X:=X_{\Is}(t_1) \cap \cdots \cap X_{\Is}(t_{k^2})\]
 of the Schubert problem $\ik$ is congruent to the number of complex points modulo four.
\end{theorem}
\begin{proof}
 We make the assumption that the list $t = (t_1,\dotsc,t_{k^2})$ is sufficiently general so that the points of $X$ are distinct.
 We observe that since the map has finite fibers the theorem holds for all lists $t$ counting multiplicities by a limiting argument.

 We use the interpretation of $X$ as an inverse Wronski problem, described in Section \ref{sec:IWP}, and we show that $f=\Wr : \Gr(k,\C_{2k}[t]) \longrightarrow \p\C_{k^2+1}[t]$ satisfies the hypotheses of Proposition \ref{prop:simpleLemma}.
 Since $\Gr(k,\C_{2k}[t])$ is smooth and connected, it is irreducible.
 The isomorphism $J$ giving $\angle\circ{\perp}$ is real, so $\angle$ is a real map.
 Complex conjugation on $\p\C_{k^2+1}[t]$ is the usual complex conjugation of coefficients.
 We have $\dim(\Gr(k,\C_{2k}[t])) = k^2 = \dim(\p\C_{k^2+1}[t])$.

 By Theorem \ref{thm:dimTrans}, $\Wr$ is finite.
 Let $H\in \Gr(k,\C_{2k}[t])$.
 By identifying the inverse Wronski problem with intersections of osculating hypersurface Schubert varieties, we apply the key fact of Corollary \ref{cor:SPfixedbyLag}, and we have $\Wr\circ\angle (H) = \Wr(H)$.
 Since $k\geq 3$, we have 
 \[\dim(\LG(\C_{2k}[t]))=\tbinom{k+1}{2}\leq k^2-2 = \dim(\p\C_{k^2+1}[t])-2\,,\]
 so $\codim_{\p\C_{k^2+1}[t]}\Wr(\Gr(k,\C_{2k}[t])_{\angle})\geq 2$.
 The points of $\p\R_{k^2+1}[t]$ are connected, since they make up the projective space of real polynomials.

 Since $\Wr$ satisfies the hypotheses of Proposition \ref{prop:simpleLemma}, we have that the number of real points in a fiber $\Wr^{-1}(z)$ over $z\in \p\R_{k^2+1}[t]$ is fixed modulo four.
 Sottile proved that there is a point $z\in \p\R_{k^2+1}[t]$ whose fiber $\Wr^{-1}(z)$ has all $\#_k$ points real \cite{Sot99}.
 Applying Proposition \ref{prop:simpleLemma} we have $\#(\Wr^{-1}(y)) \equiv \#_k \mod 4$ for any real $y$.
 Interpreting this as an intersection of Schubert varieties, we have the congruence $\#(X(\R))\equiv \#_k \mod 4$.
\end{proof}
Theorem \ref{thm:WrMod4} explains the congruence modulo four found in Table \ref{tab:sharpness2}, which presents data for $({\It}^9)$ in $\Gr(3,\mathbb C^6)$.
Eventually, we prove a congruence for more general Schubert problems, such as $(\IIIiiiIIIt,\It^7)$ in $\Gr(4,\C^8)$ presented in Table \ref{tab:gap12}.

\begin{corollary}
 Let $k=3$ and $n=6$.
 Given a set of distinct points $(t_1,\dotsc,t_{9})$ in $\p^1$, stable under complex conjugation, the number of real points of the real osculating instance
 \[X:=X_{\Is}(t_1) \cap \cdots \cap X_{\Is}(t_{9})\]
 of the Schubert problem $(\It^9)$ is at least $2$.
\end{corollary}
\begin{proof}
 The number of complex points in $X$ is $\#_3 = 42$.
 The corresponding topological lower bound of Theorem \ref{Th:EG} on the number of real points is 0, because $n=2k$ is even.
 Since $2 \equiv 42 \mod 4$ is the least non-negative integer congruent to $\#_3$, it is a lower bound on the number of real points in $X$.
 The data presented in Table \ref{tab:sharpness2} found using symbolic means verify that the lower bound 2 is sharp.
\end{proof}

To generalize Theorem \ref{thm:WrMod4}, we introduce the variety $(\p^1)^m_{\neq}$ consisting of $m$-tuples of distinct points in $\p^1$.
Let $\balpha = (\alpha^1,\dotsc,\alpha^m)$ be a Schubert problem and define $X_\balpha \subset \Gr(k,V) \times (\p^1)^m$ to be the closure of the variety
\[X^\circ := \left\{ (H,t)\ |\ t\in (\p^1)^m_{\neq}\,,\mbox{ and }H\in X_{\alpha^i}(t_i)\ \mbox{for}\ i\in[m] \right\}\,.\]

By Theorem \ref{thm:dimTrans}, the fibers of the projection $X^\circ\rightarrow (\p^1)^m_{\neq}$ are finite, so $\dim(X^\circ)=m$.
The projection $X^\circ \rightarrow (\p^1)^m_{\neq}$ induces a projection $X_\balpha \rightarrow (\p^1)^m$, and work of Purbhoo \cite{Purbhoo2010} shows that every fiber of the induced projection contains $\#(\balpha)$ points, counting multiplicities.

The variety $X_\balpha$ turns out to have the wrong real structure for our study of symmetrically defined Schubert problems.
When distinct osculation points $t_i\neq t_j$ are associated to a common Schubert condition $\alpha_i=\alpha_j$, we may have $X:=X_{\alpha^1}(t_1) \cap \cdots \cap X_{\alpha^m}(t_m)$ and $H\in X$ both real, but $(H,t_1,\dotsc,t_m)\in X_\balpha$ not real.
To rectify this, we will a variety related to $X_\balpha$ by projecting it to an auxiliary variety which forgets some of the order of the list $(t_1,\dotsc,t_m)$.

Recall the exponential notation $\widehat{\balpha}^\ba$ for a Schubert problem, introduced in Chapter \ref{chapLower}.
Given an exponent vector $\ba$ whose entries sum to $m$, and setting the convention $a_0:=0$, we give an equivalence relation $\sim$ which separates $t$ into blocks of size $a_1,\dotsc,a_p$, forgetting the order of the points $t_i\in\p^1$ within each block.
Formally, we define $\sim$ on $(\p^1)^m_{\neq}$ by $t=(t_1,\dotsc,t_m)\sim (s_1,\dotsc,s_m)=:s$ if
\[\{t_{a_0+\cdots+a_{i-1}+1},\dotsc,t_{a_0+\cdots+a_{i}}\} = \{s_{a_0+\cdots+a_{i-1}+1},\dotsc,s_{a_0+\cdots+a_{i}}\}\,,\mbox{ for }i\in[p]\,,\]
as sets.
\begin{example}
 Let $\ba = (1,2,2)$.
 We give a maximal set of equivalent points in $(\p^1)^5_{\neq}$, using vertical lines to separate the blocks given by $\ba$:
 \[(0\,|\,1,\infty\,|\,2,5) \sim (0\,|\,\infty,1\,|\,2,5) \sim (0\,|\,1,\infty\,|\,5,2) \sim (0\,|\,\infty,1\,|\,5,2)\,.\]
\end{example}
\begin{definition}\label{def:regardPoly}
 By realizing the entries in the $i$th block of $(\p^1)^m_{\neq}$ as roots of a polynomial  $f_i$ of degree $a_i$ we have
 \[\p^\ba := \frac{(\p^1)^m_{\neq}}{\sim} \subset \prod_{i=1}^p \p^{a_i} \,,\]
where the usual coordinates in $\p^{a_i}$ are the coefficients of $f_i$.
 The inclusion is as a dense open subset.
\end{definition}

 Suppose $\balpha$ contains $m$ Schubert conditions, and $\widehat{\balpha}^\ba$ contains $p$ distinct Schubert conditions, and assume they give the same Schubert problem.
 We say $\balpha$ is \emph{sorted with respect to $\widehat{\balpha}$} if for $1\leq i < j \leq p$, each occurrence of $\widehat{\alpha}^i$ precedes each occurrence of $\widehat{\alpha}^j$ in $\balpha$.
\begin{definition}\label{def:rightFamily}
 Let $\widehat{\balpha}^\ba$ be the exponential representation of a Schubert problem $\balpha = (\alpha^1,\dotsc,\alpha^m)$, and assume $\balpha$ is sorted with respect to $\widehat{\balpha}$.
 We define $\widehat{X_\balpha}\subset \Gr(k,V)\times \p^\ba$ to be the closure of the variety
 \[\left\{ (H,t)\ |\ t\in \p^\ba\,,\mbox{ and }H\in X_{\alpha^i}(t_i)\ \mbox{for}\ i \in [m]\right\}\,.\]
 This is well defined, because $\balpha$ is stable under the permutation mapping $t$ to $s\sim t$.
\end{definition}
We may define $\widehat{X_\balpha}$ when $\balpha$ is not sorted with respect to $\widehat{\balpha}$.
We use this more general definition, but we do not give the technical details since they are straightforward but unenlightening.

As a direct consequence of Corollary \ref{cor:SVarsComeInPairs}, $X := X_{\alpha^1}(t_1) \cap \cdots \cap X_{\alpha^m}(t_m)$ is a real variety if and only if $\overline{(t_1,\dotsc,t_m)}\sim (t_1,\dotsc,t_m)$, that is, if and only if $t$ is real in $\p^\ba$.
Thus we have the desired property that $X$ and $H\in X$ are simultaneously real if and only if $(H,t_1,\dotsc,t_m)\in \widehat{X_\balpha}$ is real.
Theorem \ref{thm:dimTrans} implies that each fiber of
\[\pi : \widehat{X_\balpha} \longrightarrow \p^\ba\]
has $\#(\balpha)$ points.
Thus we generalize Theorem \ref{thm:WrMod4}.
\begin{theorem}\label{thm:mod4}
 Suppose $\balpha = (\alpha^1,\dotsc,\alpha^m)$ is a symmetric Schubert problem, and $t\in (\p^1)^m_{\neq}$.
 If $\codim(\pi((\widehat{X_{\balpha}})_\angle))\geq 2$ and the instance 
 \[X:=X_{\alpha^1}(t_1) \cap \cdots \cap X_{\alpha^m}(t_{m})\]
 of $\balpha$ is real, then the number of real points in $X$ is congruent to $\#(\balpha)$ modulo four, counting multiplicities.
\end{theorem}

Since proving the subtle relation $\codim(\pi((X_{\balpha})_\angle))\geq 2$ may be difficult, we give a weaker statement which arises from calculating a lower bound on $\codim(\pi((X_{\balpha})_\angle))$.
\begin{proposition}\label{prop:mod4tech}
 Suppose $\balpha = (\alpha^1,\dotsc,\alpha^m)$ is a symmetric Schubert problem containing no trivial Schubert condition $\alpha$ with $|\alpha|=0$, $t\in (\p^1)^m_{\neq}$, and for some distinct $i,j,l\in [m]$ either $\alpha^i=\alpha^j=\alpha^l$ or $\alpha^i\neq \alpha^j$.
 If
 \begin{equation}\label{eqn:criterion}
  m - \tbinom{k+1}{2} + \|\alpha^i\| + \|\alpha^j\| - 2\geq 2\,,
 \end{equation}
 and the instance
 \[X:=X_{\alpha^1}(t_1) \cap \cdots \cap X_{\alpha^m}(t_{m})\]
 of $\balpha$ is real, then the number of real points in $X$ is congruent to $\#(\balpha)$ modulo four, counting multiplicities.
\end{proposition}
\begin{proof}
 Our goal is to apply Proposition \ref{prop:simpleLemma}.
 To do this, we describe a dense open subset $\widehat{X_\balpha}^\circ\subset\widehat{X_\balpha}$ for which $\pi((\widehat{X_\balpha}^\circ)_\angle)$ has at least codimension two, and $\pi((\widehat{X_\balpha}^\circ))(\R)$ is connected.

 Suppose the Schubert problem $\balpha$ has exponential representation $\widehat{\balpha}^\ba$ with $p$ distinct Schubert conditions.
 We have a commuting diagram of maps
 \[
   \xymatrix{
    & X_\balpha^\circ \ar[dr]_\rho \ar[rr]^\phi \ar[dl]_{\tilde{\rho}} & & \widehat{X_\balpha}\ar[d]^\pi\\
    (\p^1)^2_{\neq} \times (\p^1)^{m-2} \ar[rr]_\iota & & (\p^1)^m \ar[r]_\psi & \p^\ba
   }  \,, 
 \]
 where $\tilde{\rho}$ is given by $(H,s_1,\dotsc,s_m)\mapsto (s_i,s_j,s_1,\dotsc,\widehat{s_i},\dotsc,\widehat{s_j},\dotsc,s_m)$, and $X_\balpha^\circ$ is the dense open subset of $X_\balpha$ consisting of points $\{s\ |\ s_i\neq s_j\}$.
 The maps $\tilde{\rho},\rho,\pi$ have degree $\#(\balpha)$, the maps $\phi,\psi$ have degree $\prod_{i=1}^p a_i!$, and $\iota$ is injective.
 Thus, each map has finite fibers.
 
 We claim that $\widehat{X_\balpha}^\circ := \phi(X_\balpha^\circ)$ is a dense open subset of $\widehat{X_\balpha}$ with the codimension condition given above and that $\widehat{X_\balpha}^\circ(\R)$ is connected.
 To see this, we observe that the projection $\pi':\widehat{X_\balpha}^\circ \rightarrow \Gr(k,V)$ lifts along $\phi$ to a projection $\rho':X_\balpha^\circ \rightarrow \Gr(k,V)$ with finite fibers.
 These maps commute with the Lagrangian involution, so we have a commuting diagram of maps between the sets of fixed points,
 \[
   \xymatrix{
     (X_\balpha^\circ)_\angle \ar[dr]_{\rho'} \ar[rr]^{\phi}  & & (\widehat{X_\balpha}^\circ)_\angle\ar[dl]^{\pi'}\\
     & \LG(V)
   }\,.
 \]
 These maps have finite fibers.
 For each $r\in (\p^1)^2_{\neq} \times (\p^1)^{m-2}$, we have
 \[\rho'(\tilde{\rho}^{-1}(r))_\angle \subset Y_{\alpha^i}(r_i) \cap Y_{\alpha^j}(r_j)\,.\]
 Since $r_i\neq r_j$, we have
 \[\dim (Y_{\alpha^i}(r_i) \cap Y_{\alpha^j}(r_j)) = \tbinom{k+1}{2} - \|\alpha^1\| - \|\alpha^2\| =: C\,.\]
 The variety
 \[
  Y:=\{(H,r)\ |\ r\in (\p^1)^2_{\neq}\,,\ H\in Y_{\alpha^i}(r_i)\cap Y_{\alpha^j}(r_j) \}
 \]
 has dimension $\dim(Y) = C+2$, which implies $(\widehat{X_\balpha}^\circ)_\angle \subset (\pi')^{-1}(Y)$ has dimension at most $C+2$.
 Therefore, $\dim ( \pi((\widehat{X_\balpha}^\circ)_\angle) ) \leq C+2$.
 By Inequality (\ref{eqn:criterion}),
 \[\codim ( \pi((\widehat{X_\balpha}^\circ)_\angle) ) \geq m - C - 2 \geq 2\,.\]
 
 Having established the codimension hypothesis of Proposition \ref{prop:simpleLemma}, it is enough to prove that given two points $y,z\in \pi(\widehat{X_\balpha}^\circ)(\R)$, there is a real path connecting them.
 Thus we take a path $\Gamma:[0,1]\rightarrow (\p^1)^m$ such that $\psi\circ \Gamma$ is a real path connecting $y$ and $z$, and we show that we may require $r_i(x) \neq r_j(x)$ for $x\in [0,1]$.

 If we assume $\alpha^i=\alpha^j=\alpha^l$, then the projections $r_i(x),r_j(x),r_l(x)$ of $x\in[0,1]$ under $\Gamma$ are roots of a single polynomial $f_x$ given in Definition \ref{def:regardPoly}.
 Since $\deg(f_x)\geq 3$, we may choose $\Gamma$ so that $r_i(x) \neq r_j(x)$ for $x\in [0,1]$ with $\psi\circ \Gamma$ real.
 On the other hand, if $\alpha^i\neq \alpha^j$, then $r_i(x),r_j(x)$ are roots of different polynomials.
 Again, we may choose $\Gamma$ so that $r_i(x) \neq r_j(x)$ for $x\in [0,1]$ with $\psi\circ \Gamma$ real.
 
 Applying Proposition \ref{prop:simpleLemma}, we see the number of real points in $X(s)$ is fixed modulo four.
 The Mukhin-Tarasov-Varchenko Theorem \ref{Th:MTV} gives $s$ such that $X(s)$ has all $\#(\balpha)$ solutions real.
 Thus we have $\#(X(\R))\equiv \#(\balpha)\mod 4$.
\end{proof}

Proposition \ref{prop:mod4tech} proves the congruence modulo four for some of the problems we studied computationally, reported in Chapter \ref{chapLower}.

\begin{example}
 In Chapter \ref{chapLower}, we proved that real osculating instances of $\left(\IIIiiiIIIt,\It^7\right)$ in $\Gr(4,\C^8)$ have $20 \mod 4$ real solutions by counting real factorizations of a real polynomial.
 Table \ref{tab:gap12} gives data for this problem.

 Proposition \ref{prop:mod4tech} gives another proof for this congruence modulo four, since
 \[8-\binom{5}{2}+\left\|\IIIiiiIII\right\|+\left\|\I\right\|-2=8-10+6+1-2=3\geq 2\,.\]
\end{example}
\begin{example}
 Consider the Schubert problem $\left(\IIIiiiIIt,\It^8\right)$ in $\Gr(4,\C^8)$ with $90$ complex solutions.
 We solved $100,000$ instances of $\left(\IIIiiiIIt,\It^8\right)$, and in each instance we observed $90\mod 4$ real solutions.
 Proposition \ref{prop:mod4tech} proves this congruence modulo four, since
 \[9-\binom{5}{2}+\left\|\IIIiiiII\right\|+\left\|\I\right\|-2=9-10+5+1-2=3\geq 2\,.\]

 Since two is the least positive integer congruent to $90$ modulo four, two is a lower bound for the number of real solutions to a real osculating instance of $\left(\IIIiiiIIt,\It^8\right)$.
 In our computations, we have found $5853$ real osculating instances of $\left(\IIIiiiIIt,\It^8\right)$ with exactly two real solutions, so the lower bound of two is sharp.
 The previously known topological lower bound from Theorem \ref{Th:EG} was zero.
\end{example}
\begin{example}
 Consider the Schubert problem $\left(\IIIiIt,\IIit,\It^8\right)$ in $\Gr(4,\C^8)$ with $426$ complex solutions.
 Every real osculating instance of $\left(\IIIiIt,\IIit,\It^8\right)$ has $426\mod 4$ real solutions by Proposition \ref{prop:mod4tech}, since
 \[10-\binom{5}{2}+\left\|\IIIiI\right\|+\left\|\IIi\right\|-2=10-10+3+2-2=3\geq 2\,.\]

 Since two is the least positive integer congruent to $426$ modulo four, two is a lower bound for the number of real solutions to a real osculating instance of $\left(\IIIiIt,\IIit,\It^8\right)$.
 The previously known topological lower bound from Theorem \ref{Th:EG} was zero.
 We do not yet know if two is the sharp lower bound.
\end{example}
\begin{corollary}\label{cor:improveEGSS}
 Let $\balpha = (\alpha,\beta,\It^m)$ be a symmetric Schubert problem with $\#(\balpha)\not \equiv 0 \mod 4$ and $X$ be a real osculating instance of $\balpha$.
 Suppose the hypotheses of Proposition \ref{prop:mod4tech} are satisfied by $\balpha$.
 If the number of boxes above the main diagonal of the skew Young diagram $d((\beta)'/\alpha)$ is odd, then two is a lower bound for the number of real solutions to $X$.
 The previously known topological lower bound for such a problem was $\Sigma(\alpha,\beta)=0$.
\end{corollary}
\begin{proof}
 The only parts of Corollary \ref{cor:improveEGSS} that do not follow immediately from Proposition \ref{prop:mod4tech} are the assertion $\Sigma(\alpha,\beta)=0$ and the implicit assertion that $\#(\balpha)$ is even.

 The sign imbalance $\Sigma(\alpha,\beta)$ as defined in Proposition \ref{prop:EGSS} may be calculated by observing that every standard Young tableau of shape $d((\beta)'/\alpha)$ may be uniquely paired with another standard Young tableau of the same shape by reflecting the tableau along the main diagonal.
 We give an example of paired tableaux with opposite signs.
 \[
  \begin{picture}(127.5,36)
   \put (0,0){\IIIiiiIIminusIb}
   \put (100,0){\IIIiiiIIminusIb}

   \put (15.5,25.5){1}
   \put (26.5,25.5){2}
   \put (3.5,14){3}
   \put (15,14){4}
   \put (26.5,14){5}
   \put (3.5,2.5){6}
   \put (15,2.5){7}

   \put (57,15){$\longleftrightarrow$}

   \put (115.5,25.5){3}
   \put (126.5,25.5){6}
   \put (103.5,14){1}
   \put (115,14){4}
   \put (126.5,14){7}
   \put (103.5,2.5){2}
   \put (115,2.5){5}
  \end{picture}
\]
 This operation is an odd permutation, since there are an odd number of boxes above the diagonal, so the paired tableaux have opposite signs.
 This implies that $\Sigma(\alpha,\beta)=0$.
 
 Since $\#(\balpha)$ is the number of tableaux of shape $d((\beta)'/\alpha)$, and since the number tableaux with sign $+1$ equals the number of tableaux with sign $-1$, $\#(\balpha)$ is even.
\end{proof}
Proposition \ref{prop:mod4tech} is highly technical, and we believe a stronger, simpler statement is true.
Assuming one may generalize the dimensional transversality theorem of Eisenbud and Harris, Theorem \ref{thm:dimTrans}, to intersections of Schubert varieties in a Lagrangian Grassmannian, one could easily calculate the codimension involved in Theorem \ref{thm:mod4} using combinatorial data.
This would give the following result.
\begin{conjecture}\label{conjecture}
 Let $X$ be a real osculating instance of a symmetric Schubert problem $\balpha$ in $\Gr(k,V)$.
 If
 \begin{equation}\label{eqn:conj}
  \| \balpha \| - \tbinom{k+1}{2} \geq 2\,,
 \end{equation}
 then the number of real points in $X$ satisfies the congruence $\#(X(\R))\equiv \#(\balpha)$ mod $4$.
\end{conjecture}

By permuting the entries of $\balpha$, we may assume $i=1$ and $j=2$ in Inequality \ref{eqn:criterion}.
Since none of the Schubert conditions in $\balpha$ is trivial, we have
\[\|\alpha^3\|+\dotsb+\|\alpha^m\| \geq m-2\,.\]
This implies $\| \balpha \| - \tbinom{k+1}{2} \geq m - \tbinom{k+1}{2} + \|\alpha^1\| + \|\alpha^2\| - 2$.
Therefore, assuming Inequality (\ref{eqn:criterion}) gives Inequality (\ref{eqn:conj}).
Thus Conjecture \ref{conjecture} implies Proposition \ref{prop:mod4tech}.

\section{Support for Conjecture \ref{conjecture}}

We used supercomputers to study all 44 nontrivial symmetric Schubert problem $\balpha$ on $\Gr(k,V)$ with $k\leq 4$ and $\#(\balpha)\leq 96$.
Ten of these Schubert problems satisfy the hypotheses of Proposition \ref{prop:mod4tech} (and thus the hypotheses of Conjecture \ref{conjecture}), and we observed the expected congruence modulo four.
We gave the data for two of these problems in Tables \ref{tab:sharpness2} and \ref{tab:gap12}.

We studied 11 symmetric Schubert problems which satisfy the hypotheses of Conjecture \ref{conjecture} but not those of Proposition \ref{prop:mod4tech}.
In each of these problems, the conjectured congruence was observed.
\begin{example}
 Consider the symmetric Schubert problem $(\IIiit^4)$ for $k=4$.
 This problem does not satisfy Inequality (\ref{eqn:criterion}) of Proposition \ref{prop:mod4tech},
 \[4-\binom{5}{2}+\left\|\IIii\right\|+\left\|\IIii\right\|-2=4-10+3+3-2=-2\not\geq 2\,.\]
 However, we see that Inequality (\ref{eqn:conj}) is satisfied,
 \[4\cdot \left\|\IIii\right\| - \binom{4+1}{2} = 4\cdot 3-10 = 2\geq 2\,.\]
 Thus Conjecture \ref{conjecture} claims that the number of real solutions is fixed modulo four.
 We verified this claim for $3,000$ examples, giving the data in Table \ref{tab:22e4}.
 These data consumed $1.486$ GHz-years of processing power.

 Indeed this Schubert problem cannot be a counter example to Conjecture \ref{conjecture}.
 The computational study \cite{secant} of Schubert problems given by secant flags (a generalization of osculating flags) uncovered the congruence modulo four for real instances of $(\IIiit^4)$.
 This problem was analyzed the congruence we observe for this problem was proven for all real instances of $(\IIiit^4)$, including those which are not osculating instances.
\end{example}
\begin{table}[tp]
 \caption{Support for Conjecture \ref{conjecture}.}
  \begin{center}
 \begin{tabular}{|c||c|c|c||c|}											
  \hline											
  {\bf $\#$ Real} & \multicolumn{3}{c||}{\rule{0pt}{12pt}$\IIiit^4$} & \multirow{2}{*}{\bf Total} \\ \cline{2-4}
  {\bf Solutions} & $r_{\IIiis}=	4	$ & $r_{\IIiis}=	2	$ & $r_{\IIiis}=	0	$ & \\ \hline \hline
{\bf 	0	} &	 	&	 	&	 	& {\bf 	0	} \\ \hline
{\bf 	2	} &	 	&	687	&	 	& {\bf 	687	} \\ \hline
{\bf 	4	} &	 	&	 	&	 	& {\bf 	0	} \\ \hline
{\bf 	6	} &	1000	&	313	&	1000	& {\bf 	2313	} \\ \hline \hline
{\bf 	Total	} & {\bf	1000	} & {\bf	1000	} & {\bf	1000	} & {\bf	3000	} \\ \hline
 \end{tabular}											
  \end{center}
 \label{tab:22e4}
\end{table}
We tested 23 symmetric Schubert problems which do not satisfy the hypotheses of Conjecture \ref{conjecture}.
Nineteen of these problems, including the problem of four lines $(\I^4)$ with $k=2$, did not exhibit a congruence modulo four.
\begin{example}
 The symmetric Schubert problem $(\IIit^2,\It^3)$ in $\Gr(3,\C^6)$ does not satisfy Inequality (\ref{eqn:conj}),
 \[2\cdot \left\|\IIi\right\|+3\cdot \left\|\I\right\| - \binom{3+1}{2} = 2\cdot 2+3\cdot 1-6 = 1\not\geq 2\,,\]
 so it cannot satisfy the more restrictive Inequality (\ref{eqn:criterion}).
 The results of symbolic computations displayed in Table \ref{tab:upperLower} show that the number of real solutions to real instances of this problem is not fixed modulo four.
\end{example}
Four of the 44 symmetric Schubert problems tested do not satisfy the hypotheses of Conjecture \ref{conjecture}, but exhibit a congruence modulo four on the number of real solutions.
Tables \ref{tab:311e2.21e2}--\ref{tab:321e2.22} present data collected for these four problems.
The values of $\|\balpha\|$ for these problems are 10, 11, 11, and 11 respectively, but a symmetric problem $\balpha$ in $\Gr(4,\C^8)$ must have $\|\balpha\|\geq 12$ to satisfy the hypotheses of Conjecture \ref{conjecture}.
\begin{table}[tp]
 \caption{Congruence not implied by Conjecture \ref{conjecture}.}
  \begin{center}
 \begin{tabular}{|c||c|c|c|c||c|}											
  \hline											
  \multirow{2}{*}{\bf $\#$ Real} & \multicolumn{4}{c||}{\rule{0pt}{12pt}$\IIIiIt^2\quad \IIit^2$} & \multirow{3}{*}{\bf Total} \\ \cline{2-5}
  \multirow{2}{*}{\bf Solutions} 		& $ r_{\IIIiIs} = 	2	$ & $ r_{\IIIiIs} = 	2	$ & $ r_{\IIIiIs} = 	0	$ & $ r_{\IIIiIs} = 	0	$ & \\ \cline{2-5}
		& $ r_{\IIis} = 	2	$ & $ r_{\IIis} = 	0	$ & $ r_{\IIis} = 	2	$ & $ r_{\IIis} = 	0	$ &		\\ \hline \hline
{\bf 	0	} &	 	&	73716	&	73895	&	 	& {\bf 	147611	} \\ \hline
{\bf 	2	} &	 	&	 	&	 	&	 	& {\bf 	0	} \\ \hline
{\bf 	4	} &	 	&	26284	&	26105	&	100000	& {\bf 	152389	} \\ \hline
{\bf 	6	} &	 	&	 	&	 	&	 	& {\bf 	0	} \\ \hline
{\bf 	8	} &	100000	&	 	&	 	&	 	& {\bf 	100000	} \\ \hline \hline 
{\bf 	Total	} & {\bf 	100000	} & {\bf 	100000	} & {\bf 	100000	} & {\bf 	100000	} & {\bf 	400000	} \\ \hline
 \end{tabular}											
  \end{center}
 \label{tab:311e2.21e2}
\end{table}
\begin{table}[tp]
 \caption{Another congruence not implied by Conjecture \ref{conjecture}.}
  \begin{center}
 \begin{tabular}{|c||c|c||c|}
  \hline											
  {\bf $\#$ Real} & \multicolumn{2}{c||}{\rule{0pt}{12pt}$\IIIiiIt^2\quad \IIit\quad\It$} & \multirow{2}{*}{\bf Total} \\ \cline{2-3}
  {\bf Solutions} 		& $ r_{\IIIiiIs} = 	2	$ & $ r_{\IIIiiIs} = 0$ & \\ \hline \hline
{\bf	0}	&	 	&	160337	& {\bf	160337	} \\ \hline	
{\bf	2}	&	 	&	 	& {\bf	0	} \\ \hline	
{\bf	4}	&	 	&	39663	& {\bf	39663	} \\ \hline	
{\bf	6}	&	 	&	 	& {\bf	0	} \\ \hline	
{\bf	8}	&	200000	&	 	& {\bf	200000	} \\ \hline		\hline
{\bf	Total}	& {\bf	200000	} & {\bf	200000	} & {\bf	400000	} \\ \hline
 \end{tabular}											
  \end{center}
 \label{tab:321e2.21.1}
\end{table}
\begin{table}[tp]
 \caption{A third congruence not implied by Conjecture \ref{conjecture}.}
  \begin{center}
 \begin{tabular}{|c||c|c||c|}
  \hline											
  {\bf $\#$ Real} & \multicolumn{2}{c||}{\rule{0pt}{12pt}$\IIIiIt\quad \IIiit^2\quad\IIit$} & \multirow{2}{*}{\bf Total} \\ \cline{2-3}
  {\bf Solutions} 		& $ r_{\IIiis} = 	2	$ & $ r_{\IIiis} = 0$ & \\ \hline \hline
{\bf	0}	&	 	&	142275	& {\bf	142275	} \\ \hline	
{\bf	2}	&	 	&	 	& {\bf	0	} \\ \hline	
{\bf	4}	&	200000	&	57725	& {\bf	257725	} \\ \hline	\hline
{\bf	Total}	& {\bf	200000	} & {\bf	200000	} & {\bf	400000	} \\ \hline
 \end{tabular}											
  \end{center}
 \label{tab:311.22e2.21}
\end{table}
\begin{table}[tp]
 \caption{A fourth congruence not implied by Conjecture \ref{conjecture}.}
  \begin{center}
 \begin{tabular}{|c||c|c||c|}
  \hline											
  {\bf $\#$ Real} & \multicolumn{2}{c||}{\rule{0pt}{12pt}$\IIIiiIt^2\quad \IIiit$} & \multirow{2}{*}{\bf Total} \\ \cline{2-3}
{\bf Solutions} 		& $ r_{\IIIiiIs} = 	2	$ & $ r_{\IIIiiIs} = 0$ & \\ \hline \hline
{\bf	0}	&        &        & {\bf      0} \\ \hline	
{\bf	2}	& 200000 & 200000 & {\bf 400000} \\ \hline \hline
{\bf	Total}	& {\bf	200000	} & {\bf	 200000	} & {\bf	400000	} \\ \hline
 \end{tabular}											
  \end{center}
 \label{tab:321e2.22}
\end{table}

\chapter{\uppercase{A Square Formulation via Duality}}
  \label{chapSquare}
  We say a system of equations is \emph{square} if it has the same number of equations as variables, and \emph{overdetermined} if it has more equations than variables.
The classical determinantal formulation of an instance of a Schubert problem given by Proposition \ref{prop:detConds} is overdetermined if more than two of the Schubert varieties involved are given by Schubert conditions other than $\It$.
Following joint work with Hauenstein and Sottile \cite{square}, we realize an intersection $X$ of Schubert varieties in a larger space so that it is the solution set to a system of polynomial equations, and the number of equations is equal to the codimension of $X$ in the larger space.
If $X$ is an instance of a Schubert problem, this gives a square system and allows one to use algorithms from Smale's $\alpha$-theory to verify approximate solutions obtained by numerical methods \cite{S86}.
This procedure replaces determinantal equations of degree $\min(k,n-k)$ by bilinear equations.

\section{Background}

Computational studies have used Gr\"obner bases to produce compelling conjectures in Schubert calculus \cite{secant,monotone,RS1998,RSSS2006,Sottile2000}, some of which have been proven \cite{EG2002b,mod4,MTV2009a,MTV2009b}.
The use of Gr\"obner bases in these computational studies has the advantage that it produces exact information, and the steps taken to produce that information are inherently a proof of correctness.
This rigidity is partially responsible for the complexity of calculating a Gr\"obner
basis \cite{MM1982}, which is limiting even for zero-dimensional ideals \cite{HL2011}.
Gr\"obner basis calculations do not not appear to scale well when parallelized \cite{Leykin04}, and this makes it difficult to efficiently use modern parallel computing to mitigate their computational complexity.
Calculating the Gr\"obner basis of an instance of a typical Schubert problem with more than 100 solutions or involving more than 16 variables is infeasible in characteristic zero.

Numerical and symbolic methods are subject to different computational bottlenecks, so parallel numerical methods, such as those using a parameter homotopy \cite{SW05}, offer an alternative to symbolic methods for solving Schubert problems beyond the scope of symbolic computation.
There are optimized numerical algorithms for Schubert problems, such as the Pieri homotopy algorithm \cite{HSS98}, which has successfully solved instances of a Schubert problem with 17,589 solutions \cite{LS}.
There is work being done to develop a more general Littlewood-Richardson homotopy \cite{LMSVV,svv} based on Vakil's geometric Littlewood-Richardson rule \cite{Va06a}.
While not optimized for Schubert calculus, regeneration \cite{HSW10} offers a numerical approach for Schubert problems that extends to flag varieties, natural generalizations of the Grassmannian.

Numerical methods generally do not give exact solutions, and the approximations given are not guaranteed to be correct.
When a computer verifies the correctness of numerical output, we say that the output has a certificate of validity.
We say that an approximate solution with a certificate is \emph{certified}.

Newton's method for expressing a root of a univariate polynomial as the limit of a sequence of approximations has a generalization giving a solution to a square system of polynomial equations as a limit.
Let $E=(E_1,\dotsc,E_p)$ be a vector of polynomials in the variables $v=(v_1,\dotsc,v_p)$, and consider $x\in\C^p$ as a vector.
We define the \emph{Jacobian} of $E$ at $x$,
\[\Jac_E(x):=
  \left(
    \begin{matrix}
      \frac{\partial E_1}{\partial x_1} & \cdots & \frac{\partial E_1}{\partial x_p} \\
      \vdots            &        & \vdots            \\
      \frac{\partial E_p}{\partial x_1} & \cdots & \frac{\partial E_p}{\partial x_p} \\
    \end{matrix}
  \right)\,.
\]
We set $N_0(x):=x$ and define the \emph{$i$th Newton iteration} $N_i(x)\in\C^p$ for $i>0$,
\[N_i(x) := N_{i-1}-\Jac_E(N_{i-1}(x))^{-1}E(N_{i-1}(x))\,.\]
\begin{definition}
 Let $N_\infty(x):=\lim_{i\rightarrow\infty}N_i(x)$.
 The sequence of Newton iterations $\{N_i(x)\}$ of $x\in\C^p$ \emph{converges quadratically} to a solution of $E$ if for every $i>0$,
 \[|N_{i+1}(x) - N_\infty(x)|\leq \frac{1}{2^{2^{i}-1}}|x - N_\infty(x)|\,,\]
 where $|\cdot|$ denotes the distance norm in $\C^p$.
 The sequence of Newton iterations converges quadratically if the number of significant digits doubles with each step.
 In this case, $x$ is called an \emph{approximate solution} to $E$ with \emph{associated solution} $N_\infty(x)$.
\end{definition}

There is a positive number $\alpha(x,E)>0$ depending on a point and system of equations so that if
\[\alpha(x,E) < \frac{13-3\sqrt{17}}{4}\]
then $x$ is an approximate solution to $E$ \cite[Ch. 8]{BCSS}.
Smale studied convergence of Newton iterations and established $\alpha$-theory to certify quadratic convergence and thus approximate solutions.
Sottile and Hauenstein showed that given an approximate solution $x$, algorithms from $\alpha$-theory may be used to determine whether its associated solution is real \cite{alphaCertified}.
Given two approximate solutions, one may also determine whether their associated solutions are distinct.
These applications require that $E$ be a square system \cite{DS00}.
Schubert problems are famously overdetermined, and the main goal of this chapter is to formulate them locally using square systems.

\section{Primal-Dual Formulation}\label{sec:primal-dual}

We present a way to formulate an instance of a Schubert problem in a Grassmannian as a square system of equations.
Recall the Stiefel coordinates $\hats(\alpha)$ dual to the local coordinates $\s(\alpha^\perp)$ on $\Gr(k,V)$ and $\hats\tbox_{\alpha}$ from Definition 
\ref{def:HatSBeta} dual to the local coordinates $\s_{\alpha^\perp}$ for $X_{\alpha{\perp}}\Fdot$.
There are also coordinates for an intersection of Schubert varieties, dual to $\s_\alpha^\beta$.
\begin{definition}
 Let $\hats\tbox_\alpha^\beta \subset\Mat_{n\times (n-k)}$ be the set of matrices with entries $m_{i,j}$ satisfying
 \[m_{i,j} = 1 \mbox{ if } i=n+1-\alpha_j\,,\quad\mbox{and}\quad m_{i,j} = 0 \mbox{ if } i < n+1-\alpha_j \mbox{ or } i > \beta_{n-k-j+1}\,.\]
 The matrices $\hats\tbox_\alpha^\beta$ give \emph{Stiefel coordinates} for $X_{\alpha}\Fdot^1 \cap X_\beta \Fdot^2\subset \Gr(n-k,V^*)$.
\end{definition}
\begin{example}
 Let $n=7$, and consider the Grassmannian $\Gr(4,V^*)$ and the Schubert conditions $\alpha=(2,4,5,7)$ and $\beta=(3,4,6,7)$.
 The coordinates $\hats\tbox_\alpha^\beta$ are given by matrices of the form
 \[
   \left(
     \begin{matrix}
       0      & 0      & 0      & 1      \\
       0      & 0      & 0      & m_{24} \\
       0      & 0      & 1      & m_{34} \\
       0      & 1      & m_{43} & 0      \\
       0      & m_{52} & 0      & 0      \\
       1      & m_{62} & 0      & 0      \\
       m_{71} & 0      & 0      & 0      \\
     \end{matrix}
   \right)\,.
 \]
\end{example}
By Corollary \ref{cor:numGens}, the classical determinantal formulation of $X_\alpha\Fdot$ requires more than $|\alpha|$ equations, unless it is given by the Schubert condition $\It$.
Thus an instance of a Schubert problem $\balpha=(\alpha^1,\dotsc,\alpha^m)$ such that $\alpha^i\neq \It$ for some $i\in[m]$ is the solution set to an overdetermined system in its classical determinantal formulation in local coordinates $\s(\alpha)$ for $\Gr(k,V)$.
However, the coordinates $\s_\alpha$ give $X_\alpha \Fdot \subset \Gr(k,V)$ by setting $|\alpha|$ variables equal to zero in the coordinates $\s(\alpha)$.
\begin{example}
 Consider the Grassmannian $\Gr(3,\C^7)$ and $d(2,5,7)=\IIIit$.
 We give coordinates $\s({\IIIit})$ of $\Gr(3,V)$ and $\s_{\IIIis}$ of $X_{\IIIis}\Fdot$ respectively,
 \[
   \left(
     \begin{matrix}
       m_{11} & 1 & m_{13} & m_{14} & 0 & m_{16} & 0 \\
       m_{11} & 0 & m_{13} & m_{14} & 1 & m_{16} & 0 \\
       m_{11} & 0 & m_{13} & m_{14} & 0 & m_{16} & 1 \\
     \end{matrix}
   \right)\quad\mbox{and}\quad
   \left(
     \begin{matrix}
       m_{11} & 1 & {\It}  & {\It}  & 0 & {\It}  & 0 \\
       m_{11} & 0 & m_{13} & m_{14} & 1 & {\It}  & 0 \\
       m_{11} & 0 & m_{13} & m_{14} & 0 & m_{16} & 1 \\
     \end{matrix}
   \right)\,,
 \]
 where $\It$ denotes a coordinate which is identically zero.
\end{example}
Similarly, if $\Fdot^1$ and $\Fdot^2$ are in linear general position, $\s_\alpha^\beta$ parametrizes a dense open subset of $X_\alpha\Fdot^1\cap X_\beta\Fdot^2$ using $\dim(X_\alpha\Fdot^1\cap X_\beta\Fdot^2)$ coordinates.

Recall the map ${\perp}:\Gr(k,V)\rightarrow \Gr(n-k,V^*)$ given by mapping a $k$-plane $H$ to its annihilator $H\mapsto H^\perp$.
\begin{definition}
 Let $\Delta\colon\Gr(k,V)\rightarrow\Gr(k,V)\times\Gr(n - k,V^*)$ be the \emph{dual diagonal map} given by $H\mapsto (H,H^\perp)$.
\end{definition}
\begin{proposition}\label{prop:dualDiagonal1}
 Let $A,B\subset\Gr(k,V)$ be subsets.
 Then we have the equality of sets
 \[\Delta(A\cap B) = (A\times{\perp}(B))\cap\Delta(\Gr(k,V))\,.\]
\end{proposition}
\begin{proof}
 This is a dual version of the classical argument of reduction to the diagonal.
 Abbreviating $\Delta_G:=\Delta(\Gr(k,V))$, we observe
 \[\Delta(A) = (A\times {\perp}(A)) \cap \Delta_G = \left(A\times {\perp}(\Gr(n-k,V^*))\right) \cap \Delta_G\,.\]
 Similarly, we have
 \[\Delta(B) = (B\times {\perp}(B)) \cap \Delta_G = (\Gr(k,V)\times {\perp}(B)) \cap \Delta_G\,.\]
 Together, these give
 \[\Delta(A\cap B) = \Delta(A) \cap \Delta(B) = (A\times {\perp}(B)) \cap \Delta_G\,.\]
\end{proof}
We call $\Delta(A\cap B)$ the \emph{primal-dual} formulation of $A\cap B$.
We call the first factor of $\Delta(A\cap B)$ the \emph{primal factor} and the second factor of $\Delta(A\cap B)$ the \emph{dual factor}.

This gives us a new way to exhibit an intersection $X$ of two Schubert varieties.
Let $M$ be a $k\times n$ matrix of $kn$ indeterminates, giving global Stiefel coordinates for $\Gr(k,V)$ with respect to the standard basis $\be$ of $V$, and let $N$ be a $n\times (n-k)$ matrix of $n(n-k)$ indeterminates, giving global Stiefel coordinates for $\Gr(n-k,V^*)$ with respect to the dual basis $\be^*$ of $V^*$.
Then the rows of $M$ span a point in $\Gr(k,V)$, the columns of $N$ span a point in $\Gr(n-k,V^*)$, and $\Delta(\Gr(k,V))$ is the solution set to the matrix equation
\[MN=0_{k\times (n-k)}\,,\]
where $0_{k\times (n-k)}$ denotes the $k\times (n-k)$ zero matrix.
This equation consists of $k(n-k)$ equations which are bilinear in the entries of $M$ and $N$.

Suppose $\Fdot^1$ and $\Fdot^2$ are flags, and $\alpha,\beta\in\tbinom{[n]}{k}$.
By Proposition \ref{prop:dualDiagonal1}, we have
\[\Delta(X_\alpha \Fdot^1\cap X_\beta \Fdot^2) = (X_\alpha \Fdot^1\times X_{\beta{\perp}} \Fdot^{2{\perp}}) \cap\Delta(\Gr(k,V))\,.\]
Let $\Fdot^1$ be a matrix that is a basis for the flag $\Fdot^1$.
Recall that $\s_\alpha$ gives coordinates for $X_\alpha \Fdot^1$ with respect to some basis $\bff$ of $V$.
Let $M_\alpha$ be a $k\times n$ matrix of indeterminates in $\s_\alpha$.
A change of basis for $V$ from $\bff$ to $\be$ induces a dual action on Stiefel coordinates, so the matrix product $M_\alpha\Fdot^1$ locally parametrizes $X_\alpha \Fdot^1$ with respect to the standard basis $\be$.

Let $\widehat{\Fdot^{2{\perp}}}$ be a basis for the flag $\Fdot^{2{\perp}}$, that is, the first $i$ columns of the matrix $\widehat{\Fdot^{2{\perp}}}$ span $\Fdot^{2{\perp}}$ for all $i$.
Let $\widehat{M}_{\beta{\perp}}$ be a matrix of indeterminates in $\hats_{\beta{\perp}}$ giving Stiefel coordinates for $X_{\beta{\perp}} \Fdot^{2{\perp}}$.
By the argument above, the product $\widehat{\Fdot^{2{\perp}}}\widehat{M}_{\beta{\perp}}$ locally parametrizes $X_{\beta{\perp}} \Fdot^{2{\perp}}$ with respect to the standard dual basis $\be^*$.
It follows that the matrix equation
\begin{equation}\label{eqn:bilin2SV}
 M_\alpha\Fdot^1 \widehat{\Fdot^{2{\perp}}}\widehat{M}_{\beta{\perp}} =0_{k\times (n-k)}
\end{equation}
defines the dual diagonal $\Delta(X_\alpha \Fdot^1\cap X_\beta \Fdot^2)$ in $X_\alpha \Fdot^1\times X_{\beta{\perp}} \Fdot^{2{\perp}}$.

Equation (\ref{eqn:bilin2SV}) gives $k(n-k)$ equations defining the restriction of the dual diagonal $\Delta(\Gr(k,V))$ to a dense open subset of $X_\alpha \Fdot^1\times X_{\beta{\perp}} \Fdot^{2{\perp}}$.
The equations are bilinear in $k(n-k)-|\alpha|$ variables from $M_\alpha$ and $n(n-k)$ variables from $\widehat{M}_{\beta{\perp}}$.
We describe the dense subset of $X_\alpha \Fdot^1\times X_{\beta{\perp}} \Fdot^{2{\perp}}$ involved.

We used the action of $\GL(V)$ to adapt the open cover $\mathcal{G}$ of the Grassmannian to the cover $\mathcal{G}(t)$ from Definition \ref{def:UalphaT} so that $X_\alpha(t)\cap G_\alpha(t)$ is the dense open set parametrized by Stiefel coordinates $\s_\alpha$.
Given a flag $\Fdot$ whose basis is the matrix $\Fdot$, we similarly adapt $\mathcal{G}$ to an open cover
\[\mathcal{G}\Fdot := \{G_\alpha\Fdot^{-1}\ |\ \alpha\in\tbinom{[n]}{k}\}\,.\]
This has the feature that $X_\alpha\Fdot \cap G_\alpha \Fdot^{-1}$ is the dense open set parametrized by the Stiefel coordinates $\s_\alpha$.
Throughout this chapter, we use $X^\circ_\alpha\Fdot$ to denote this dense open set.
Similarly, we write $(X_\alpha\Fdot \cap X_\beta\Gdot)^\circ$, $X^\circ_{\alpha^{\perp}}\Fdot^{\perp}$, and $(X_{\alpha^{\perp}}\Fdot^{\perp} \cap X_{\beta^{\perp}}\Gdot^{\perp})^\circ$ to denote the open dense sets parametrized by $\s_\alpha^\beta$, $\hats\tbox_{\alpha^{\perp}}$, and $\hats\tbox_{\alpha^{\perp}}^{\beta^{\perp}}$ respectively.
The coordinates $\hats\tbox_{\alpha^{\perp}}$ and $\hats\tbox_{\alpha^{\perp}}^{\beta^{\perp}}$ were chosen in a way that yields
\[{\perp}(X_\alpha^\circ \Fdot ) = X^\circ_{\alpha^{\perp}} \Fdot^{\perp} \quad \mbox{and} \quad {\perp} ( X_\alpha \Fdot \cap X_\beta\Gdot)^\circ = (X_{\alpha^{\perp}}\Fdot^{\perp} \cap X_{\beta^{\perp}}\Gdot^{\perp})^\circ\,.\]
We have shown the following.
\begin{proposition}\label{prop:SchubIntCompInt}
 If $X_\alpha \Fdot^1,X_\beta \Fdot^2$ are Schubert varieties in $\Gr(k,V)$, then we have the equality
 \[\Delta(X_\alpha \Fdot^1 \cap X_\beta \Fdot^2) = (X_\alpha \Fdot^1 \times X_{\beta^\perp} \Fdot^{2{\perp}})\cap\Delta(\Gr(k,V))\,,\]
 as sets.
 Furthermore, $\Delta(X_\alpha^\circ \Fdot^1 \cap X_\beta^\circ \Fdot^2)$ is the solution set to the $k(n-k)$ bilinear equations given by Equation $(\ref{eqn:bilin2SV})$ in the coordinates $(\s_\alpha,\hats\tbox_{\beta^\perp})$.
\end{proposition}
More may be said with a straightforward dimension calculation.
\begin{corollary}\label{cor:SchubIntCompInt}
 If $\Fdot^1,\Fdot^2$ are flags in $\Gr(k,V)$ in linear general position and $\alpha,\beta\in\tbinom{[n]}{k}$ satisfy $\alpha_i+\beta_{k-i+1}\leq n+1$ for $i\in[k]$, then the equations of Proposition \ref{prop:SchubIntCompInt} define $\Delta(X_\alpha^\circ \Fdot^1\cap X_\beta^\circ \Fdot^2)$ in $X:=X_\alpha^\circ \Fdot^1 \times X_{\beta^{\perp}}^\circ \Fdot^{2{\perp}}$ as the solution set to a system of $\codim_X\Delta(X_\alpha^\circ \Fdot^1\cap X_\beta^\circ \Fdot^2)$ equations, and the projection to the primal factor
 \[\Delta(X_\alpha^\circ \Fdot^1\cap X_\beta^\circ \Fdot^2)\longrightarrow X_\alpha^\circ \Fdot^1\cap X_\beta^\circ \Fdot^2\]
 gives a bijection of sets.
\end{corollary}
We note that $\alpha_i+\beta_{k-i+1}\leq n+1$ for $i\in[k]$ is the condition that $X_\alpha\Fdot^1\cap X_\beta\Fdot^2\neq \emptyset$.
We may extend this method of obtaining a system of codimension-many equations to intersections of more than two Schubert varieties.
\begin{definition}
 Let $\Delta^m\colon\Gr(k,V)\rightarrow\Gr(k,V)\times\Gr(n - k,V^*)\times \cdots \times \Gr(n - k,V^*)$ be the map given by $H\mapsto (H,H_2,\dotsc,H_m)$ such that $H_i = H^\perp$ for $2\leq i\leq m$.
We call $\Delta^m$ the \emph{dual diagonal map} and observe that $\Delta = \Delta^2$.
\end{definition}
\begin{proposition}\label{prop:dualDiagonalm}
 Let $A_1,\dotsc,A_m\subset\Gr(k,V)$ be subsets.
 Then
 \[\Delta^m(A_1\cap \cdots \cap A_m)= (A_1\times{\perp}(A_2)\times \cdots \times{\perp}(A_m) )\cap\Delta^m(\Gr(k,V))\,.\]
\end{proposition}
The proof is omitted, since it is given by iterating the proof of Proposition \ref{prop:dualDiagonal1}.
This gives a straightforward generalization of \ref{prop:SchubIntCompInt}.
\begin{proposition}\label{prop:PrimalDualSchubProb}
 If $X_{\alpha^1} \Fdot^1,\dotsc,X_{\alpha^m} \Fdot^m$ are Schubert varieties in $\Gr(k,V)$, then the set $\Delta^m(X_{\alpha^1}^\circ \Fdot^1 \cap \cdots \cap X_{\alpha^m}^\circ \Fdot^m)$ is equal to
 \[(X_{\alpha^1}^\circ \Fdot^1 \times X_{\alpha^{2{\perp}}}^\circ \Fdot^{2{\perp}} \times \cdots \times X^\circ_{\alpha^{m{\perp}}} \Fdot^{m{\perp}})\cap\Delta^m(\Gr(k,V))\,,\]
 and is the solution set to a system of $k(n-k)(m-1)$ bilinear equations in the coordinates $(\s_{\alpha^1},\hats\tbox_{\alpha^{2{\perp}}},\dotsc,\hats\tbox_{\alpha^{m{\perp}}})$.
\end{proposition}
The $k(n-k)(m-1)$ equations come from pairing the primal factor with each of the $m-1$ dual factors in Equation (\ref{eqn:bilin2SV}).
Kleiman's theorem, Proposition \ref{prop:genTrans}, implies that if $\balpha$ is a Schubert problem, and $\Fdot^1,\dots,\Fdot^m$ are in general position, then
\[X_{\alpha^1}^\circ \Fdot^1 \cap \cdots \cap X_{\alpha^m}^\circ \Fdot^m = X_{\alpha^1} \Fdot^1 \cap \cdots \cap X_{\alpha^m} \Fdot^m\,.\]
More generally, if $\balpha$ is not a Schubert problem, but $\Fdot^1,\dots,\Fdot^m$ are in general position, we still have that
\[X_{\alpha^1}^\circ \Fdot^1 \cap \cdots \cap X_{\alpha^m}^\circ \Fdot^m \subset X_{\alpha^1} \Fdot^1 \cap \cdots \cap X_{\alpha^m} \Fdot^m\]
is dense.
We give the result of a straightforward dimension calculation.
\begin{theorem}\label{thm:SchubProbCompInt}
 Suppose $\Fdot^1,\dotsc,\Fdot^m$ are sufficiently general flags in $V$ and $\balpha = (\alpha^1,\dotsc,\alpha^m)$ is a list of Schubert conditions.
 The intersection
 \[X:=X_{\alpha^1}^\circ \Fdot^1\cap \cdots \cap X^\circ_{\alpha^m}\Fdot^m\]
 is the solution set to the bilinear equations of Proposition \ref{prop:PrimalDualSchubProb}.
 This involves formulating $X$ using $k(n-k)(m-1)$ equations in a space of dimension $k(n-k)m-|\balpha|$.
 
 In particular, if $\balpha$ is a Schubert problem, then $X=X_{\alpha^1} \Fdot^1\cap \cdots \cap X_{\alpha^m}\Fdot^m$, and $X$ is formulated as the set of solutions to a square system of equations.
\end{theorem}
Using this formulation, we may certify approximate solutions to Schubert problems and therefore may use numerical methods to study Schubert calculus from a pure mathematical point of view.
In some circumstances, this square formulation may lead to more efficient computation than the determinantal formulation.
We give an example comparing the classical system of equations with the primal-dual system of equations.
\begin{example}\label{ex:overdetVSprimaldual}
 Let $\balpha=(\alpha^1,\dotsc,\alpha^4)$ be the Schubert problem in $\Gr(4,\C^8)$ given by $\alpha^i=(2,5,7,8)$ for $i=1,\dots,4$, and let $\Fdot^1,\dotsc,\Fdot^4$ be flags in general position.
 We denote $\alpha^i$ by its Young diagram $\IIIit$.
 The classical formulation of the instance
 \[X:=X_{\IIIis} \Fdot^1 \cap \cdots \cap X_{\IIIis}\Fdot^4\]
 of $\balpha$ uses determinantal equations in the coordinates $\s_{\IIIis}$ of $X_{\IIIis} \Fdot^1$.
 By Corollary \ref{cor:numGens}, this formulation involves a system of $3\cdot 17 = 51$ linearly independent quartic determinants in $16-4=12$ variables.

 The competing primal-dual formulation is a square system of bilinear equations in the coordinates $\left(\s_{\IIIis} , \hats\tbox_{\IIiIs} , \hats\tbox_{\IIiIs} , \hats\tbox_{\IIiIs}\right)$ of $X_{\IIIis} \Fdot^1\times X_{\IIiIs} \Fdot^{2{\perp}}\times X_{\IIiIs} \Fdot^{3{\perp}}\times X_{\IIiIs} \Fdot^{4{\perp}}$.
 This system involves $48$ bilinear equations in $48$ variables.
\end{example}
A feature of the primal-dual formulation for an instance of a Schubert problem is that it requires more variables than the classical formulation, but it typically lowers the degrees of the polynomials which must be solved.
If we have flags in linear general position, then we may reduce the number of variables and equations.
\begin{example}\label{ex:overVSprimdualLinGen}
 Let $\balpha=(\alpha^1,\dotsc,\alpha^4)$ be the Schubert problem for $\Gr(4,\C^8)$ given by $\alpha^i=(2,5,7,8)$ for $i=1,\dots,4$, and let $\Fdot^1,\dotsc,\Fdot^4$ be flags in general position.
 We denote $\alpha^i$ by its Young diagram $\IIIit$.
 The classical formulation of the instance
 \[X:=X_{\IIIis} \Fdot^1 \cap \cdots \cap X_{\IIIis}\Fdot^4\]
 of $\balpha$ uses determinantal equations in the coordinates $\s_{\IIIis}^{\IIIis}$ of $(X_{\IIIis} \Fdot^1 \cap X_{\IIIis} \Fdot^2)^\circ$.
 By Corollary \ref{cor:numGens}, this formulation involves a system of $2\cdot 17 = 34$ linearly independent quartic determinants in $16-4-4=8$ variables.

 The competing primal-dual formulation is a square system of bilinear equations in the coordinates $\left(\s_{\IIIis}^{\IIIis} , \hats\tbox_{\IIiIs}^{\IIiIs}\right)$ of $(X_{\IIIis} \Fdot^1 \cap X_{\IIIis} \Fdot^2)^\circ \times (X_{\IIiIs} \Fdot^{3{\perp}} \cap X_{\IIiIs} \Fdot^{4{\perp}})^\circ$.
 This system involves $16$ bilinear equations in $16$ variables.
\end{example}
\begin{proposition}\label{prop:PrimalDualSchubProbHalf}
 Suppose $m\geq 2$ is even.
 If $\alpha^i\in\tbinom{[n]}{k}$ for $i\in [m]$ and $\Fdot^i$ for $i\in [m]$ are flags in linear general position, then the set $\Delta^m(X_{\alpha^1} \Fdot^1 \cap \cdots \cap X_{\alpha^m} \Fdot^m)$ is equal to
 \[\left(X_{\alpha^1} \Fdot^1\cap X_{\alpha^2} \Fdot^2 \times \prod_{i=2}^{m/2} X_{\alpha^{(2i-1){\perp}}} \Fdot^{(2i-1){\perp}}\cap X_{\alpha^{2i{\perp}}} \Fdot^{2i{\perp}} \right)\cap\Delta^{m/2}(\Gr(k,V))\]
 and is expressed locally as a system of $k(n-k)(m/2-1)$ bilinear equations in the coordinates $(\s_{\alpha^1}^{\alpha^2},\hats\tbox_{\alpha^{3{\perp}}}^{\alpha^{4{\perp}}},\dotsc,\hats\tbox_{\alpha^{(m-1){\perp}}}^{\alpha^{m{\perp}}})$.
\end{proposition}
With this proposition, we eliminate roughly half of the variables and equations needed to define a Schubert problem with a square system.
\begin{theorem}\label{thm:SchubProbCompIntHalfEqs}
 Suppose $m\geq 2$ is even, $\Fdot^1,\dotsc,\Fdot^m$ are sufficiently general flags in $\Gr(k,V)$, and $\balpha = (\alpha^1,\dotsc,\alpha^m)$ is a list of Schubert conditions.
 The intersection
 \[X:=X^\circ_{\alpha^1}\Fdot^1\cap \cdots \cap X^\circ_{\alpha^m}\Fdot^m\]
 is the solution set to the bilinear equations of Proposition \ref{prop:PrimalDualSchubProbHalf}.
 This involves formulating $X$ using $k(n-k)(m/2-1)$ equations by realizing it in a space of dimension $k(n-k)m/2-|\balpha|$.

In particular, if $\balpha$ is a Schubert problem, then $X = X_{\alpha^1}\Fdot^1\cap \cdots \cap X_{\alpha^m}\Fdot^m$, and $X$ is formulated as the set of solutions to a square system of equations.
\end{theorem}
We may use Theorem \ref{thm:SchubProbCompIntHalfEqs} in the case where $m$ is odd by appending a trivial Schubert condition $\alpha^{m+1}:=(n-k+1,\dotsc,n)$ to $\balpha$.

Since $X_{\Is}\Fdot$ is a hypersurface defined by one determinant, we may formulate a Schubert problem involving some hypersurface Schubert varieties using a square system involving fewer equations and variables than suggested by Theorem \ref{thm:SchubProbCompIntHalfEqs}.
We use the primal-dual formulation to express the intersection of non-hypersurface Schubert varieties and a determinant to define each hypersurface in the primal factor.
\begin{example}
 Consider the Schubert problem $(\IIIit^3,\It^4)$ in $\Gr(4,\C^8)$.
 Suppose $\Fdot^1,\dotsc,\Fdot^7$ are general flags.
 We may express the instance
 \[X:=X_{\IIIis} \Fdot^1 \cap \cdots \cap X_{\IIIis}\Fdot^3 \cap X_{\Is} \Fdot^4\cap \cdots \cap X_{\Is} \Fdot^7\]
 of $(\IIIit^3,\It^4)$ by a system of determinantal equations in the local coordinates $\s_{\IIIis}^{\IIIis}$ on $(X_{\IIIis} \Fdot^1\cap X_{\IIIis} \Fdot^2)^\circ$.
 By Corollary \ref{cor:numGens}, this formulation involves a system of $1\cdot 17 + 4\cdot 1= 21$ quartic determinants in $16-4-4=8$ variables.

 The na\"ive competing primal-dual formulation is a square system of $48$ bilinear equations in $48$ variables.

 Using a primal-dual formulation with $X_{\It}\Fdot^5 \cap X_{\It}\Fdot^6 \cap X_{\It}\Fdot^7$ defined by determinants in the primal factor yields a square system of equations in the coordinates $\left(\s_{\IIIis}^{\IIIis} , \hats\tbox_{\IIiIs}^{\Is}\right)$ of $X_{\IIIis} \Fdot^1 \cap X_{\IIIis} \Fdot^{2}\times X_{\IIiIs} \Fdot^{3{\perp}} \cap X_{\Is} \Fdot^{4{\perp}}$
 consisting of $16$ bilinear equations and $3$ quartic determinants in $19$ variables.
\end{example}

\section{Flag Varieties}

Many of the results of this chapter extend to Schubert problems more general than those in a Grassmannian.
As an example of this, we describe flag varieties, which are generalizations of Grassmannians.

Fix a positive integer $\ell$, and let $\bk:=(0 < k_1 < \cdots < k_\ell < n)$ be an increasing $\ell$-tuple of positive integers less than $n$.
\begin{definition}
 The \emph{flag variety $\Fl(\bk;V)$} is the set of $\ell$-tuples $H$ of nested $k_i$-planes,
 \[\Fl(\bk;V):=\{H\ |\ H_1\subset \cdots \subset H_\ell \subset V \,,\ \dim(H_i)=k_i\mbox{ for }i\in[\ell]\}\,.\]
\end{definition}
If $\ell=1$, then $\Fl(\bk;V)=\Gr(k_1,V)$.
We generalize the notion of a Schubert condition.
\begin{definition}
 Let $\alpha\in\tbinom{[n]}{\bk}$ denote the set of permutations on $[n]$ such that $\alpha_i < \alpha_{i+1}$ for $i\in[n]\setminus \bk$.
 We call $\alpha\in\tbinom{[n]}{\bk}$ a \emph{Schubert condition}.
\end{definition}

 We give a few Schubert conditions for the flag variety $\Fl(2,5;\C^7)$:
 \[(3,6\,|\,1,2,4\,|\,5,7)\,,\qquad (6,7\,|\,3,4,5\,|\,1,2)\,,\qquad(1,2\,|\,3,4,5\,|\,6,7)\,.\]
 We use a vertical line instead of a comma to denote positions where entries of $\alpha$ are allowed to decrease.

Flag varieties have local coordinates similar to the Stiefel coordinates.
\begin{definition}
 Let $\alpha\in \tbinom{[n]}{\bk}$ be a Schubert condition.
 The subset $S_\bk(\alpha)\subset \Mat_{k_\ell\times n}$ is the subset of matrices whose entries $m_{ij}$ satisfy the condition,
 \[m_{i,\alpha_j}=\delta_{ij}\quad\mbox{for}\quad k_{p-1}+1 \leq i \leq k_p\,,\ 1 \leq j \leq k_p\,,\]
 for $p\in[\ell]$ with the convention $k_0=0$.
 We call the coordinates given by these matrices the \emph{Stiefel coordinates}.
\end{definition}
\begin{example}
 Consider the flag variety $\Fl(2,4;6)$.
 Using $*$ to denote arbitrary entries, we give arbitrary matrices in $\s_\bk(\alpha)$ for $\alpha=(5,6\,|\,3,4)$ and $\alpha=(2,4\,|\,1,5)$ respectively,
 \[\left(
     \begin{matrix}
       * & * & * & * & 1 & 0 \\
       * & * & * & * & 0 & 1 \\
       * & * & 1 & 0 & 0 & 0 \\
       * & * & 0 & 1 & 0 & 0 \\
     \end{matrix}
   \right)\qquad\mbox{and}\qquad\left(
     \begin{matrix}
       * & 1 & * & 0 & * & * \\
       * & 0 & * & 1 & * & * \\
       1 & 0 & * & 0 & 0 & * \\
       0 & 0 & * & 0 & 1 & * \\
     \end{matrix}
   \right)\,. 
 \]
The $1$ in position $(i,\alpha_i)$ for $i\in[k_\ell]$ is called a \emph{pivot}.
\end{example}
The proof that $\Gr(k,V)$ is smooth extends to flag varieties.
\begin{proposition}\label{prop:dimFlKV}
 The flag variety $\Fl(\bk;V)$ is a smooth variety of dimension
 \[\dim(\Fl(\bk;V)) = \sum_{i=1}^\ell (k_i-k_{i-1})(n-k_i)\,,\]
 with the convention $k_0=0$.
\end{proposition}
\begin{definition}
 Given a flag $\Fdot$ and $\alpha\in\tbinom{[n]}{\bk}$, we have a \emph{Schubert variety},
\begin{multline*}
 X_\alpha \Fdot := \{ H \in \Fl(\bk;V)\ |\ \dim(H_p \cap F_{\alpha_i})\geq \#\{\alpha_j\ |\ j\leq i\,,\ \alpha_j\leq \alpha_i\} \\
 \mbox{for }p\in[\ell]\,,\ k_{p-1}+1\leq i\leq k_p \}\,,
\end{multline*}
 with the convention $k_0=0$.
\end{definition}
Schubert varieties in flag varieties have local coordinates similar to the Stiefel coordinates.
\begin{definition}
 Let $\alpha\in \tbinom{[n]}{\bk}$.
 The \emph{Stiefel coordinates} of $X_\alpha \Fdot$ are given by the subset of matrices $(\s_\bk)_\alpha\subset\s_\bk(\alpha)$ which satisfy the requirement that every entry to the right of a pivot is zero.
\end{definition}
\begin{example}
 Consider the flag variety $\Fl(2,4;6)$.
 Using $*$ to denote arbitrary entries, we give arbitrary matrices in $(\s_\bk)_\alpha$, which give coordinates for $X_\alpha \Fdot$, for $\alpha=(5,6\,|\,3,4)$ and $\alpha=(2,4\,|\,1,5)$ respectively,
 \[\left(
     \begin{matrix}
       * & * & * & * & 1 & 0 \\
       * & * & * & * & 0 & 1 \\
       * & * & 1 & 0 & 0 & 0 \\
       * & * & 0 & 1 & 0 & 0 \\
     \end{matrix}
   \right)\qquad\mbox{and}\qquad\left(
     \begin{matrix}
       * & 1 & 0 & 0 & 0 & 0 \\
       * & 0 & * & 1 & 0 & 0 \\
       1 & 0 & 0 & 0 & 0 & 0 \\
       0 & 0 & * & 0 & 1 & 0 \\
     \end{matrix}
   \right)\,.
 \]
\end{example}
As in the Grassmannian case, one may count indeterminates to determine the dimension (or codimension) of a Schubert variety in a flag variety.
We write $|\alpha|$ to denote the codimension of $X_\alpha \Fdot$ in $\Fl(\bk;V)$.
The Stiefel coordinates for $X_\alpha \Fdot^1 \cap X_\beta \Fdot^2$ do not have a straightforward generalization for general flag varieties.

We extend properties of duality to flag varieties.
\begin{definition}
 Let $\bk^\perp$ denote the increasing $\ell$-tuple of integers defined as follows,
 \[\bk^\perp:=(0 < n-k_\ell < \cdots < n-k_1 <n)\]
\end{definition}
Recall that the duality between $V$ and $V^*$ gives a natural map ${\perp}:\Gr(k,V) \rightarrow \Gr(n-k,V^*)$ defined by $H\mapsto H^\perp$.
We may extend this map to a map from a flag variety to an associated flag variety,
\[{\perp}:\Fl(\bk,V) \rightarrow \Fl(\bk^\perp,V^*)\,,\]
given by $(H_1,\dotsc,H_\ell)\mapsto (H_\ell^\perp,\dotsc,H_1^\perp)$.

Let $\alpha\in\tbinom{[n]}{\bk^\perp}$.
The \emph{dual Stiefel coordinates} are the coordinates given by matrices in $\hats_\bk(\alpha)\subset \Mat_{n\times k^\perp_\ell} = \Mat_{n\times (n-k_1)}$ with entries $m_{ij}$ satisfying
 \[m_{n-\alpha_i+1,j}=\delta_{ij}\quad\mbox{for}\quad k^\perp_{p-1}+1 \leq j \leq k^\perp_p\,,\ 1 \leq i \leq k^\perp_p\,,\]
for $p\in[\ell]$ with the convention $k_0=0$.

The coordinates given by $\hats\tbox_\bk(\alpha)$ parametrize $\Fl(\bk^\perp;V^*)$, as the first $k^\perp_p$ columns parametrize a $k^\perp_p$-plane $H_p\in\Gr(k^\perp_p,V^*)$ for each $p\in[\ell]$, and we have $H_1\subset \cdots \subset H_\ell$.
This gives coordinates for a dense subset of $\Fl(\bk^\perp;V^*)$, because the first $k^\perp_p$ columns of the matrices in $\hats_\bk(\alpha)$ give coordinates for a dense subset of $\Gr(k^\perp_p,V^*)$ for each $p\in[\ell]$.
\begin{example}
 Let $n=7$.
 Consider the flag variety $\Fl(3,4,6;V^*)$ associated to $\Fl(1,3,4;V)$.
 We give arbitrary matrices in $(\hats_{(3,4,6)})_\alpha$ that provide local coordinates for the flag variety, for $\alpha=(5,6,7\,|\,4\,|\,2,3)$ and $\alpha=(1,3,5\,|\,4\,|\,2,7)$ respectively,
 \[\left(
     \begin{matrix}
       0 & 0 & 1 & 0 & 0 & 0 \\
       0 & 1 & 0 & 0 & 0 & 0 \\
       1 & 0 & 0 & 0 & 0 & 0 \\
       * & * & * & 1 & 0 & 0 \\
       * & * & * & * & 0 & 1 \\
       * & * & * & * & 1 & 0 \\
       * & * & * & * & * & * \\
     \end{matrix}
   \right)\qquad \mbox{and} \qquad \left(
     \begin{matrix}
       * & * & * & * & 0 & 1 \\
       * & * & * & * & * & * \\
       0 & 0 & 1 & 0 & 0 & 0 \\
       * & * & * & 1 & 0 & 0 \\
       0 & 1 & 0 & 0 & 0 & 0 \\
       * & * & * & * & 1 & 0 \\
       1 & 0 & 0 & 0 & 0 & 0 \\
     \end{matrix}
   \right)\,.
 \]
\end{example}
For $j\in [k_\ell]$, the 1 in the $(n-\alpha_j+1,j)$ position is called a pivot.
\begin{definition}
 The \emph{dual Stiefel coordinates} for the Schubert variety $X_\alpha \Fdot \subset \Fl(\bk^\perp;V^*)$ are the local coordinates given by the subset $(\hats\tbox_{\bk^\perp})_\alpha\subset \hats\tbox_{\bk^\perp}(\alpha)$ consisting of matrices whose entries above each pivot are zero.
\end{definition}

\begin{definition}\label{def:flagAlphaPerp}
 Let $\alpha\in\tbinom{[n]}{\bk}$ be a Schubert condition.
 We define $\omega = (n,n-1,\dotsc,2,1)$ to be the longest permutation on $[n]$.
 The Schubert condition $\alpha^\perp\in \tbinom{[n]}{\bk^\perp}$ associated to $\alpha$ is given by the composition of permutations,
 \[\alpha^\perp:=\omega\alpha\omega\,.\]
\end{definition}
Definition \ref{def:flagAlphaPerp} allows us to extend Proposition \ref{prop:dualSV} to Schubert varieties in a general flag variety.
\begin{proposition}
 If $\alpha\in\tbinom{[n]}{\bk}$ then $X_\alpha \Fdot\cong {\perp}(X_\alpha \Fdot) = X_{\alpha^\perp} \Fdot^\perp$.
\end{proposition}
\begin{example}
 Let $n=7$ and $\alpha = (4\,|\,2,5\,|\,1,6\,|\,3,7) \in\tbinom{[7]}{1,3,5}$.
 We have an associated Schubert condition $\alpha^\perp = (1,5\,|\,2,7\,|\,3,6\,|\,4,) \in \tbinom{[7]}{2,4,6}$.
 We give Stiefel coordinates $(\s_\bk)_\alpha$ and $(\hats\tbox_{\bk^\perp})_{\alpha^\perp}$ for $X_\alpha \Fdot$ and $X_{\alpha^\perp}\Fdot^\perp$ respectively,
 \[\left(
     \begin{matrix}
       a & b & c & 1 & 0 & 0 & 0 \\
       d & 1 & 0 & 0 & 0 & 0 & 0 \\
       e & 0 & f & 0 & 1 & 0 & 0 \\
       1 & 0 & 0 & 0 & 0 & 0 & 0 \\
       0 & 0 & g & 0 & 0 & 1 & 0 \\
     \end{matrix}
   \right)\qquad\mbox{and}\qquad
   \left(
     \begin{matrix}
       0 & 0  & 0 & 1  & 0 & 0  \\ 
       0 & 0  & 0 & -d & 0 & 1  \\
       0 & 1  & 0 & 0  & 0 & 0  \\
       0 & -c & 0 & -a & 1 & -b \\
       0 & -f & 0 & -e & 0 & 0  \\ 
       0 & -g & 1 & 0  & 0 & 0  \\ 
       1 & 0  & 0 & 0  & 0 & 0  \\ 
     \end{matrix}
   \right)\,.
 \]
 These parametrizations pair a point in $X_\alpha \Fdot$ with its dual in $X_{\alpha^\perp} \Fdot^\perp$.
\end{example}
Kleiman's theorem of general transitivity applies to intersections in a flag variety \cite{Kleiman}.
\begin{proposition}\label{prop:flagGenTrans}
 Let $\balpha = (\alpha^1,\dotsc,\alpha^m)$ be a list of Schubert conditions for $\Fl(\bk;V)$.
 If $\Fdot^1,\dotsc,\Fdot^m$ are general flags, then
 \begin{equation}\label{eqn:FlagGeneralIntersection}
  X:=X_{\alpha^1} \Fdot^1 \cap \cdots \cap X_{\alpha^m} \Fdot^m
 \end{equation}
 is generically transverse.
 That is, $X=\emptyset$ or $\codim(X) = |\alpha^1|+\cdots+|\alpha^m| =: |\balpha|$.
\end{proposition}
We say that $\balpha$ is a Schubert problem in $\Fl(\bk;V)$ if $X$ of Equation (\ref{eqn:FlagGeneralIntersection}) has expected dimension zero, that is, if $|\balpha| = \dim(\Fl(\bk;V))$.

The proposition and theorems of Section \ref{sec:primal-dual} which do not use the coordinates $\s_\alpha^\beta$ extend to flag varieties.
Thus we may formulate Schubert problems in flag varieties as solution sets to square systems of bilinear equations.
\begin{example}
 Consider the flag variety $\Fl(2,4;\mathbb C^6)$ which is $12$-dimensional, general flags $\Fdot^1,\dotsc,\Fdot^4$ in $V$, and the Schubert condition $\alpha = (3,6\,|\,2,5\,|\,1,4)$. 
 We have $|\alpha| = 3$, and
 \[X:=X_\alpha \Fdot^1 \cap \cdots \cap X_\alpha \Fdot^4\]
 contains $12$ points.
 The relevant conditions characterizing $(H_1,H_2)\in X_\alpha \Fdot^i$ are
 \[\dim(H_1\cap F_3)\geq 1\,,\qquad\mbox{and}\qquad\dim(H_2\cap F_2)\geq 1\,.\]
 Since $\dim(H_1)=2$, the first relevant condition is given by $3$ linearly independent quadratic determinants.
 Since $\dim(H_2)=4$, the second relevant condition is given by one maximal quartic determinant.
 Using local coordinates for $X_\alpha\Fdot^1$, the classical determinantal formulation of $X$ involves $3\cdot 3 = 9$ quadratic and $3 \cdot 1 = 3$  quartic equations in $12-3=9$ variables.

 The alternative primal-dual formulation involves a square system of $36$ bilinear equations in local coordinates $((\s_{(2,4)})_\alpha,(\s_{(2,4)})_{\alpha^\perp},(\s_{(2,4)})_{\alpha^\perp},(\s_{(2,4)})_{\alpha^\perp})$.
 Note that $\bk:=(2,4)$ implies $\bk^\perp=(2,4)$.
\end{example}

  \chapter{SUMMARY} 
  \label{chapter: Summary}

In Chapter \ref{chapSchubert}, we gave background needed to understand our study in enumerative real algebraic geometry.
We outlined the history surrounding some theorems and conjectures in Schubert calculus.
The Mukhin-Tarasov-Varchenko Theorem is a surprisingly elegant result in enumerative real algebraic geometry, which demonstrates that the enumerative theory of real Schubert calculus is a rich field of study.
This remarkable theorem was not generally accepted when it was first conjectured, and computations played a large role in giving credence to it.

In recent years, supercomputers have been used to solve billions of polynomial systems in order to investigate problems related to the Mukhin-Tarasov-Varchenko Theorem.
These investigations have lead to theorems and strongly supported conjectures.
We continued this practice of studying reality problems with the use of supercomputers.

In Chapter \ref{chapLower}, we described a study of Eremenko and Gabrielov, which used topological methods to obtain lower bounds to the number of real points in a fiber of the Wronski map over a real point.
We realized this inverse Wronski problem as a problem in Schubert calculus and used modern software tools to investigate these bounds from a more general point of view.
We discovered that the Eremenko-Gabrielov type lower bounds are often sharp.
In some cases, however, sharpness fails in an interesting way.

We solved over 339 million instances of 756 Schubert problems, using over 469 gigahertz-years of processing power.
While studying the data, we observed a remarkable congruence modulo four in the number of real solutions to problems with certain symmetries, and this congruence was the topic of Chapter \ref{chapMod4}.
We also discovered a family of Schubert problems, which has unusual gaps in the numbers of real solutions to real osculating instances.
These relate to work of Sottile and Soprunova, and we used their method of counting real factorizations of a real polynomial to explain the observed lower bounds and gaps.

In Chapter \ref{chapMod4}, we proved a congruence modulo four in the number of real solutions to real osculating instances of Schubert problems given by symmetric Schubert conditions.
This work affirmed the most surprising and compelling conjecture to come out of the computational project described in Chapter \ref{chapLower}.
One would typically expect the number of real solutions to a real osculating instance of a Schubert problem to be fixed modulo two, because nonreal solutions come in pairs.
We discovered that there is a Lagrangian involution which also acts on symmetric Schubert problems.
For a rich family of such problems, the Lagrangian involution and complex conjugation are independent and the nonreal solutions come in sets of four.
Establishing a congruence modulo four on the number of real solutions to real osculating instances, we established a new invariant in enumerative real algebraic geometry.

The work in Chapter \ref{chapLower}, and a lot of other work done in Schubert calculus, relied heavily on formulating problems in a way that is efficient for computation.
The computational complexity of calculating a Gr\"obner basis in characteristic zero was a bottleneck, which we hope to overcome through the use of certifiable numerical methods.
Algorithms from Smale's $\alpha$-theory may be used to certify numerical output, when the problem involved is given by a square system of polynomial equations.
However, Schubert problems are famously overdetermined.

In Chapter \ref{chapSquare}, we recast instances of Schubert problems as solution sets to square systems.
While this has the practical application of allowing us to use numerical methods in a pure mathematical study of Schubert calculus, our ability to reformulate such a problem as a square system is interesting by its own right.
The duality between $V$ and $V^*$ induces a duality between Schubert varieties in a Grassmannian and a dual Grassmannian, and we use this to give a primal-dual formulation for an instance of a Schubert problem.
This requires that we work in a larger space, adding variables, but we benefit from replacing higher-degree determinantal equations by bilinear equations.
The square system of equations may be used to certify approximate solutions obtained via an overdetermined system of determinantal equations, but if the bilinear equations provide a more efficient setting for solving instances via numerical methods, then we may do away with the overdetermined system altogether.



%
 \bibliographystyle{amsplain}
 \bibliography{References}

\providecommand{\bysame}{\leavevmode\hbox to3em{\hrulefill}\thinspace}
\providecommand{\MR}{\relax\ifhmode\unskip\space\fi MR }
\providecommand{\MRhref}[2]{%
  \href{http://www.ams.org/mathscinet-getitem?mr=#1}{#2}
}
\providecommand{\href}[2]{#2}
\begin{thebibliography}{10}

\bibitem{BPR}
S.~Basu, R.~Pollack, and M.-F. Roy, \emph{Algorithms in {R}eal {A}lgebraic
  {G}eometry}, Algorithms and Computation in Mathematics, vol.~10,
  Springer-Verlag, Berlin, 2003.

\bibitem{shape}
E.~Becker, M.~G. Marinari, T.~Mora, and C.~Traverso, \emph{The {S}hape of the
  {S}hape {L}emma}, In proc. of the international symposium on symbolic and
  algebraic computation, ACM Press, 1993, pp.~129--133.

\bibitem{BCSS}
L.~Blum, F.~Cucker, M.~Shub, and S.~Smale, \emph{Complexity and {R}eal
  {C}omputation}, Springer-Verlag, New York, 1998.

\bibitem{Castelnuovo1889}
G.~Castelnuovo, \emph{Numero delle {I}nvoluzioni {R}azionali {G}aicenti {S}opra
  una {C}urva di {D}ato {G}enere}, Rendi. R. Accad. Lineci \textbf{4} (1889),
  130--133.

\bibitem{CLO2007}
D.~A. Cox, J.~Little, and D.~O'Shea, \emph{Ideals, {V}arieties, and
  {A}lgorithms: {A}n {I}ntroduction to {C}omputational {A}lgebraic {G}eometry
  and {C}ommutative {A}lgebra}, Springer-Verlag New York, Inc., Secaucus, NJ,
  USA, 2007.

\bibitem{Singular}
W.~Decker, G.-M. Greuel, G.~Pfister, and H.~Sch{\"o}nemann, \emph{Singular
  3-1-6 -- {a} {C}omputer {A}lgebra {S}ystem for {P}olynomial {C}omputations},
  (2012), http://www.singular.uni-kl.de.

\bibitem{DS00}
J.~P. Dedieu and M.~Shub, \emph{Newton's {M}ethod for {O}verdetermined
  {S}ystems of {E}quations}, Math. Comp. \textbf{69} (2000), no.~231,
  1099--1115.

\bibitem{EH1983}
D.~Eisenbud and J.~Harris, \emph{Divisors on {G}eneral {C}urves and {C}uspidal
  {R}ational {C}urves}, Invent. Math. \textbf{74} (1983), no.~3, 371--418.

\bibitem{EG2002a}
A.~Eremenko and A.~Gabrielov, \emph{Degrees of {R}eal {W}ronski {M}aps},
  Discrete Comput. Geom. \textbf{28} (2002), no.~3, 331--347.

\bibitem{EG2002b}
\bysame, \emph{Rational {F}unctions with {R}eal {C}ritical {P}oints and the
  {B}. and {M}. {S}hapiro {C}onjecture in {R}eal {E}numerative {G}eometry},
  Ann. of Math. (2) \textbf{155} (2002), no.~1, 105--129.

\bibitem{EGSV}
A.~Eremenko, A.~Gabrielov, M.~Shapiro, and A.~Vainshtein, \emph{Rational
  {F}unctions and {R}eal {S}chubert {C}alculus}, Proc. Amer. Math. Soc.
  \textbf{134} (2006), no.~4, 949--957 (electronic).

\bibitem{secant}
L.~Garc{\'\i}a-Puente, N.~Hein, C.~Hillar, A.~Mart{\'\i}n del Campo, J.~Ruffo,
  F.~Sottile, and Z.~Teitler, \emph{The {S}ecant {C}onjecture in the {R}eal
  {S}chubert {C}alculus}, Exp. Math. \textbf{21} (2012), no.~3, 252--265.

\bibitem{HL2011}
A.~Hashemi and D.~Lazard, \emph{Sharper {C}omplexity {B}ounds for
  {Z}ero-{D}imensional {G}r\"obner {B}ases and {P}olynomial {S}ystem
  {S}olving}, Internat. J. Algebra Comput. \textbf{21} (2011), no.~5, 703--713.

\bibitem{monotone}
J.~D. Hauenstein, N.~Hein, C.~Hillar, A.~Mart{\'\i}n del Campo, F.~Sottile, and
  Z.~Teitler, \emph{The {M}onotone {S}ecant {C}onjecture in the {R}eal
  {S}chubert {C}alculus}, 2013, in preparation.

\bibitem{square}
J.~D. Hauenstein, N.~Hein, and F.~Sottile, \emph{Certifiable {N}umerical
  {C}omputations in {S}chubert {C}alculus}, 2013, in preparation.

\bibitem{HSW10}
J.~D. Hauenstein, A.~J. Sommese, and C.~W. Wampler, \emph{Regeneration
  {H}omotopies for {S}olving {S}ystems of {P}olynomials}, Math. Comp.
  \textbf{80} (2011), no.~273, 345--377.

\bibitem{alphaCertified}
J.~D. Hauenstein and F.~Sottile, \emph{{A}lgorithm 921: alpha{C}ertified:
  {C}ertifying {S}olutions to {P}olynomial {S}ystems}, ACM Trans. Math.
  Software \textbf{38} (2012), no.~4, Article No. 28.

\bibitem{lower}
N.~Hein, C.~Hillar, and F.~Sottile, \emph{Lower {B}ounds in {R}eal {S}chubert
  {C}alculus}, 2013, in preparation.

\bibitem{mod4}
N.~Hein, F.~Sottile, and I.~Zelenko, \emph{A {C}ongruence {M}odulo {F}our in
  the {R}eal {S}chubert {C}alculus}, 2012, arXiV.org/1211.7160.

\bibitem{framework}
C.~Hillar, L.~Garc{\'{\i}}a-Puente, A.~Mart{\'{\i}}n del Campo, J.~Ruffo,
  Z.~Teitler, S.~L. Johnson, and F.~Sottile, \emph{Experimentation at the
  {F}rontiers of {R}eality in {S}chubert {C}alculus}, Gems in experimental
  mathematics, Contemp. Math., vol. 517, Amer. Math. Soc., Providence, RI,
  2010, pp.~365--380.

\bibitem{HSS98}
B.~Huber, F.~Sottile, and B.~Sturmfels, \emph{Numerical {S}chubert {C}alculus,
  {S}ymbolic {N}umeric {A}lgebra for {P}olynomials}, J. Symbolic Comput.
  \textbf{26} (1998), no.~6, 767--788.

\bibitem{Kleiman}
S.~L. Kleiman, \emph{The {T}ransversality of a {G}eneral {T}ranslate},
  Compositio Math. \textbf{28} (1974), 287--297.

\bibitem{KL1972}
S.~L. Kleiman and D.~Laksov, \emph{Schubert {C}alculus}, Amer. Math. Monthly
  \textbf{79} (1972), 1061--1082.

\bibitem{Leykin04}
A.~Leykin, \emph{On {P}arallel {C}omputation of {G}r{\"o}bner {B}ases}, ICPP
  Workshops, 2004, pp.~160--164.

\bibitem{LMSVV}
A.~Leykin, A.~Mart{\'{\i}}n del Campo, F.~Sottile, R.~Vakil, and J.~Verschelde,
  \emph{Implementation of the {L}ittlewood-{R}ichardson {H}omotopy}, In
  progress.

\bibitem{LS}
A.~Leykin and F.~Sottile, \emph{Galois {G}roups of {S}chubert {P}roblems via
  {H}omotopy {C}omputation}, Math. Comp. \textbf{78} (2009), no.~267,
  1749--1765.

\bibitem{MM1982}
E.~W. Mayr and A.~R. Meyer, \emph{The {C}omplexity of the {W}ord {P}roblems for
  {C}ommutative {S}emigroups and {P}olynomial {I}deals}, Adv. in Math.
  \textbf{46} (1982), no.~3, 305--329.

\bibitem{MTV2009a}
E.~Mukhin, V.~Tarasov, and A.~Varchenko, \emph{The {B}. and {M}. {S}hapiro
  {C}onjecture in {R}eal {A}lgebraic {G}eometry and the {B}ethe {A}nsatz}, Ann.
  of Math. (2) \textbf{170} (2009), no.~2, 863--881.

\bibitem{MTV2009b}
\bysame, \emph{Schubert {C}alculus and {R}epresentations of the {G}eneral
  {L}inear {G}roup}, J. Amer. Math. Soc. \textbf{22} (2009), no.~4, 909--940.

\bibitem{Purbhoo2010}
K.~Purbhoo, \emph{Jeu de {T}aquin and a {M}onodromy {P}roblem for {W}ronskians
  of {P}olynomials}, Adv. Math. \textbf{224} (2010), no.~3, 827--862.

\bibitem{RS1998}
J.~Rosenthal and F.~Sottile, \emph{Some {R}emarks on {R}eal and {C}omplex
  {O}utput {F}eedback}, Systems Control Lett. \textbf{33} (1998), no.~2,
  73--80.

\bibitem{Rou99}
F.~Rouillier, \emph{Solving {Z}ero-{D}imensional {S}ystems through the
  {R}ational {U}nivariate {R}epresentation}, Appl. Algebra Engrg. Comm. Comput.
  \textbf{9} (1999), no.~5, 433--461.

\bibitem{RSSS2006}
J.~Ruffo, Y.~Sivan, E.~Soprunova, and F.~Sottile, \emph{Experimentation and
  {C}onjectures in the {R}eal {S}chubert {C}alculus for {F}lag {M}anifolds},
  Experiment. Math. \textbf{15} (2006), no.~2, 199--221.

\bibitem{Schubert1886}
H.~Schubert, \emph{Anzahl-{B}estimmungen f\"ur {L}ineare {R}\"aume {B}eliebiger
  {D}imension}, Acta Math. \textbf{8} (1886), no.~1, 97--118.

\bibitem{S86}
S.~Smale, \emph{Newton's {M}ethod {E}stimates from {D}ata at {O}ne {P}oint},
  The {M}erging of {D}isciplines: {N}ew {D}irections in {P}ure, {A}pplied, and
  {C}omputational {M}athematics ({L}aramie, {W}yo., 1985), Springer, New York,
  1986, pp.~185--196.

\bibitem{SW05}
A.~J. Sommese and C.~W. Wampler, \emph{The {N}umerical {S}olution of {S}ystems
  of {P}olynomials {A}rising in {E}ngineering and {S}cience}, World Scientific
  Publishing Co. Pte. Ltd., Hackensack, NJ, 2005.

\bibitem{SS2006}
E.~Soprunova and F.~Sottile, \emph{Lower {B}ounds for {R}eal {S}olutions to
  {S}parse {P}olynomial {S}ystems}, Adv. Math. \textbf{204} (2006), no.~1,
  116--151.

\bibitem{Sot99}
F.~Sottile, \emph{The {S}pecial {S}chubert {C}alculus is {R}eal}, Electron.
  Res. Announc. Amer. Math. Soc. \textbf{5} (1999), 35--39 (electronic).

\bibitem{Sottile2000}
\bysame, \emph{Real {S}chubert {C}alculus: {P}olynomial {S}ystems and a
  {C}onjecture of {S}hapiro and {S}hapiro}, Experiment. Math. \textbf{9}
  (2000), no.~2, 161--182.

\bibitem{Sottile2011}
\bysame, \emph{Real {S}olutions to {E}quations from {G}eometry}, University
  Lecture Series, vol.~57, American Mathematical Society, Providence, RI, 2011.

\bibitem{svv}
F.~Sottile, R.~Vakil, and J.~Verschelde, \emph{Solving {S}chubert {P}roblems
  with {L}ittlewood-{R}ichardson {H}omotopies}, {ISSAC} 2010 -- {P}roceedings
  of the 2010 {I}nternational {S}ymposium on {S}ymbolic and {A}lgebraic
  {C}omputation, ACM, New York, 2010, pp.~179--186.

\bibitem{Va06a}
R.~Vakil, \emph{A {G}eometric {L}ittlewood-{R}ichardson {R}ule}, Ann. of Math.
  (2) \textbf{164} (2006), no.~2, 371--421.

\bibitem{Va06b}
\bysame, \emph{Schubert {I}nduction}, Ann. of Math. (2) \textbf{164} (2006),
  no.~2, 489--512.

\end{thebibliography}


\end{document}